\documentclass[12pt]{amsart}

\let\oldtocsection=\tocsection

\renewcommand{\tocsection}[2]{\hspace{0em}\scriptsize \bfseries \oldtocsection{#1}{#2}}

\usepackage{amsfonts}
\usepackage{amsmath}
\usepackage{amssymb}
\usepackage[margin=1.2in]{geometry}
\usepackage{xcolor}
\usepackage{tikz}
\usepackage{mathtools}							
\usepackage{commath}
\usepackage{longtable}

\newcommand{\inv}{^{-1}}

\newtheorem{theorem}{Theorem}[section]
\newtheorem{corollary}[theorem]{Corollary}
\newtheorem{lemma}[theorem]{Lemma}

\theoremstyle{definition}

\newtheorem{remark}[theorem]{Remark}

\newtheorem*{openproblem}{Open problem}

\renewcommand{\Re}{\operatorname{Re}}
\renewcommand{\Im}{\operatorname{Im}}
\newcommand{\N}{\mathbb{N}}
\newcommand{\R}{\mathbb{R}}
\newcommand{\C}{\mathbb{C}}

\numberwithin{equation}{section}

\makeatletter
\@namedef{subjclassname@2020}{
  \textup{2020} Mathematics Subject Classification}
\makeatother

\begin{document}
\title[A general quantified Ingham-Karamata Tauberian theorem]{A general quantified Ingham-Karamata Tauberian theorem}

\author[G. Debruyne]{Gregory Debruyne}
\thanks{G. Debruyne gratefully acknowledges support by a senior postdoctoral fellowship of Research--Foundation Flanders through the grant number 1249624N}
\address{G. Debruyne\\ Department of Mathematics: Analysis, Logic and Discrete Mathematics \\ Ghent University\\ Krijgslaan 281 \\ B 9000 Gent\\ Belgium}
\email{gregory.debruyne@UGent.be}

\begin{abstract} We provide a general quantified Ingham-Karamata Tauberian theorem with a flexible one-sided Tauberian condition under several types of boundary behavior for the Laplace transform. Our results in particular improve a theorem by Stahn, removing a vexing restriction on the growth of the Laplace transform. Improving existing optimality results, we also show that the obtained quantified rate is optimal in almost all cases.  
\end{abstract}

\maketitle

\tableofcontents

\section{Introduction} 

Tauberian theory occupies a central role in modern mathematical analysis, retrieving asymptotic information of an object from information about its integral transforms. It has widespread utility in versatile domains such as number theory \cite{Tenenbaumbook}, complex analysis \cite{titchmarsh}, harmonic analysis \cite{rudin}, spectral theory \cite{shubin2001}, probability theory \cite{binghambook}, operator theory \cite{batty-seifert} and differential equations \cite{a-b-h-n, b-d}, among many others \cite{korevaarbook}.  

The Ingham-Karamata theorem \cite{ingham1935,karamata1934} is a jewel of Tauberian theory and has come to be regarded as one of the landmarks of 20th century analysis \cite{c-q, lax-zalcman}. One of its many versions reads as follows. 

\begin{theorem}[Ingham, Karamata] \label{thoriginal} Let $S: [0,\infty) \rightarrow \mathbb{C}$ be a Lipschitz continuous function. If the Laplace transform $\mathcal{L}\{S;s\} := \int^\infty_{0} e^{-su} S(u) \dif u$, initially convergent on the half-plane $\Re \: s > 0$, admits a continuous extension to the line $\Re \: s = 0$, then
\begin{equation}
  S(x) = o(1), \ \ \ x \rightarrow \infty.
\end{equation} 
\end{theorem}

Next to the assumption on the Laplace transform, the theorem has a Lipschitz regularity hypothesis on the function $S$, without which the theorem becomes false. This feature is paramount in Tauberian theory and such a regularity hypothesis on the function is called a \emph{Tauberian condition}. Karamata actually showed Theorem \ref{thoriginal} under a more general \emph{one-sided Tauberian condition}, but with a much more complicated proof. In fact, up until now one-sided Tauberian conditions have been notoriously difficult to handle, see e.g. \cite[Remark 3.2]{Debruyne-VindasComplexTauberians}. An important objective of this paper is to provide a new method that allows for an easier treatment of one-sided Tauberian conditions. 

Since Newman in 1980 popularized the Ingham-Karamata theorem with his short proof of the prime number theorem \cite{newman}, see also the presentations by Korevaar \cite{korevaar1982} and Zagier \cite{zagier1997}, this Tauberian theorem has attracted renewed interest which has led to the discovery of many new applications. The $C_0$-semigroup community in particular has in the last decades actively stimulated research in Tauberian theory after the realization the Ingham-Karamata theorem may be invoked to study the abstract Cauchy problem \cite{a-b, a-b-h-n}, but there are also recent applications in other fields \cite{d-v-l1}.

Fairly recently \cite{b-c-t, b-d}, the $C_0$-semigroup community shifted its attention to \emph{quantified versions} of the Ingham-Karamata theorem as it has consequences for decay rates of classical solutions of the abstract Cauchy problem, a textbook example being the study of \emph{damped wave equations} \cite{b-b-t}. The best result currently available is due to Stahn \cite{stahn}, which we state for scalar-valued functions. We use the Vinogradov notation $f(x) \ll g(x)$ indicating the existence of a constant $C$ such that $|f(x)| \leq Cg(x)$ in the domains of $f$ and $g$. If the constant $C$ depends on an external parameter, it is incorporated in the notation. For convenience of writing, we shall say throughout this work that a function $g$ is analytic on a closed domain if it is analytic on its interior and admits a continuous extension to the boundary of that domain.

\begin{theorem}[Stahn] \label{thqikstahn} Let $S: [0,\infty) \rightarrow \mathbb{C}$ be a Lipschitz continuous function and $M,K: \R_{+}  \rightarrow (0,\infty)$ two continuous non-decreasing functions for which there exists $\varepsilon \in (0,1)$ such that
\begin{equation} \label{eqqikcondstahn}
 K(t) \ll \exp\big(\exp\big((tM(t))^{1-\varepsilon}\big)\big),\quad t\to\infty.
\end{equation}

 If $\mathcal{L}\{S;s\}$ admits an analytic extension to the region
\begin{equation} \label{eq:defomega}
 \Omega_{M} := \{s : |\Re s| \leq 1/M(|\Im s|) \},
\end{equation} 
where $|\mathcal{L}\{S;s\}| \ll K(|s|)/|s|$ as $|s| \rightarrow \infty$, then
\begin{equation} \label{qikeqmklogstahn}
 S(x) \ll_{S,M,K} \frac{1}{M_{K,\log}^{-1}(x)}, \quad x \rightarrow \infty,
\end{equation}
where $M_{K,\log}^{-1}$ is the inverse function of $M_{K,\log}(t) = M(t) (\log t + \log \log t + \log K(t))$.

Furthermore, if $K$ is of positive increase, that is, there exist $a,t_0 > 0$ such that $t^{-a}K(t) \ll R^{-a}K(R)$ for all $t_0 \leq t \leq R$ as $R \rightarrow \infty$, then
\begin{equation} \label{eqqikmkstahn}
 S(x) \ll_{S,M,K} \frac{1}{M_{K}^{-1}(x)}, \quad x \rightarrow \infty,
\end{equation}
with $M_K^{-1}$ the inverse function of $M_{K}(t) = M(t) (\log t + \log K(t))$.
\end{theorem}

We refer to appendix \ref{rfrappendix} what this entails for some concrete instances of $M$ and $K$. In \cite{DebruyneSeifert2} it was shown that one cannot improve the $M_K^{-1}$-estimate in Theorem \ref{thqikstahn} if $M_K(t) \ll e^{\alpha t}$ for some $\alpha > 0$, improving upon earlier optimality results from \cite{b-t, DebruyneSeifert1, stahn}. Furthermore, in \cite{stahn} the question was raised whether the restriction \eqref{eqqikcondstahn} is optimal in a suitable sense. The answer to this question is one of the main results of this article. We show that, at the cost of only a very mild regularity assumption, see \eqref{eqqikspecreggrowth}, one may entirely remove the hypothesis \eqref{eqqikcondstahn}. Additionally, we also improve the decay rate in Theorem \ref{thqikstahn} in some instances when $K$ is not of positive increase, for example in the important case $M(t) = K(t) = \max\{1,\log t\}$. Another main result of this paper is that we shall drastically weaken the restriction $M_K(t) \ll e^{\alpha t}$ for the optimality theorem.

Stahn's theorem, tailored for the application in the abstract Cauchy problem, has a few defects that somewhat inhibit its application in other fields. In particular, the lack of a one-sided Tauberian condition limits its use in number theory, where summatory functions of non-negative arithmetic functions are often studied, or even in probability theory, where the distribution functions must be non-decreasing.

The main goal of this paper is to devise a general quantified version of the Ingham-Karamata theorem that is applicable in a wide array of situations, in an attempt to bridge the gap between the versatile consumers of Tauberian theorems. Our main theorem is stated in section \ref{secstatemain}. Specifically, our study is centered around the following aspects.
\begin{enumerate}
 \item \textbf{Flexible one-sided Tauberian conditions}: We shall work under the Tauberian condition that $S(x) + F(x)$ is non-decreasing, where $F$ is a function that should satisfy a few mild regularity hypotheses, see section \ref{sectaubcon}. The function $F$ may be adapted depending on the application at hand and internalizes the desired flexibility. Of course the eventual rate in the quantified Tauberian theorem will inevitably depend on $F$. For example, if one wishes a one-sided version of Theorem \ref{thoriginal}, one may select $F$ to be $F(x) = Cx$, for some positive constant $C$. Our Tauberian condition covers in particular $S'(x) \geq - f(x)$ where the function function $f$ is reasonably flexible.
 
 \item \textbf{Various singularities on the Laplace transform}: A major advantage of our flexible Tauberian condition is that one may treat certain singularities at once. Concretely, if one wishes to find the asymptotic behavior of a non-decreasing function $S$, say, one should determine $F$ depending on the singularities, such that the Laplace transform $\mathcal{L}\{S - F; s\}$ does admit acceptable boundary behavior to apply a Tauberian theorem to $S - F$. The singularities of $\mathcal{L}\{S; s\}$ effectively describe the main terms of the asymptotic expansion of $S$. In this way, one may for example treat poles $s^{-k}$ corresponding to $F(x) = x^{k-1}/(k-1)!$, but one can also treat other wilder singularities, see section \ref{sectaubcon}.
 
Sometimes, such as in the \emph{Selberg-Delange method} \cite[Ch. II.5]{Tenenbaumbook} where one considers singularities related to $\zeta(s)^z$, $z$ being a complex number, the singularity must be expressed as an expansion. In our framework the function $F$ should encode a truncation of this expansion. Thus, an \emph{effective}\footnote{The location of the truncation in $F$ should be allowed to depend on the variable $x$ (from Theorem \ref{thqikstahn}).} Tauberian approach with a flexible one-sided Tauberian condition could serve as an alternative to the traditional Selberg-Delange method, but with the advantage that the boundary behavior of the Laplace transform may now be different from analyticity or differentiability assumptions. 
  
 \item \textbf{New boundary behavior of the Laplace transform}: Quantified versions of the Ingham-Karamata theorem thus far have featured mostly either analytic continuation or $C^{k}$-differentiability assumptions \cite{chill-seifert} on the Laplace transform. We shall also consider fractionally differentiable boundary behavior (H\"older continuity) and ultradifferentiable boundary behavior, a notion connecting $C^{\infty}$-functions and real analytic functions, see section \ref{qiksecprelim}. We also consider weighted $L^{p}$-variants related to these notions. These new boundary assumptions arise naturally in plenty of contexts (cf. e.g. \cite{kahane1, Malliavin}). 
 
 \item \textbf{Optimality}: The optimality of the obtained rates plays a central role in this work. We show these are optimal in almost all cases. Roughly speaking, we rigorously show optimality except under two specific circumstances. The first exception is when the bounds for the derivatives of the Laplace transform are very strong, close to being integrable. We remedy this by introducing, from the point of view of the Tauberian problem, a more appropriate notion of boundary behavior in terms of $L^1$-estimates for which one is able to always show optimality. The second exception is when the class of ultradifferentiability is close (or equal) to being analytic. Then there arises an extra restriction for the bound of the Laplace transform in terms of the shape of the region of analytic continuation or what corresponds to these notions in the ultradifferentiable setting. This extra restriction dissolves fairly rapidly as the class of ultradifferentiability strays further from encoding analyticity, see section \ref{qikremopti}. We establish in particular the optimality of the $M_K$-estimate in Theorem \ref{thqikstahn} in a broader range than the currently known $M_K(t) \ll \exp(\alpha t)$. Curiously, our new range is governed by a condition that is similar to and even slightly weaker than \eqref{eqqikcondstahn}, see \eqref{eqqikoptico}. 
 
 Nevertheless, even in the few cases where we do not quite reach optimality, we do establish a barrier for the best possible rate in our Tauberian theorem that is often fairly close to the one we actually achieve.   
 
 \item \textbf{Overall improvements}: As already mentioned, we improve the decay rate \eqref{qikeqmklogstahn} in Theorem \ref{thqikstahn} in a number of cases and omit the restriction \eqref{eqqikcondstahn} at the cost of only a mild regularity assumption.
\end{enumerate}

As a side-product of our investigations we also obtain the \emph{exact behavior} of the quantified Ingham-Karamata theorem, that is, an \emph{if-and-only-if}-theorem in the sense that the conclusion of the Tauberian theorem also implies that the Laplace transform must admit this exact behavior. This type of investigation was initiated by Korevaar in \cite{korevaar2005} where he, in a quest for a solution of the twin prime conjecture, established so-called \emph{local pseudofunction boundary behavior} as the exact behavior of the Wiener-Ikehara Tauberian theorem. A few years later local pseudofunction boundary behavior was also established as the exact behavior for the unquantified Ingham-Karamata theorem \cite{Debruyne-VindasComplexTauberians}. We mention \cite{Debruyne-VindasGeneralWI, olsen, zhang2014, zhang2019} for some alternative formulations of exact behavior in the Wiener-Ikehara Tauberian theorem. Exact behavior is especially useful in Beurling prime number theory \cite{Debruyne-VindasPNTEquivalences, d-v-l1} where the objects have very little a priori structure and where the better behavior that is otherwise assumed on the Laplace transform may not always be available. Yet these considerations may also be useful in the development of a convolution theory where the functions only admit one-sided estimates (corresponding to our flexible Tauberian condition) or in other domains where the integral transforms do not readily admit analytic continuations.

As a small sample result of our study we give the following theorem.
\begin{theorem}  \label{coran} Let $S: [0,\infty) \rightarrow \R$ be a function such that $S(x) + Cx$ is non-decreasing for some constant $C > 0$. Let $M, K$ be continuous non-decreasing functions on $\mathbb{R}_{+}$ such that $\mathcal{L}\{S;s\}$ admits an analytic extension to $\Omega_{M} $ where it is bounded by $K(\left|s\right|)/|s|$ as $|s| \rightarrow \infty$. Then, for any $c < 1$,
\begin{equation} \label{equsefulresult}
 S(x) \ll_{S,M,K,c} \frac{1}{M_{K,\log}\inv(cx)}, \ \ \ x\rightarrow \infty.
\end{equation} \label{eqqikspecr}
Furthermore, if $M_{K,\log}$ is of $M(t)t/675$-regular growth, that is, there exist $C',t_0 > 0$ such that
\begin{equation}  \label{eqqikspecreggrowth}
 \frac{M_{K,\log}(C't)}{M_{K,\log}(t)} \geq 1 + \frac{675}{M(t) t}, \ \ \ t \geq t_0,
\end{equation}
then \eqref{equsefulresult} holds with $c = 1$. 

If, additionally, $K(t)$ is of positive increase, see \eqref{eqqikpi}, or $M(t)/\log^{\beta} t$ is eventually non-decreasing for some $\beta > 0$ and $M_{K,\log}$ is of $M(t)t/1350$-regular growth, then
\begin{equation} \label{eqmkest}
 S(x) \ll_{S,M,K} \frac{1}{M_{K}\inv(x)}, \ \ \ x\rightarrow \infty.
\end{equation}

Moreover, if for some $C'' > 0$,
\begin{equation} \label{eqqikoptico} K(t) \leq \exp(\exp(C'' tM(t))), \ \ \ t \rightarrow \infty,
\end{equation}
then it is impossible to improve the decay rate in \eqref{equsefulresult} beyond $M_{K}\inv(x)\inv$.
\end{theorem}
We show Theorem \ref{coran} in section \ref{qiksecanal}, but the final optimality statement is given in Corollary \ref{coroptian} in section \ref{qikremopti}. We emphasize that with some more involved notation the techniques from this paper can provide an explicit description\footnote{The dependence of $S$ in \eqref{equsefulresult} and \eqref{eqmkest} is solely to control the Laplace transform on compacts within $\Omega_M$ (and to absorb the dependence of $C$). It could be omitted if one had a uniform bound for $\mathcal{L}\{S;s\}$ (in terms of $K$) on the whole of $\Omega_M$ instead of just for $|s| \rightarrow \infty$, but the implicit constant in \eqref{equsefulresult} should then depend additionally on $C$.} (in terms of $S,M,K$ and $c$) of the implicit constants in \eqref{equsefulresult} and \eqref{eqmkest}, see our main Theorem \ref{thrikimain}. 

We observe that with our Tauberian condition it is also possible to generate Tauberian theorems of Wiener-Ikehara type. As an illustration and for convenience of the reader we record the simplest case. More general versions involving more flexible Tauberian conditions allowing one to treat other singularities are also possible, but as the Wiener-Ikehara theorem is not the main focus of this work, we decide to not pursue those here. We also mention \cite[Th. 7.13, pp. 335]{Tenenbaumbook} and \cite{revesz-roton} for some other quantified effective versions of the Wiener-Ikehara theorem, although the Laplace transform hypothesis is quite different from the one here.

\begin{theorem}[Quantified Wiener-Ikehara with simple pole] Let $\tau: [0,\infty) \rightarrow \mathbb{R}$ be a non-decreasing function whose Laplace transform $\int^{\infty}_{0} \tau(x) e^{-sx} \dif x$ is convergent on $\Re s > 1$ and there exists $\mathfrak{A}> 0$ such that $\mathcal{L}\{\tau;s-1\} - \mathfrak{A}/s$ admits an analytic extension to $\Omega_{M}$ where it satisfies the bound $K(\left|s\right|)/|s|$ as $|s| \rightarrow \infty$. Here $M,K$ are as in Theorem \ref{coran}. Then, for any $c < 1$,
\begin{equation}\label{qikequsewi} \tau(x) = \mathfrak{A}e^{x} + O_{\tau,M,K,c}(e^{x}M_{K,\log}\inv(cx)\inv), \ \ \ x \rightarrow \infty. 
\end{equation}
If $M_{K,\log}$ is of $tM(t)/675$-regular growth, then $c$ may equal $1$ in \eqref{qikequsewi}. 

Furthermore, if additionally $K(t)$ is of positive increase or $M(t)/\log^{\beta} t$ is eventually non-decreasing for some $\beta > 0$ with $M_{K,\log}$ being of $tM(t)/1350$-regular growth, then
\begin{equation} \label{qikequsewimkest}
 \tau(x) = \mathfrak{A}e^{x} + O_{\tau,M,K}(e^x M_{K}\inv(x)\inv), \ \ \ x\rightarrow \infty.
\end{equation}
\end{theorem}
\begin{proof}[Deduction from Theorem \ref{coran}]
We let $S(x) := e^{-x} \tau(x) - \mathfrak{A}$ such that the Tauberian condition translates to $e^x(S(x) + \mathfrak{A})$ being non-decreasing. The Laplace transform of $S$ is $\mathcal{L}\{S;s\} = \mathcal{L}\{\tau; s-1\} - \mathfrak{A}/s$ and satisfies the hypotheses of Theorem \ref{coran}. It remains to verify that $S(x) + Cx$ is non-decreasing for a suitable $C$.

The traditional (unquantified) Wiener-Ikehara theorem (or a boundedness version \cite[Th. 3.1]{Debruyne-VindasComplexTauberians}) yields that $S(x) \ll_{\tau} 1$. This, together with the non-decrease of $e^x(S(x) +  \mathfrak{A})$, implies for $x,y > 0$ that
$$ S(x+y) -S(x) \geq (e^{-y} - 1)(S(x) + \mathfrak{A})) \gg_{\tau} -y,
$$
guaranteeing that $S(x) + Cx$ is non-decreasing for a suitable $C$.
\end{proof}

\medskip
Our work features a number of innovative methods that we anticipate should have utility in other Tauberian problems and even in other subjects within mathematical analysis. In contrast to \cite{stahn} where mainly complex analytic techniques were employed, our method shall be of pure Fourier analytic nature. We believe that a Fourier approach is more flexible, allowing us to treat more general behavior of the Laplace transform, but also providing a clean treatment of one-sided Tauberian conditions. An attractive feature of this technique is that it completely splits the analysis of the Tauberian condition and the analysis of the Laplace transform. Consequently, one may easily mix different Tauberian conditions with other types of behavior for the integral transform. For example, if in future investigations, one is able to deduce a corresponding Tauberian lemma under a flexible Wiener-Ikehara or gap-like Tauberian condition, one may combine it with our treatment of the behavior of the Laplace transform to obtain new Tauberian theorems. Vice versa, if one wishes to treat other boundary assumptions for the Laplace transform together with our flexible one-sided Tauberian condition, one may apply our Tauberian Lemma \ref{lemtaubarg2}. As a consequence, our investigations are also valuable in other contexts besides the Ingham-Karamata theorem. 

Our improvements largely stem from a good choice of test functions. For the Fourier method there is always a tension between the test function and its Fourier transform. Ideally both of these functions would have compact support; typically a compactly supported test function eases the analysis of the Tauberian condition while a compactly supported Fourier transform is beneficial for dealing with the behavior of the Laplace transform. Naturally, due to \emph{uncertainty principles} for the Fourier transform one cannot have both support conditions at the same time. In this work, we work with test functions whose Fourier transform have compact support and compromise on the Tauberian condition.

In order to handle one-sided Tauberian conditions efficiently, we shall select our test functions in such a way that $x\phi(x)$ is non-negative. Even though this new idea is fairly simple, it is surprisingly powerful, and significantly simplifies proofs of other important one-sided Tauberian theorems. We mention \cite{d-simpletauberian} where the author wrote an expository article based on this idea in a more simplified setting. 

The other main technical innovation in the paper is that, instead of selecting a single test function $\phi$ having the fastest possible decay that is admissible by the Denjoy-Carleman theorem---this would deliver the same restriction \eqref{eqqikcondstahn} as in Theorem \mbox{\ref{thqikstahn}---,} we select a \emph{sequence of test functions} pretending to be an analytic function with compact support; namely we choose a sequence of compactly supported test functions $\phi_n$ such that $\phi_n^{(n)}(t) \ll A^n n!$ (with some additional properties). If a single $\phi$ would obey these bounds on the derivatives, it must be analytic and therefore cannot be compactly supported. This construction is quite powerful and we anticipate it should have applications in plenty of other contexts within mathematical analysis as it can, sometimes, provide a rigorous proof of a heuristic argument requiring a function to be both analytic and compactly supported. The construction is based on Mandelbrojt's proof of the Denjoy-Carleman theorem \cite[Lemma 1.3.5]{hormander1990}, although the idea that such a sequence may fulfill the role of an analytic function is due to H\"ormander. Nevertheless, this technique appears not to be that well-known in a Tauberian context. 

For our optimality theorem, instead of constructing counterexamples which can be quite challenging in this general setting, we reason through an attractive argument based on the open mapping theorem. We learned this approach from \cite{ganelius}, where the idea is attributed to H\"ormander, but it seemed mostly forgotten until it resurfaced in Tauberian theory under the impulse of the author \cite{DebruyneSeifert1, DebruyneSeifert2, d-v-abswi} and his coauthors. The open mapping theorem turns the optimality question into a dual problem. Thus, as we appeal here to a Fourier method as well, we shall need \emph{dual test functions}. The test functions that we shall work with here were constructed in \cite{DebruyneSeifert2} and can be interpreted as some kind of dual sequence to the functions $\phi_n$ that will be constructed in the proof of our Tauberian theorem. These test functions $h_k$ are, roughly speaking, approximations of the indicator function $\chi_{E}$ where $E = \{t : |t| \geq 1\}$, are uniformly bounded and analytic in increasingly smaller strips around the real axis and have a zero of multiplicity $k$ at $0$, along with some other properties. As for the sequence $\phi_n$, we anticipate the sequence $h_k$ should also have applications in other areas. The core ideas in our argument are then related to \cite{DebruyneSeifert2}, but one however still requires a substantial amount of new technical work to apply these to our more general framework and to weaken the restriction $M_K(t) \ll e^{\alpha t}$ in the analytic case.

The proof of our Tauberian theorem is entirely \emph{effective}. Without always calculating explicit constants, we have made an effort to mention the dependence of the implicit constants on the external parameters in the Vinogradov symbol $\ll$ as these may be subject to optimization, depending on the application.  
On the other hand, our optimality result Theorem \ref{thrfrsop} is \emph{ineffective}, due to the application of the open mapping theorem.  

We formulate our results in terms of scalar-valued functions as we believe this is the most \emph{accessible} phrasing for a wide audience. It could be more challenging to recognize some other formulations that are specifically designed for concrete applications as the Ingham-Karamata Tauberian theorem. Moreover, one often even requires no substantial new input to establish these alternative formulations. For modern applications in $C_0$-semigroup theory for instance, one often applies a slightly modified form of the Ingham-Karamata theorem in terms of vector-valued functions, see Section \ref{secvec} for an argument how a vector-valued version of our Tauberian theorem can be deduced.

\medskip

The paper is centered around two main theorems: the Tauberian Theorem \ref{thrikimain} and its optimality, Theorem \ref{thrfrsop}. In Section \ref{qiksecstatemain}, we state our main theorem after a preliminary section and two sections where we explain in detail the flexible Tauberian condition and the different kinds of boundary behavior that we will consider for the Laplace transform. Section \ref{sectaub} is devoted to the treatment of the one-sided Tauberian condition. In section \ref{qikseckey} we show the key \emph{Tauberian lemma} and in section \ref{qiksecextest} we show that a sequence $\phi_n$ of test functions satisfying all desired properties in fact exists. In sections \ref{qiksecunik} and \ref{qiksecberess} we illustrate the strength of our approach by providing simplified proofs of an unquantified Ingham-Karamata theorem and the Berry-Esseen inequality. In section \ref{qiksechightaubcond} we briefly discuss how our method can be adapted to treat Tauberian conditions involving higher-order derivatives. In Section \ref{secest} we complete the proof of our main Theorem \ref{thrikimain} with an analysis of the different types of boundary behavior. In section \ref{sqikexact} we translate the convolution average arising in the Tauberian lemma \ref{lemtaubarg2} into a relation on the Fourier transform. Here we also derive the \emph{exact behavior} for our quantified Tauberian theorem. In the subsequent sections, we consider the boundary hypotheses separately. Specifically, in section \ref{qiksecanal} we complete the proofs of Theorems \ref{coran} and \ref{thrikimain}. In Section \ref{qiksecopti} we investigate whether the quantified rate obtained in our main Theorem \ref{thrikimain} is optimal. We state the optimality result in section \ref{secqikstate} and in the subsequent sections provide its proof. In the concluding section \ref{qikremopti} we discuss the employed methods in a broader context. Finally, in Section \ref{secvec} we briefly discuss an alternative scheme than what is usually done in the literature for obtaining vector-valued Tauberian theorems from their scalar-valued analogues. 

We have devised the main Theorem \ref{thrikimain} to be as general as reasonably possible without having overly technical assumptions. Throughout the paper we have included many remarks discussing the strengths and limitations of the methods, possible generalizations with more technical assumptions and other background information, yet the proof of the main theorem can be understood without these remarks. In particular, if one is only interested in the proof of Theorem \ref{coran} (without the optimality statement), it suffices to read sections \ref{qikseckey}, \ref{qiksecextest}, \ref{sqikexact}, \ref{secriknqu} and \ref{qiksecanal} without the remarks with the first item of Remark \ref{qikremnotherform} as the only exception. 

Our notation is standard throughout except that $\mathbb{N} = \{0, 1, 2, \dots\}$ includes $0$. As already mentioned we use the Vinogradov notation $f(x) \ll g(x)$ or $f(x) = O(g(x))$, that is, there exists a constant $C$ such that $f(x) \leq Cg(x)$ for all $x$ in the domain of $f$ and $g$. If the implicit constant $C$ depends on external parameters, these shall always be incorporated into the notation. If we write $f(x) \ll_{M_n} g(x)$, then the implicit constant is understood to depend on the whole sequence $M_n$ instead of only on the $n$'th entry. We write $f(x) = \Omega (g(x))$ if there exist $\varepsilon > 0$ and a sequence $x_n \rightarrow \infty$ such that $|f(x_n)| \geq \varepsilon g(x_n)$. We let $\chi_{E}$ denote the indicator function of a measurable set $E$, that is, $\chi_{E}(x) = 1$ if $x \in E$ and $0$ otherwise. $\R_{+}$ denotes the set of non-negative real numbers $[0,\infty)$ and $p,q \in [1,\infty]$ always represent \emph{conjugate} numbers, that is, $p^{-1} + q^{-1} = 1$. The letters $C$ or $C'$ shall always denote a positive constant, but these are not always the same throughout the paper.

The author wishes to thank Frederik Broucke and Jasson Vindas for valuable discussions and for proofreading previous versions of the manuscript.

\section{Statement of the main theorem} \label{qiksecstatemain}

\subsection{Preliminaries} \label{qiksecprelim}

We shall make occasional use of standard Schwartz distribution calculus in our manipulations throughout this work. Background material on distribution theory and Fourier transforms can be found in many classical textbooks, e.g. \cite{bremermann, hormander1990, vladimirov}; see \cite{p-s-v,vladimirov-d-z1} for asymptotic analysis and Tauberian theorems for generalized functions. 

If $U\subseteq\mathbb{R}$ is open, $\mathcal{D}(U)$ is the space of all smooth functions with compact support in $U$; its topological dual $\mathcal{D}'(U)$ is the space of distributions on $U$.
The standard Schwartz test function space of rapidly decreasing functions is denoted as usual by $\mathcal{S}(\mathbb{R})$, while $\mathcal{S}'(\mathbb{R})$ stands for the space of tempered distributions. The dual pairing between a distribution $f$ and a test function $\varphi$ is denoted as $\langle f, \varphi\rangle$, or  as $\langle f(x), \varphi(x)\rangle$ with the use of a dummy variable of evaluation. Locally integrable functions are regarded as distributions via $\langle f(x),\varphi(x)\rangle=\int_{-\infty}^{\infty}f(x)\varphi(x)\mathrm{d}x$. 

We fix the constants in the Fourier transform as
$\hat{\varphi}(t)=\mathcal{F}\{\varphi;t\}=\int_{-\infty}^{\infty}e^{-itx}\varphi(x)\:\mathrm{d}x.$ Naturally, the Fourier transform is well defined on $\mathcal{S}'(\mathbb{R})$ via duality, that is, the Fourier transform of $f\in\mathcal{S}'(\mathbb{R})$ is the tempered distribution $\hat{f}$ determined by $\langle \hat{f}(t),\varphi(t)\rangle=\langle f(x),\hat{\varphi}(x)\rangle$.
If $f\in\mathcal{S}'(\mathbb{R})$ has support in $[0,\infty)$, its Laplace transform is
$\mathcal{L}\left\{f;s\right\}=\left\langle f(u),e^{-su}\right\rangle,$ analytic on $\Re \:s>0$, and its Fourier transform $\hat{f}$ is the distributional boundary value of $\mathcal{L}\left\{f;s\right\}$ on $\Re \:s=0$. See \cite{vladimirov,zemanian} for complete accounts on Laplace transforms of distributions.

Let $M_n$ be a sequence of positive numbers. A smooth function $g : \R \rightarrow \C$ is said to be \emph{ultradifferentiable} at $t$ with respect to the sequence $M_{n}$ if there exist $B > 0$ and an open neighborhood $U$ of $t$ such that
\begin{equation} \label{eqriknnqu}
 \left|g^{(n)}(u)\right| \ll B^n M_{n}, \ \ \ \text{for all }n \in \N \text{ and } u \in U.
\end{equation}
The notion of ultradifferentiability bridges the gap between $C^{\infty}$-functions, where no uniform bounds on the derivatives are imposed, and \emph{real analytic} functions roughly corresponding to the sequence $M_n = n!$. Namely, if $M_{n} = n!$, the ultradifferentiable functions in $t$ for which \eqref{eqriknnqu} holds with the constant $B$ and interval  $U = (t-\varepsilon,t+\varepsilon)$ are analytic in the open ball of radius $\min\{B\inv,\varepsilon\}$ around $t$. A useful tool in this context is the so-called \emph{associated function}. This is
\begin{equation} \label{eqriknquass}                                                                                       
 M(x) = \sup_{n \in \mathbb{N}} \log\left(\frac{x^n}{M_{n}}\right), \ \ \ x > 0,
\end{equation}
whenever the supremum exists. For $M_{n} = n^n$ for instance, the associated function is $M(x) = x/e + O(1)$ as $x \rightarrow \infty$. 

Spaces of ultradifferentiable functions and their topological duals have been extensively studied in the literature. We recommend the interested reader to consult \cite{komatsu} for a detailed discussion. A central theme is to translate function-theoretic properties of the space into conditions for the sequence $M_n$. We briefly mention the results relevant to this work. 

The main properties of a sequence $M_n$ that we shall work with here are\footnote{The names (M.1) and (M.3)$'$ are standard in the theory of ultradifferentiable functions, but (SA) is not.} 
\begin{itemize}
\item[(M.1) :] \emph{logarithmic convexity}, $M_n^2 \leq M_{n-1}M_{n+1}$ for all $n \geq 1$,
\item[(M.3)$'$ :] \emph{non-quasianalyticity}, $\sum_{n = 0}^{\infty} \frac{M_n}{M_{n+1}} < \infty$.
\item[(SA) :] \emph{subanalyticity}, there exist $C,L > 0$ such that $M_n \geq CL^n n^n$ for all $n \in \N$.
 \end{itemize}

The \emph{Denjoy-Carleman theorem}, see e.g. \cite[Th. 1.3.8]{hormander1990}, asserts that, for a logarithmically convex sequence $M_n$, there exists a non-trivial compactly supported function $\phi$ such that $\phi$ is ultradifferentiable with respect to the class $M_n$ on $\mathbb{R}$ \emph{if and only if} the sequence $M_n$ satisfies (M.3)$'$. Correspondingly we shall say that a logarithmically convex sequence $M_n$ is \emph{quasianalytic} if $\sum_{n = 0}^{\infty} \frac{M_n}{M_{n+1}} = \infty$. The logarithmic-convexity-assumption (M.1) is no significant restriction as one may replace $M_n$ with the \emph{largest logarithmically convex minorant} of the sequence $M_n$, see e.g. \cite[Eq. (1.3.20)]{hormander1990}. For a logarithmically convex sequence $M_n$ is (M.3)$'$ also equivalent to $\int^{\infty}_1 M(y)y^{-2} \dif y < \infty$, with $M$ the associated function of $M_n$, see e.g \cite[pp. 91]{koosis}.
We also remark that, given two logarithmically convex non-quasianalytic sequences $M_n$ and $Q_n$, one may easily construct another logarithmically convex non-quasianalytic sequence that minorizes both $M_n$ and $Q_n$.

The subanalyticity assumption (SA) simply allows the functions in the ultradifferentiable class to admit a weaker (or equal) regularity than analyticity. We stress that there are sequences $M_n$, such as $M_n = (n \log (n+1))^n$ (with $M_0 =1)$, that are both quasianalytic and strictly subanalytic, that is, there are no constants $C, \tilde{L} > 0$ such that $n^n \geq C\tilde{L}^{n}M_n$ for all $n \in \mathbb{N}$.

A typical example are the \emph{Gevrey} ultradifferentiable functions corresponding to the sequence $M_n = (n!)^{1/\alpha}$, where $\alpha > 0$. The sequence is quasianalytic if and only if $\alpha \geq 1$ and its associated function is $ x^{\alpha}/\alpha - \log (x) / 2 + O_{\alpha}(1)$.

\subsection{The Tauberian condition} \label{sectaubcon}

We begin with a discussion of the \emph{Tauberian condition} treated in this work. A real-valued function $S: \R \rightarrow \R$ is said to satisfy the Tauberian condition $\mathfrak{T}(X,F,f, \alpha)$ if 
\begin{enumerate}
 \item $S$ is identically $0$ on the negative half-axis,
 \item $S(x) + F(x)$ is non-decreasing on $[X,\infty)$,
 \item the functions $F, f: \R \rightarrow \R_{+}$ are identically $0$ on the negative half-axis and satisfy 
 \begin{equation}\label{eqcructaubarg}
 \left|F(x+y)-F(x)\right| \leq f(x) |y| \exp(|y|^\alpha), \ \ \ \text{for all } x \geq X \text{ and } y \in \mathbb{R}, 
\end{equation}
with $0 \leq \alpha < 1$,
\item $ \displaystyle{ S(X) + F(X) \geq \sup_{y \leq X} \left(S(y) + F(y)\right) }$.
\end{enumerate} 
A complex-valued function $S : \R \rightarrow \C$ belongs to the class $\mathfrak{T}(X,F,f, \alpha)$ if both $\Re S$ and $\Im S$ satisfy the above requirements. 

As discussed in the Introduction the second property is the crucial one. The first property is only imposed because we work with Laplace transforms and the last one serves to achieve an essentially effective conclusion in the main Tauberian Theorem \ref{thrikimain} below. The third property is the mild regularity requirement of the function $F$. \par
The function $f$ in \eqref{eqcructaubarg} should be interpreted as the \emph{derivative} of $F$. In particular, if a positive function $f$ satisfies 
\begin{equation} \label{eqrikireg}
 f(x+y) \leq C f(x) \exp(\left|y\right|^{\alpha}), \ \ \ \text{for all } x \geq X  \text{ and } y \in \mathbb{R},
\end{equation}
for some $C > 0$ and $0 \leq \alpha < 1$, then \eqref{eqcructaubarg} is valid for $F(x) := \int^{x}_{0} f(y) \mathrm{d}y$ at the cost of replacing $f$ in \eqref{eqcructaubarg} with $Cf$. When the function $F$ is not explicitly mentioned, we always assume it represents the primitive of $f$. Thus, if $f$ satisfies \eqref{eqrikireg}, then the Tauberian condition $S'(x) \geq - f(x)$ guarantees that $S$ belongs to $\mathfrak{T}(0,F,Cf,\alpha)$. The classical Lipschitz Tauberian condition from Theorem \ref{thoriginal} can be encoded through $f(x) = Cx$ which satisfies \eqref{eqrikireg}. 

We underline that \eqref{eqrikireg} is not so restrictive. Many standard functions corresponding to standard singularities on the Laplace transform satisfy this requirement; we provide a non-exhaustive list below. We also note that \eqref{eqrikireg} is closed under taking sums and products, possibly with worse $C$ and $\alpha$. The functions $f$ are always $0$ on the negative half-axis. Therefore it suffices to verify \eqref{eqrikireg} for $x \geq X$, $y \geq -x$.
\begin{itemize}
 \item $f(x) = x^{\gamma}$ with $\gamma \geq 0$: $(x+y)^{\gamma} \leq x^{\gamma}(1+\left|y\right|)^{\gamma}$ for all $x \geq 1$ and $y\geq -x$.
\item $f(x) = \min\{1,x^{-\gamma}\}$ with $\gamma > 0$: for $x \geq 2$ and $|y| > x/2$, we have $f(x+y) \leq 1 \leq x^{-\gamma}2^{\gamma}|y|^{\gamma}$, whereas $f(x+y) \leq 2^{\gamma} x^{-\gamma}$ for $x \geq 2$ and $|y| \leq x/2$. 
\item $f(x)= \max\{0,\log^{\gamma}x\}$ with $\gamma \geq 0$: for $x\geq e$ we have $f(x+y) \leq f(x) (1 + \log(1+|y|))^{\gamma}$.
\item $f(x) = \log^{-\gamma}x$ with $\gamma > 0$ for $x \geq e$ and $0$ elsewhere: for $x \geq e$ and $|y| > x/2$ we have $f(x+y) \leq 1 \leq \log^{-\gamma} (x) \log^{\gamma}(2|y|)$ whereas $f(x+y) \leq (1-\log 2)^{-\gamma}\log^{-\gamma} x $ for $x \geq e$ and $|y| \leq x/2$.
\item $f(x) = \exp(\beta x^{\gamma})$ with $ 0 < \gamma < 1$ and $\beta > 0$: since $x^{\gamma}$ is subadditive and non-decreasing,
\begin{equation*}
 \exp(\beta (x+y)^{\gamma}) \leq \exp(\beta(x^{\gamma} + \left|y\right|^{\gamma})) = \exp(\beta x^{\gamma}) \exp(\beta \left|y\right|^{\gamma}), \ \ \ \ x \geq 0.
\end{equation*}
\end{itemize}
The above functions generally correspond to logarithmic, polynomial or exponential singularities for the Laplace transform.
\par
We discuss two possible generalizations of the Tauberian condition $\mathfrak{T}$ which come at the cost of bit more technical formulation. The third property of $\mathfrak{T}$ lays a restriction on the function $F$ and therefore on the treatable singularities of the Laplace transform. The key property of the factor $\exp(|y|^\alpha)$ in \eqref{eqcructaubarg} is that it is roughly the exponential of the associated function of a \emph{non-quasianalytic sequence}. It should be possible to establish our main Tauberian theorem under the Tauberian condition where $|y|^\alpha$ in \eqref{eqcructaubarg} is generalized to a general non-quasianalytic weight function $\Upsilon(|y|)$, that is $\int^{\infty}_0 \Upsilon(x)/(1+x^2) \dif x < \infty$.  

By a small adaptation of Lemma \ref{lemtaubarg2} it can be seen that our main Theorem \ref{thrikimain} below is actually valid under a slight generalization of the Tauberian condition $\mathfrak{T}$. Namely, one may replace the non-decrease of $\tau(x) := S(x) + F(x)$ on $[X,\infty)$ with
\begin{align*}  \tau(x+ y) - \tau(x) & \geq -E(x),  & x \geq X, \ \  y \geq 0, &\ \ \ \text{ and} \\
\tau(x+ y) - \tau(x) & \leq E(x), & x \geq X, \ \  y \leq 0,&
\end{align*}
where $E: [X,\infty) \rightarrow \R_{+}$ is a rate that is smaller than the one in the conclusion of the Tauberian Theorem \ref{thrikimain}, that is, $E(x) \ll$ the right-hand side of \eqref{eqrikires} below.  This generalization may be interpreted as a flexible quantified analogue of the frequently appearing \cite{korevaarbook} \emph{slowly decreasing} Tauberian condition, see also section \ref{qiksecunik}.

Finally we mention that our Tauberian condition can also be applied to summatory functions $\sum_{n \leq x} a_n$ with non-necessarily positive coefficients, but still satisfying $a_{n} \geq -b_n$ if the positive sequence $b_n$ is somewhat easier to handle than $a_n$. This can be done by applying a Tauberian theorem\footnote{In principle it is allowed to derive the asymptotics for $b_n$ via other non-Tauberian methods. Once these asymptotics for $b_n$ are established, this yields the exact behavior for the part of the integral transform coming from $b_n$. If the integral transform corresponding to $a_n$ then also admits this exact behavior, one can still apply an \emph{exact} Tauberian theorem to the sequence $a_n + b_n$.} to the non-decreasing functions $\sum_{n \leq x} (a_n + b_n)$ and $\sum_{n \leq x} b_n$ after subtracting the singularities of the integral transform which are encoded in the function $F$. We clarify by giving a simple example. Let $a_n = \mu(n)/n$, where $\mu$ denotes the M\"obius function, and $b_n = 1/n$. The Laplace transforms of the non-decreasing functions $\sum_{n \leq e^x} (\mu(n) + 1)/ n $ and $\sum_{n \leq e^x}  1/ n $ are respectively $1/ (s\zeta(s+1))  + \zeta(s+ 1)/s$ and $ \zeta(s+ 1)/s$, where $\zeta$ is the Riemann zeta function. Therefore, after subtracting the double pole at $s = 0$, which corresponds to subtracting $x + \gamma_0$---$\gamma_0 = 0.57\dots$ is the Euler-Mascheroni constant---from the functions $\sum_{n \leq e^x} (\mu(n) + 1)/ n $ and $\sum_{n \leq e^x}  1/ n$, both functions satisfy the Tauberian condition $\mathfrak{T}(0,x,1,0)$.

\subsection{The hypotheses on the Laplace transform}

In this section we introduce the different types of boundary behavior we consider for the Laplace transform. Here $N$ denotes a natural number, $B, D$ and $ \delta_0$ are positive real numbers, $M_n$ is a logarithmically convex \textbf{subanalytic} sequence of positive numbers, $1 \leq p \leq \infty$---in particular we treat $L^{\infty}$-bounds for the Laplace transform---and $G,H$ are positive (locally integrable) functions defined on appropriate domains. Furthermore $\omega: \R_{+} \rightarrow \R_{+}$ is a positive non-decreasing function satisfying $\omega(y) \gg_{\omega} y$ as $y \rightarrow 0^{+}$.
 
We treat the following function classes. Depending on the class we sometimes impose additional constraints on $G$ and $H$.
  
\begin{enumerate}
 \item $\mathcal{T}_{\mathrm{Dif}}(N,G,p,D)$ contains the functions $g: \R \rightarrow \C$ that are $N$ times \emph{differentiable} and satisfy $\|g^{(N)}(t)/G(t)\|_{L^{p}} \ll 1$ and\footnote{The hypothesis $\max_{0 \leq j < N} |g^{(j)}(0)| \leq D$ is only imposed to ensure that the implicit constant in the Tauberian theorem depends on an explicit parameter $D$ instead of on the values of (the derivatives of) the Laplace transform at $0$. This becomes relevant in our treatment of vector-valued functions in Section \ref{secvec} where one might otherwise argue that the implicit constant in the error term in the Tauberian theorem for $S_{\textbf{e}}$ might depend on $\textbf{e}$.} $\max_{0 \leq j \leq N-1} |g^{(j)}(0)| \leq D$.   
\item $\mathcal{T}_{\mathrm{DifI}}(N,G,D)$ is comprised of the functions $g: \R \rightarrow \C$ that are $N$ times \emph{differentiable}, satisfy $\int^{R}_{-R}|g^{(N)}(t)| \mathrm{d}t \ll G(R)$ for all $R > 0$ and $\max_{0 \leq j \leq N-1} |g^{(j)}(0)| \leq D$. Here $G: [0,\infty) \rightarrow \R$ shall always be non-decreasing.
 \item $\mathcal{T}_{\mathrm{HC}}(N,G,p,\omega,D,\delta_{0})$ contains the functions $g: \R \rightarrow \C$ that are $N$ times differentiable, satisfy $\max_{0 \leq j \leq N} |g^{(j)}(0)| \leq D$ and for which $g^{(N)}$ is \emph{uniformly H\"older continuous} with respect to $L^{p}$, $G$ and $\omega$, that is 
 \begin{equation} \label{eqrfruhc}
 \left\|\frac{\sup_{\left|u\right| \leq \delta}\left|g^{(N)}(t+u) - g^{(N)}(t)\right|}{G(t)}\right\|_{L^{p}} \ll \omega(\delta), \quad 0 < \delta \leq\delta_{0},
\end{equation}
where the $L^{p}$-norm is taken with respect to the variable $t$. These are weighted versions of \emph{Triebel-Lizorkin} spaces.
 \item $\mathcal{T}_{\mathrm{HCI}}(N,G,\omega,D,\delta_{0})$ comprises the functions $g: \R \rightarrow \C$ that are $N$ times differentiable, satisfy $\max_{0 \leq j \leq N} |g^{(j)}(0)| \leq D$ and for which
$$ \int^{R}_{-R} \sup_{|u| \leq \delta}|g^{(N)}(t+u) - g^{(N)}(t) | \mathrm{d}t \ll G(R) \omega(\delta), \ \ \ R > 0, \ \ \  0 < \delta \leq \delta_{0},
$$
Here $G$ is non-decreasing.  

 \item $\mathcal{T}_{\mathrm{SA}}(M_{n},B,G,H,p)$ consists of the functions $g: \R \rightarrow \C$ that are infinitely differentiable and for which there exist locally integrable functions $G_{n}$ such that 
 \begin{equation} \label{equltradif}
 \left\| \frac{g^{(n)}(t)}{G_{n}(t)} \right\|_{L^{p}} \ll B^{n} M_{n}, \ \ \ n \in \N,
\end{equation}
with
\begin{equation} \label{eqaux}
  \left\| G_{n}\chi_{(-\lambda,\lambda)}\right\|_{L^{q}} \leq G(\lambda)^{n} H(\lambda), \ \ \ \text{for all }n \text{ and }  \lambda > 0.
\end{equation}
\item $\mathcal{T}_{\mathrm{SAI}}(M_{n},B,G,H)$ comprises the infinitely differentiable functions $g: \R \rightarrow \C$ satisfying
\begin{equation} \label{eqqiksai} \int^{R}_{-R} |g^{(n)}(t)| \mathrm{d}t \ll B^{n}M_{n}G(R)^{n} H(R), \ \ \ R > 0, \ \ \ n \in \mathbb{N}.
\end{equation}
Here $G$ and $H$ are non-decreasing. 
\item $\mathcal{T}_{\mathrm{An}}(G,H)$: this class contains the functions $g: \R \rightarrow \C$ that admit an \emph{analytic} extension to $\{z : |\Im z| \leq 1/G(|\Re z|)\}$ where it satisfies the bound $H(|z|)$. The function $G: \mathbb{R}_{+} \rightarrow \mathbb{R}_{+} $ is continuous and non-decreasing with $G(0) \neq 0$. The function $tH(t)$ is non-decreasing on $\mathbb{R}_{+}$. 
\end{enumerate}
If all classes are considered at once, we use the notation $\mathcal{T}_\ast$. We remark that $\mathcal{T}_{\mathrm{SA}}$ is a generalization of the class $\mathcal{T}_{\mathrm{An}}$.

A common occurrence for $\omega$ is $\omega(\delta) = \delta^{\beta}$ with $0 < \beta < 1$, after which the function classes $\mathcal{T}_{\mathrm{HC}}$ and $\mathcal{T}_{\mathrm{HCI}}$ represent some form of uniform fractional differentiability, recovering in fact the familiar class of H\"older continuous functions when $p = \infty$.

We have chosen to write the bounds \eqref{eqaux} and \eqref{eqqiksai} in the form $G(\cdot)^n H(\cdot)$ for two reasons. First, this hypothesis often appears naturally in practice, see e.g. the treatment of $\mathcal{T}_{\mathrm{An}}$ in section \ref{qiksecanal}, and secondly it allows us to explicitly carry out a certain optimization providing a simpler expression for the quantified rate in the Tauberian Theorem \ref{thrikimain} below. However, we do emphasize that other expressions in terms of $n$ are possible in  \eqref{eqaux} and \eqref{eqqiksai} and those may in principle yield a better rate in the Tauberian theorem. We refer to Remark \ref{qikremnotherform} how our argument may be adapted to this situation.

The function classes $\mathcal{T}_{\mathrm{DifI}}$, $\mathcal{T}_{\mathrm{HCI}}$ and $\mathcal{T}_{\mathrm{SAI}}$ are introduced because, from the point of view of the Tauberian problem, they are more appropriate to work with (and may therefore in some applications deliver a better quantified rate) than the classes $\mathcal{T}_{\mathrm{Dif}}$, $\mathcal{T}_{\mathrm{HC}}$ and $\mathcal{T}_{\mathrm{SA}}$ respectively. In particular, in Section \ref{qiksecopti}, we manage to show the optimality of the obtained rate for the classes $\mathcal{T}_{\mathrm{DifI}}$, $\mathcal{T}_{\mathrm{HCI}}$ and $\mathcal{T}_{\mathrm{SAI}}$ without introducing extra lower bounds for the function $G$ (compared to the optimality theorem for the classes $\mathcal{T}_{\mathrm{Dif}}$, $\mathcal{T}_{\mathrm{HC}}$ and $\mathcal{T}_{\mathrm{SA}}$).

\subsection{The main theorem} \label{secstatemain}

We now arrive at our main theorem. Before its formulation we mention that if the Laplace transform admits a continuous extension to the line $\Re s = 0$, then it automatically has distributional boundary values, where the boundary distribution is the continuous function coming from the continuous extension. A similar comment applies to $L^{p}$-extensions. 

\begin{theorem} \label{thrikimain} Suppose $S \in L_{loc}^{1}(\mathbb{R}_{+})$ satisfies the Tauberian condition $\mathfrak{T}(X,F,f,\alpha)$. Assume that $\mathcal{L}\{S;s\}$ converges for $\Re s > 0$ and that it admits distributional boundary values on the line $\Re s = 0$, that is, there exists a distribution $g \in \mathcal{D}'(\R)$ such that
\begin{equation*} \lim_{\sigma \rightarrow 0+} \int^{\infty}_{-\infty} \mathcal{L}\{S; \sigma + it\} \phi(t) \dif t = \left\langle  g(t), \phi(t)\right\rangle, \ \ \ \text{for all } \phi \in \mathcal{D}(\R),
\end{equation*}
and that the boundary distribution $g$ belongs to the class $\mathcal{T}_{\ast}$. Then, for $x \geq X$ in case $\mathcal{T}_\ast \notin \{\mathcal{T}_{\mathrm{HC}} , \mathcal{T}_{\mathrm{HCI}}\}$, and $x \geq \max\{X, \pi/\delta_0, 2\pi\}$ in case $\mathcal{T}_\ast = \mathcal{T}_{\mathrm{HC}}$ or $\mathcal{T}_\ast = \mathcal{T}_{\mathrm{HCI}}$, there holds that
\begin{equation} \label{eqrikires}
 |S(x)| \ll_{\alpha} \inf_{\lambda \geq 1} E_\ast(x,\lambda) +  \frac{f(x)}{\lambda},
\end{equation}
where $E_{\ast}(x,\lambda)$ is an error term associated with the transform condition $\mathcal{T}_\ast$, see below.
\end{theorem}
\begin{itemize}
 \item $E_{\mathrm{Dif}}(N,G,p,D,x,\lambda) \ll_{N,D} x^{-N} (1 + \|G \chi_{(-\lambda,\lambda)}\|_{L^{q}})$. 
 \item $E_{\mathrm{DifI}}(N,G,D,x,\lambda) \ll_{N,D} x^{-N} (1 + G(\lambda))$.
 \item $E_{\mathrm{HC}}(N,G,p,\omega,D,\delta_{0},x,\lambda) \ll_{N,\omega,D,\delta_{0}} x^{-N} \omega(\pi/x) (1 +\|G \chi_{(-\lambda,\lambda)}\|_{L^{q}})$. 
 \item $E_{\mathrm{HCI}}(N,G,\omega,D,\delta_{0},x,\lambda) \ll_{N,\omega,D,\delta_{0}} x^{-N} \omega(\pi/x) (1+ G(\lambda))$.
\item $E_{\mathrm{SA}}(M_{n},B,G,H,p,x,\lambda) \ll_{M_{n}}  \exp\left(-M\left(\frac{x}{B G(\lambda) + \kappa\lambda^{-1}}\right)\right)H(\lambda)$, for some real constant $\kappa$ depending\footnote{An admissible value for $\kappa$ in $E_{\mathrm{SA}}$ and $E_{\mathrm{SAI}}$ is $\kappa = 124/L$ where $L$ is such that $M_n \gg L^n n^n$. In particular, for strictly subanalytic sequences any positive value for $\kappa$ is acceptable, but naturally the implicit constant in $E_{\mathrm{SA}}$ and $E_{\mathrm{SAI}}$ must then also depend on $\kappa$.} only on $M_{n}$ and where $M$ is the associated function of the sequence $M_{n}$. 
\item $E_{\mathrm{SAI}}(M_{n},B,G,H,x,\lambda) \ll_{M_{n}} \exp\left(-M\left(\frac{x}{BG(\lambda) +\kappa\lambda^{-1}}\right)\right)H(\lambda)$, for some real constant $\kappa$ depending only on $M_{n}$. 
\item $E_{\mathrm{An}}(G,H,x,\lambda) \ll  \exp\left(-\frac{x}{G(\lambda)(1+\kappa G(\lambda)\inv\lambda^{-1})}\right) \left\{\int^{\lambda}_{0}H(t)\mathrm{d}t + G(0)\inv \displaystyle{\max_{t \leq 5G(0)\inv}} H(t)\right\} \allowbreak \ll_{G(0),H} \exp\left(-\frac{x}{G(\lambda)(1+\kappa\lambda^{-1})}\right)\int^{\lambda}_{0}H(t)\mathrm{d}t$ if $\lambda \geq 2 G(0)\inv$ and $E_{\mathrm{An}}(G,H,x,\lambda) = \infty$ otherwise. An admissible value for $\kappa$ here is $675$.
\end{itemize} \par
We inserted the middle expression in $E_{\mathrm{An}}$ to clarify that the implicit constant of the Vinogradov symbol is essentially independent of $G$ and $H$. This can be important in applications where it is advantageous to allow the region of analytic continuation (and according bounds) to be dependent on $x$.

In principle one can retrieve the explicit dependence of $N$ in $E_{\mathrm{Dif}}$ (and in the other $E_\ast$) through a careful analysis of our proof, but we choose not to formulate them for the following reason. Our estimates under $\mathcal{T}_{\mathrm{Dif}}$ were established with the understanding that one only knows bounds for the $N$'th derivative of the Laplace transform and consequently the quality of our estimates in terms of $N$ is often rather poor. In applications where $N$ is to be optimized, there are usually good quantitative estimates available for all derivatives smaller than $N$ as well and incorporating this information will often lead to better estimates in terms of $N$. This occurs for example in the Selberg-Delange method if one wishes to optimize the location of the truncation of the expansion of a singularity. In this situation we strongly advise to follow the argument in section \ref{secriknqu} and Remark \ref{qikremnotherform} which improves the quality of the effective estimates in terms of $N$.

 \par 
The dependence on $\alpha$ of the implicit constant in \eqref{eqrikires} is usually not that important---it often suffices to take $\alpha$ fixed---but can be turned explicit. By exploiting that $x/\log^2 x$ is a non-quasianalytic weight function, the proof below delivers that an acceptable additional factor would be $\int^{\infty}_{1} x\exp(x^{\alpha} - x/\log^2 x) \dif x \ll \sup_{x \geq 1} \exp(x^{\alpha} - x/\log^2 x + 3 \log x) \ll \exp (\exp(3\log(1/(1-\alpha))/(1-\alpha)))$, see Lemmas \ref{lemtaubarg2} and \ref{lemanalyticcuts}, but by no means we claim that this bound should be optimal. We also remark that the dependence of the implicit constant in \eqref{eqrikires} on the implicit constants introduced in the Vinogradov notation in $\mathcal{T}_{\ast}$ is linear, that is, if e.g. in \eqref{equltradif} the implicit constant is $C \geq 1$, then the implicit constant in \eqref{eqrikires} should be multiplied with $C$ as well.

In Appendix \ref{rfrappendix} we have performed the optimization in \eqref{eqrikires} for several important concrete boundary assumptions for the Laplace transform under the classical Tauberian condition of boundedness (from below) of the derivative of $S$.

\section{The Tauberian argument: Estimation of the high frequencies} \label{sectaub}

In this Section we present an innovative argument that permits us to efficiently handle one-sided Tauberian conditions. In fact, up until now Tauberian theorems with one-sided conditions often required a more involved and delicate treatment, see e.g. \cite{Debruyne-VindasComplexTauberians}, \cite[Ch. III]{korevaarbook} or \cite[Ch. II.7]{Tenenbaumbook}. We demonstrate the strength of our technique in section \ref{qiksecunik} and \ref{qiksecberess} by simplifying the proofs of two influential one-sided Tauberian theorems, an unquantified version of the Ingham-Karamata theorem and the Berry-Esseen inequality.

The central idea of our approach is explained in section \ref{qikseckey}. In section \ref{qiksecextest} we construct a sequence of test functions satisfying the necessary requirements of the Tauberian Lemma \ref{lemtaubarg2}. This sequence is, to a large extent, responsible for the removal of the hypothesis \eqref{eqqikcondstahn} in Theorem \ref{thqikstahn}. In section \ref{qiksechightaubcond} we explain how one may adapt our argument for some other one-sided Tauberian conditions.

\subsection{The key lemma} \label{qikseckey}

The key idea to handle the one-sided Tauberian condition is captured in the following lemma.

\begin{lemma} \label{lemtaubarg2} Suppose $S$ is a real-valued function satisfying the Tauberian condition $\mathfrak{T}(X,F,f,\alpha)$. Let $\phi: \mathbb{R} \rightarrow \mathbb{R}$ be such that 
\begin{enumerate}
\item $\int^{\infty}_{-\infty} \phi = 1$, 
\item $\int^{\infty}_{-\infty} \left|y| \exp(|y|^\alpha) |\phi(y)\right| \mathrm{d}y \leq C$,
\item $\phi(y) \geq 0$ for $y \geq 0$ and $\phi(y) \leq 0$ for $y \leq 0$ .
\end{enumerate}
Then, for each $\lambda \geq 1$ and $x\geq X$,
\begin{equation} \label{eqrfrtm}
\lambda \int^{\infty}_{-\infty} S(x+ y) \phi(-\lambda y)\mathrm{d}y - \frac{Cf(x)}{\lambda} \leq S(x) \leq  \lambda \int^{\infty}_{-\infty} S(x+ y) \phi(\lambda y)\mathrm{d}y + \frac{Cf(x)}{\lambda}.
\end{equation}
\end{lemma}
\begin{proof} 
As $S(x) + F(x)$ is non-decreasing from $X$ onwards, we obtain for all $x \geq X$, 
\begin{align*}
S(x) & =  \lambda \int^{\infty}_{-\infty} S(x) \phi(\lambda y)\mathrm{d}y \\
& \leq \lambda \int^{\infty}_{-\infty} S(x+ y) \phi(\lambda y)\mathrm{d}y + \lambda \int^{\infty}_{-\infty} (F(x+y)-F(x)) \phi(\lambda y) \mathrm{d}y.
\end{align*}
Observe that the sign of $\phi$ ensures the validity of the inequality for both positive and negative $y$. Notice also that the inequality $S(x+y) + F(x+y) \leq S(x) + F(x)$ remains valid for $y \leq X-x$ because of the fourth property of $\mathfrak{T}$.

The second term can be estimated via \eqref{eqcructaubarg}. For $\lambda \geq 1$ we obtain
\begin{align*}
\lambda \int^{\infty}_{-\infty} (F(x+y)-F(x)) \phi(\lambda y) \mathrm{d}y& \leq   f(x) \int^{\infty}_{-\infty} |\lambda y|\exp(|y|^\alpha) |\phi(\lambda y)|\mathrm{d}y\\
& \leq   \frac{f(x)}{\lambda}\int^{\infty}_{-\infty} |y| \exp(|\lambda^{-1}y|^\alpha) |\phi(y)|\mathrm{d}y\\
& \leq   \frac{f(x)}{\lambda}\int^{\infty}_{-\infty} |y| \exp(|y|^\alpha) |\phi(y)|\mathrm{d}y\\
& \leq   \frac{Cf(x)}{\lambda}.
\end{align*}
 The lower inequality for $S(x)$ can be deduced analogously. The key difference is that the function $\phi(-y)$ is now non-positive for $y \geq 0$ and non-negative for $y \leq 0$. 
\end{proof}

Applying Lemma \ref{lemtaubarg2} to $\Re S$ and $\Im S$, one finds that for \emph{complex-valued} functions $S$ satisfying the Tauberian condition $\mathfrak{T}(X,F,f,\alpha)$, one has 
\[  |S(x)| \leq 2\lambda \max\left\{  \abs[3]{  \int^{\infty}_{-\infty} S(x+ y) \phi(-\lambda y)\mathrm{d}y},  \abs[3]{\int^{\infty}_{-\infty} S(x+ y) \phi(\lambda y)\mathrm{d}y} \right\} + \frac{2Cf(x)}{\lambda},
\]
if $x \geq X$ and $\lambda \geq 1$.

\subsection{Existence of test functions} \label{qiksecextest}

We now wish to show the existence of a test function $\phi$ that satisfies the hypotheses of Lemma \ref{lemtaubarg2} and some additional properties for its Fourier transform $\hat{\phi}$ that allow us to adequately handle the behavior of the class $\mathcal{T}_{\ast}$. In case the Laplace transform admits an ultradifferentiable extension with respect to the sequence $M_n$, one ideally has that $\hat{\phi}$ is compactly supported and sufficiently regular, in the sense that there exist $C, A > 0$ such that
\begin{equation} \label{eqderivfour}
 |\hat{\phi}^{(n)}(t)| \leq CA^{n}M_{n}, \ \ \ \text{for all } n \in \N \text{ and } t \in \R.
\end{equation}

Unfortunately, in the important case of $M_n = n^n$ which corresponds to analytic continuation, such a test function $\phi$ does not exist. In fact, in that case, the estimates \eqref{eqderivfour} imply that $\hat{\phi}$ is real analytic and this cannot be reconciled with $\hat{\phi}$ being non-trivial and compactly supported. In order to remedy this, we present here a new technique, which enables us to remove the hypothesis \eqref{eqqikcondstahn} in Theorem \ref{thqikstahn}.\par
Although it is impossible to find a single $\hat{\phi}$ for which \eqref{eqderivfour} holds for \emph{all} $n$ at once, it is possible to find a sequence $\phi_{n}$ such that a slightly modified form of \eqref{eqderivfour} holds for all derivatives up to $n$ with some degree of uniformity. Our construction is inspired by H\"ormander's treatment of the Denjoy-Carleman theorem, particularly \cite[Th. 1.3.5]{hormander1990}.

\begin{lemma} \label{lemanalyticcuts} Let $0 < \gamma < 1$. There exists a sequence of real-valued functions $\phi_{n} \in \mathcal{S}(\mathbb{R})$ such that 
\begin{enumerate}
\item $\int^{\infty}_{-\infty} \phi_{n}(y) \mathrm{d}y = 1$,
\item $\phi_{n}(y) \ll_{\gamma}\exp(-|y|^{\gamma})$, 
\item $y\phi_{n}(y)$ is non-negative, 
\item$\operatorname*{supp} \hat{\phi}_{n} \subseteq [-1,1]$,
\item there exist constants $C,A > 0$ such that for all $n \in \mathbb{N}$,
\begin{equation} \label{eqanalyticcuts}
 \sup_{t \in [-1,1]} \left|\hat{\phi}_{n}^{(j)}(t)\right| \leq C A^{j}n^{j}, \ \ \ j \leq n.
\end{equation}
\end{enumerate}
An admissible value for $A$ is $124$.
\end{lemma}
\begin{proof} Before we start the construction of the sequence, we introduce an auxiliary function that will enable us to achieve the decay rate $\phi_{n}(y) \ll_{\gamma}\exp(-|y|^{\gamma})$. Let $\psi \in \mathcal{D}(-1/4,1/4)$ be such that $\hat{\psi}(y) \ll_{\gamma}\exp(-|y|^{\gamma})$ and $\int^{\infty}_{-\infty} \psi = 1$. Its existence is guaranteed by the Denjoy-Carleman theorem. By considering, if necessary, the function $|\psi(t)|^2 = \psi(t) \overline{\psi(t)}$ and starting with a slightly larger $\gamma$, we may assume that $\psi$ is real-valued and non-negative. Furthermore, by, if necessary, convolving $\psi$ with $\chi_{[-1/2,1/2]}$ and rescaling with a factor $3$, we may assume that $\psi(t) \geq 1$ on the interval $(-1/12, 1/12) = (-2\varepsilon, 2\varepsilon)$ after selecting $\varepsilon = 1/24$. \par 
Now we begin constructing the sequence $\phi_n$. We modify H\"ormander's construction slightly; we allow the positive numbers $a_{j,n}$ to depend also on $n$. We may suppose without loss of generality that $n \geq 1$. We define a sequence of $n-1$ times differentiable real-valued functions $u_n$ as
\begin{equation}
 u_{n} = \frac{1}{2^{n+1}a_{0,n} \dots a_{n,n}} \left(\chi_{[-a_{0,n},a_{0,n}]} \ast \dots \ast \chi_{[-a_{n,n},a_{n,n}]} \right).
\end{equation}
We observe that $u_n$ is non-negative, supported on $[-(a_{0,n} + \dots + a_{n,n} ),a_{0,n} + \dots + a_{n,n}]$, has integral $\int^{\infty}_{-\infty} u_n = 1$ and its Fourier transform is bounded by
\begin{equation}
 |\hat{u}_{n}(t)| = \left|\prod_{j = 0}^{n} \frac{\sin(a_{j,n}t)}{a_{j,n}t}\right| \leq 1, \ \ \ t \in \mathbb{R}.
\end{equation}
Its derivatives satisfy
\begin{equation} \label{eqancut}
 \left|u_{n}^{(j)}(x)\right| \leq \frac{1}{2 a_{0,n} \dots a_{j,n}}, \ \ \ j \leq n-1.
\end{equation}
We choose $a_{0,n} = a_{1,n} = \varepsilon/4$, $a_{2,n} = \dots = a_{n,n} = \varepsilon/2(n-1)$. Thus $u_n$ is supported on $[-\varepsilon,\varepsilon]$, satisfies $|u_n| \leq 2/\varepsilon$ and
\begin{equation} \label{eqvn}             
 \left|u_{n}^{(j)}(x)\right| \leq 2^{j+2}\varepsilon^{-j-1}(n-1)^{j-1} = (4/\varepsilon)(2/\varepsilon)^{j}(n-1)^{j-1}, \ \ \ 1 \leq  j \leq n-1.
\end{equation}
This essentially establishes property \eqref{eqanalyticcuts} which will play a crucial role in Section \ref{secest}. In the rest of the proof we apply some technical modifications to this sequence in order for the other properties of the Tauberian Lemma \ref{lemtaubarg2} to be fulfilled. 

We define a first candidate $\varphi_{1,n}$ for the sequence $\phi_n$ through its Fourier transform. We set $\hat{\varphi}_{1,n}(t) = \psi \ast u_{n} \in \mathcal{D}(-1/2,1/2)$. We immediately obtain the decay estimate $\varphi_{1,n}(y) = (2\pi)^{-1} \hat{\psi}(-y)\hat{u}_n(-y) \ll_{\gamma} \exp(-|y|^{\gamma})$, but the functions $\varphi_{1,n}(y)$ may not be real-valued. We therefore consider $\varphi_{2,n}(y) = |\varphi_{1,n}(y)|^{2} =\varphi_{1,n}(y) \overline{\varphi_{1,n}(y)}$. Its Fourier transform is $\hat{\varphi}_{2,n} = (2\pi)^{-1}(\hat{\varphi}_{1,n}(\cdot) \ast \overline{\hat{\varphi}_{1,n}(-\cdot)}) \in \mathcal{D}(-1,1)$. As $\int^{\infty}_{-\infty} \hat{\varphi}_{1,n} = \int^{\infty}_{-\infty} \psi = 1$, the bounds on the derivatives \eqref{eqvn} also hold with $\hat{\varphi}_{2,n}$ replacing $u_n$. It now only remains to ensure the non-negativity of $y\phi_n(y)$ and $\int^{\infty}_{-\infty} \phi_n = 1$.

As the real-valued functions $\varphi_{2,n}$ belong to $\mathcal{S}(\mathbb{R})$, the integral $\int^{\infty}_{-\infty} y\varphi_{2,n}(y) \dif y =: c_n$ must always be real. We let $d_n = \int^{\infty}_{-\infty} y \varphi_{2,n}(y-\pi/\varepsilon) \dif y$ and define
$$ \varphi_{3,n}(y) = \begin{cases}
 c_n^{-1}y\varphi_{2,n}(y), & \text{if } c_n \geq 1/2,\\
 |c_n|^{-1} y \varphi_{2,n}(-y), & \text{if } c_n \leq -1/2, \\
 d_n^{-1} y \varphi_{2,n}(y-\pi/\varepsilon), & \text{if } |c_n| < 1/2.
\end{cases}
$$
We note that 
$$ d_n - c_n = \frac{\pi}{\varepsilon}\int^{\infty}_{-\infty} \varphi_{2,n}(y) \dif y = \frac{\pi}{\varepsilon} \int^{\infty}_{-\infty}| \varphi_{1,n}(y)|^2 \dif y = \frac{1}{2 \varepsilon} \int^{\infty}_{-\infty}| \hat{\varphi}_{1,n}(t)|^2 \dif t \geq 1,
$$
as for $|t|\leq \varepsilon$,
$$ \hat{\varphi}_{1,n}(t) = \int^{\varepsilon}_{-\varepsilon} u_n(w) \psi(t-w) \dif w \geq 1,
$$
because $\psi \geq 1$ on $(-2\varepsilon,2\varepsilon)$. Therefore $d_n$ is always larger than $1/2$ if $|c_n| < 1/2$. 

Finally we set $\phi_n(y) = \varphi_{3,n+2}(y)$. One can easily check that all the desired properties are fulfilled. The only one that requires a little effort is the verification of \eqref{eqanalyticcuts} in case $|c_{n+2}| < 1/2$. Then we find for $j \leq n$,
\begin{align*} |\hat{\phi}_n^{(j)}| & = |d_{n+2}^{-1}(e^{-i\pi ./\varepsilon}\hat{\varphi}_{2,n+2})^{(j+1)}| \leq 2 \sum_{\ell = 0}^{j+1} {j+1 \choose \ell} \left(\frac{\pi}{ \varepsilon}\right)^{\ell} |\hat{\varphi}_{2,n+2}^{(j+1-\ell)}| \\
& \leq \frac{8}{\varepsilon} (n+1)^j \sum_{\ell = 0}^{j+1} {j+1 \choose \ell} \left(\frac{\pi}{ \varepsilon}\right)^{\ell} \left(\frac{2}{\varepsilon}\right)^{j+1-\ell} = \frac{8}{\varepsilon} \left(\frac{\pi + 2}{\varepsilon} \right)^{j+1} (n+1)^j\\
& \leq \frac{8e}{\varepsilon} \left(\frac{\pi + 2}{\varepsilon} \right)^{j+1} n^j. 
\end{align*}
This concludes the proof of this lemma. The admissible value for $A$ follows as $(\pi +2)/ \varepsilon \approx 123.39 < 124$.
\end{proof}
We emphasize that the constants $C$ and $A$ in \eqref{eqanalyticcuts} are independent of $n$.\par
\begin{remark} \ 
\begin{enumerate}
\item When the sequence $M_n$ is non-quasianalytic, one can pick a single test function $\phi$ satisfying the first four properties of Lemma \ref{lemanalyticcuts} and \eqref{eqderivfour}. A non-trivial $\phi$ satisfying \eqref{eqderivfour} and the first, second and fourth property of Lemma \ref{lemanalyticcuts} is provided by the Denjoy-Carleman theorem, while the third property can be guaranteed by performing similar manipulations as in the above proof. 
\item We note that in an unpublished work Stahn successfully achieved some results with the Fourier method for $\mathcal{T}_{\mathrm{An}}$ employing just a Denjoy-Carleman argument, but he also had to impose here the additional growth condition \eqref{eqqikcondstahn} on the analytic extension of the Laplace transform. 
\item 
The main goal of the construction of the sequence $\phi_n$ is to have a compactly supported function $\phi_n$ realizing \eqref{eqderivfour} for the $n$'th derivative of $\phi_n$ (with some manageable control on the lower-order derivatives). It is impossible to achieve Lemma \ref{lemanalyticcuts} for a sequence $M_n$ that grows slower than $n!$. In other words, the constant $A$ in \eqref{eqanalyticcuts} cannot be taken arbitrarily close to $0$. 

Namely, let $\varphi \in \mathcal{D}(\R)$ with $\varphi \equiv 0$ on $(-\infty, 0)$. Then, for $t \geq 0$,
$$ |\varphi(t)| = \left|\int^{t}_{0} \dif t_1 \int^{t_1}_0 \dif t_2 \  \dots \int^{t_{n-1}}_0 \varphi^{(n)}(t_n) \dif t_n \right| \leq \frac{t^n}{n!}\max_{t' \in [0,t]} |\varphi^{(n)}(t')|.
$$

So, if $\int^{\infty}_{-\infty} \phi_n = \hat{\phi}_n(0) = 1$ and $\hat{\phi}_n(t) \equiv 0$ for $t \in (-\infty,-1)$, we can apply the above argument to $\varphi(t) = \hat{\phi}(t-1)$ to find $t'$ such that $|\hat{\phi}_n^{(n)}(t')| \geq n!$. In view of Stirling's formula, the constant $A$ in Lemma \ref{lemanalyticcuts} cannot be $e^{-1}$.

We also observe that the value of $A$ in Lemma \ref{lemanalyticcuts} is independent of $\gamma$, although the optimal value for $A$ may in principle still depend on it. We have put no effort into optimizing $A$.
\end{enumerate}
\end{remark}

To conclude this Section we wish to elucidate the strength of the technique from Lemma \ref{lemtaubarg2} by significantly simplifying the proofs of two one-sided Tauberian theorems. The reader solely interested in a proof of Theorem \ref{thrikimain} may jump straight to the next Section \ref{secest}.

\subsection{Application: the unquantified Ingham-Karamata theorem} \label{qiksecunik}
 We considerably simplify the proof of \cite[Th. 3.1]{Debruyne-VindasComplexTauberians}, a boundedness theorem related to the Ingham-Karamata theorem. It states that if a locally integrable function $S : [0,\infty) \rightarrow \R$ which is \emph{boundedly decreasing}, that is, there exists $v_0 > 0$ such that
 $$ \liminf_{x \rightarrow \infty} \inf_{v \in [0,v_0]} (S(x+v) - S(x)) > -\infty,
 $$
has a Laplace transform, assumed to initially converge on $\Re s > 0$, that admits \emph{local pseudomeasure boundary behavior} near $s =0$, then $S$ is eventually bounded, that is $\limsup_{x \rightarrow \infty} |S(x)| < \infty$.

Local pseudomeasure behavior near $s = 0$ plays the role of the exact boundary behavior behavior of the Laplace transform for this theorem. One of its characterizations is that there exists $\varepsilon > 0$ such that $\langle\hat{S}(t), e^{ivt} \phi(t)\rangle = O_{S,\phi}(1)$ as $v \rightarrow \infty$ for all test functions $\phi \in \mathcal{D}(-\varepsilon, \varepsilon)$ and where $\hat{S}$ is the boundary distribution of the Laplace transform on $\Re s = 0$.  
 
We remark that a weaker version of this boundedness theorem had already been communicated by Korevaar \cite[Ch. III, Prop. 10.2]{korevaarbook} but with a flawed proof, see \cite[Remark 3.2]{Debruyne-VindasComplexTauberians}. 

Concretely, we follow the proof of \cite[Th. 3.1]{Debruyne-VindasComplexTauberians} up to step 2. These first steps consist essentially of some preliminary work to set the stage properly, while the crux of the proof lies in step 3 where in \cite{Debruyne-VindasComplexTauberians} a reasonably elaborate argument is given to conclude the eventual boundedness from above of $S(x)$. We now replace that argument. We use the following ingredients that were established in step 2: from the bounded decrease of $S$, we infer that we may without loss of generality assume the existence of $C,X> 0$ such that
\begin{equation} \label{eqqikunq}
 S(y) - S(x) \geq -C(y-x+1), \ \ \ \text{for all }X \leq x \leq y,
\end{equation}
and from $S(x) = O_{S}(x)$, $x \rightarrow \infty$, and the local pseudomeasure boundary behavior, it follows that $S \ast \phi(x) = O_{S,\phi}(1)$ for all $\phi \in \mathcal{F(\mathcal{D}(-\varepsilon,\varepsilon))}$ for a sufficiently small $\varepsilon > 0$, by a similar argument as in step 1 of the proof. (We use dominated convergence instead of monotone convergence to justify the switch of limit and integral.) 

We now pick a fixed $\phi \in \mathcal{F}(\mathcal{D}(-\varepsilon,\varepsilon))$ satisfying the properties of Lemma \ref{lemtaubarg2}. The existence of such a $\phi$ is for instance provided by Lemma \ref{lemanalyticcuts} after a rescaling. We may suppose that $X$ is so large that $I_x := \int^{\infty}_{X-x} \phi(y) \dif y$ belongs to the interval $[1,2]$ for all $x \geq 2X$. Applying \eqref{eqqikunq} we obtain that
\begin{align*}
 S(x) &  = \frac{1}{I_x}  \int^{\infty}_{X-x} S(x) \phi(y) \mathrm{d}y \\
 & \leq \frac{1}{I_x}\int^{\infty}_{X-x} S(x+y) \phi(y)\mathrm{d}y + \frac{C}{I_x} \int^{\infty}_{X-x} y\phi(y)\mathrm{d}y + \frac{C}{I_x} \int^{\infty}_{X-x} \left|\phi(y)\right|\mathrm{d}y\\
& = \frac{1}{I_x}\int^{\infty}_{-\infty} S(x+y) \phi(y)\mathrm{d}y + O_{S,C,\phi,X}(1)  = \frac{S \ast \check{\phi}(x)}{I_x}  + O_{S,C,\phi,X}(1) = O_{S,C,\phi,X}(1),
\end{align*}
as $x \rightarrow \infty$ where $\check{\phi}(x) = \phi(-x)$. The lower bound for $S$ can be derived analogously by replacing $\phi$ with $\check{\phi}$ or even via Step 4 of the proof of \cite[Th. 3.1]{Debruyne-VindasComplexTauberians}.

\subsection{Application: the Berry-Esseen inequality} \label{qiksecberess}
As another illustration of the strength of our technique we present a new short proof of the \emph{Berry-Esseen inequality}, which is more direct than the classical one involving the non-negative test function $(1-\cos(x))/x^2$, see e.g. \cite[Lemma 1.47]{elliott} or \cite[pp. 536--538]{feller}. The Berry-Esseen inequality plays a crucial role in obtaining quantified central limit theorems in probability theory. 

\begin{theorem}[Berry-Esseen] Let $F,G$ be two distribution functions\footnote{A function $F: \mathbb{R} \rightarrow [0,1]$ is called a \emph{distribution function} if it is non-decreasing, right continuous and satisfies $F(-\infty) = 0$ and $F(\infty) = 1$. Its \emph{characteristic function} $f$ is given by $f(t) = \int^{\infty}_{-\infty} e^{itx} \mathrm{d}F(x)$.}. For each $T > 0$, there holds that
\begin{equation} \label{eqrfrbe}
 \sup_{x \in \mathbb{R}} |F(x) - G(x)| \leq 10 \sup_{\substack{x \in \mathbb{R} \\ 0 \leq y \leq 1/T}} \left(G(x+y) - G(x)\right) + \frac{1}{5}\int^{T}_{-T} \left|\frac{f(t) - g(t)}{t}\right| \mathrm{d}t,
\end{equation}
where $f$ and $g$ are the characteristic functions of $F$ and $G$ respectively.
\end{theorem}
\begin{proof} Let $S(x) = F(x) - G(x)$ and let 
$ \mathfrak{G} := \sup \left(G(x+y) - G(x)\right)
$, the supremum taken over $x \in \R$ and $0 \leq y \leq 1/T$. As $F$ and $G$ are distribution functions, $S$ is bounded and belongs thus to $\mathcal{S}'(\R)$. Since $F$ is non-decreasing, $S$ also satisfies the following inequalities,
\begin{align*}
 S(x) &\leq S(x+y) + \mathfrak{G}(Ty + 1),   & \text{if } y \geq 0, \\
 S(x) &\geq S(x+y) - \mathfrak{G}(T|y| + 1), & \text{if } y \leq 0.
\end{align*} 
Therefore, if $\phi \in \mathcal{F}(\mathcal{D}(-1,1))$ such that $\int^{\infty}_{-\infty} \phi = 1$ and $y\phi(y)$ is non-negative, we find
\begin{align*}
 S(x) &= \int^{\infty}_{-\infty} S(x) T\phi(Ty)\mathrm{d}y \\
& \leq  \int^{\infty}_{-\infty} S(x+y) T\phi(Ty)\mathrm{d}y + \mathfrak{G}\int^{\infty}_{-\infty} T^{2}y \phi(Ty)\mathrm{d}y + \mathfrak{G}\int^{\infty}_{-\infty} T |\phi(Ty)|\mathrm{d}y \\
& = \frac{1}{2\pi}\int^{\infty}_{-\infty} \hat{S}(t)e^{ixt} \hat{\phi}(-t/T)\mathrm{d}t + \mathfrak{G} \int^{\infty}_{-\infty} y\phi(y) + |\phi(y)|\mathrm{d}y,
\end{align*}
where $\hat{S}$ is to be interpreted in the space $\mathcal{S}'(\R)$, but as we verify below, this corresponds to integration with $(it)^{-1}(f(-t) -g(-t))$. 

Indeed, one may suppose that $(it)^{-1}(f(-t) -g(-t))$ is integrable near $t = 0$ as otherwise the conclusion \eqref{eqrfrbe} becomes trivial. The claim $\left\langle \hat{S}, \psi \right\rangle= \int^{\infty}_{-\infty} (it)^{-1}(f(-t) -g(-t)) \psi(t) \dif t$ is also easily checked if $\psi(0) = 0$. On the other hand, for $\psi \in \mathcal{D}(\R)$, one has
\begin{align*}
\int^{\infty}_{-\infty} (it)^{-1}(f(-t) -g(-t)) \psi(t) \dif t &= \lim_{u \rightarrow \infty} \int^{\infty}_{-\infty}(it)^{-1}(f(-t) -g(-t)) (1 - e^{iut})\psi(t) \dif t \\
& = \lim_{u \rightarrow \infty} \left\langle \hat{S}(t), (1-e^{iut})\psi(t)\right\rangle \\
& = \lim_{u \rightarrow \infty} \left\langle S(y) - S(y+u),\hat{\psi}(y)\right\rangle = \left\langle \hat{S}(t),\psi(t)\right\rangle,
\end{align*}
because of the Riemann-Lebesgue lemma and the fact that $\lim_{u \rightarrow \infty} S(u) = 0$. As $\mathcal{D}(\R)$ is dense in $\mathcal{S}(\R)$, the claim follows.
 
 The Berry-Esseen inequality (with unspecified constants) now follows by taking a valid $\phi$ from for example Lemma \ref{lemanalyticcuts}. To get the constants in \eqref{eqrfrbe}, one may take 
 \begin{equation} \label{eqqikexpex}
\phi(x) = \frac{x\sin^{4}\left(\frac{x}{4} - \frac{3}{2\pi}\right)}{16\left(\frac{x}{4} - \frac{3}{2\pi}\right)^4}.
\end{equation}
 This choice of $\phi$ is permitted since, as $S$ is bounded and therefore belongs to the dual of $L^1(\R)$, a density argument allows one to replace the condition $\phi \in \mathcal{F}(\mathcal{D}(-1,1))$ with $\phi \in L^1(\R)$ and $\operatorname*{supp} \hat{\phi} \subseteq [-1,1]$. The function $\hat{\phi}$ is supported on $[-1,1]$ with $\max|\hat{\phi}(t)| \approx 1.22$, while $\int^{\infty}_{-\infty} y\phi(y) \dif y \approx 8.19$ and $\int^{\infty}_{-\infty} |\phi|\approx 1.61$. 
\end{proof}
We did not attempt to optimize the constants in \eqref{eqrfrbe} here, which in applications usually play no significant role. The function $\phi$ in \eqref{eqqikexpex} was obtained by performing the procedure explained in the proof of Lemma \ref{lemanalyticcuts}, but by only considering a fourfold convolution of step functions and by not implementing the function $\psi$.

\subsection{Tauberian conditions with higher-order derivatives} \label{qiksechightaubcond}

So far we have treated the Tauberian condition $S(x) + F(x)$ being non-decreasing, which roughly corresponds to a one-sided hypothesis for the first derivative of $S$. With more effort and a somewhat more involved notation, our argument can be adapted to generate one-sided Tauberian lemmas for Tauberian conditions incorporating higher-order derivatives of $S$.

We consider the Tauberian condition $\mathfrak{T}_{m}(X, F_m,f_m,\alpha)$ where $X$ is a real number, $0 < \alpha < 1$ and $F_m, f_m$ are real-valued locally bounded functions\footnote{We observe that the local boundedness of $F_m$ together with \eqref{eqqikexpm} implies that $F_m(x) \ll_{m,f_m,\beta} \exp(|y|^{\beta})$ for all $\beta > \alpha$.}. A real-valued function $S : \R \rightarrow \R$ belongs to the class $\mathfrak{T}_{m}$ if
\begin{enumerate}
\item $S$ is identically $0$ on $(-\infty,0)$ with\footnote{We imposed this extra assumption to ensure that the argument in section \ref{sqikexact} below to translate the integral $\int^{\infty}_{-\infty} S(x+y) \phi(\lambda y) \dif y$ in terms of the Laplace transform may be adapted without too much difficulty. We stress that this should not be regarded as a two-sided restriction. Namely, if $S$ is $m$ times differentiable, the convergence of the Laplace transform of $S$ on the half-plane $\Re s > 0$ together with the one-sided assumption $S^{(m)}(x) \geq 0$ for sufficiently large $x$ already implies that $e^{-\sigma x}S(x) \ll_{\sigma} 1$, see e.g. the argument at the beginning of the proof of \cite[Th. 2.2]{d-simpletauberian}.} $ e^{-\sigma x} S(x) \in \mathcal{S}'(\R)$ for every $\sigma > 0$, 
\item there holds
\begin{equation} \label{eqqiktaubm} y^m \Delta^{m}_{y}(\tau; x) := y^m \sum_{j = 0}^m  (-1)^{m-j} {m \choose j} \tau(x + jy) \geq 0, \ \ \ x \geq X,\ \  y \in \R,
\end{equation}
where $\tau(x) = \tau_m(x) := S(x) + F_{m}(x)$,
\item there holds
\begin{equation}  \label{eqqikexpm} |\Delta^{m}_{y}(F_m;x)| \leq |y|^m f_m(x) \exp(|y|^\alpha),  \ \ \ x \geq X, \ \ y \in \mathbb{R},
\end{equation}
\item for even $m$, there holds
\begin{equation} \label{eqqikexpmadj} |\Delta^{m}_{y}(F_m;x-y)| \leq |y|^m f_m(x) \exp(|y|^\alpha),  \ \ \ x \geq X, \ \ y \in \mathbb{R}.
\end{equation}
\end{enumerate}
The function $f_m$ is to be interpreted as the $m$'th derivative of $F_m$ and often the fourth property follows from the same calculations needed to verify the third one. Again the crucial property is the second one which is a generalization of $\tau^{(m)}(x) \geq 0$ (for $x \in \R$). Note also that the second property of $\mathfrak{T}_{m}(X,F_m,f_m,\alpha)$ covers both the second and fourth property of $\mathfrak{T}(X,F,f,\alpha)$ defined in section \ref{sectaubcon}.

\begin{lemma} \label{lemtaubargm} Suppose $S: \R \rightarrow \R$ satisfies the Tauberian condition $\mathfrak{T}_{m}(X,F_m,f_m,\alpha)$. Let $\phi$ be a real-valued function such that 
\begin{enumerate}
\item $\int^{\infty}_{-\infty} \phi = 1$, 
\item $\int^{\infty}_{-\infty} \left|y|^m \exp(|y|^\alpha) |\phi(y)\right| \mathrm{d}y \leq C_m$,
\item $y^m\phi(y) \geq 0$ for all $y \in \R$.
\end{enumerate}
Then, for each $\lambda \geq 1$ and $x\geq X$,
\begin{align} 
 |S(x)|  \label{eqqikreshigh} & \leq 2^m \lambda \max_{\substack{ |j| \leq m \\  0 \neq j \in \mathbb{Z}}} \left|\int^{\infty}_{-\infty} S(x+ jy) \phi(\lambda y)\mathrm{d}y\right|+ \frac{C_mf_m(x)}{\lambda^m}.
\end{align}
\end{lemma}
We omit the proof which is similar to Lemma \ref{lemtaubarg2}. For even $m$ one cannot use the test function $\phi(-y)$ to get the reverse inequality as $\phi$ is then non-negative everywhere. In this situation one exploits \eqref{eqqiktaubm} in a slightly different way to get both inequalities. For example, for $m = 2$ one may use $\tau(x) \geq 2\tau(x+y) - \tau(x+2y)$ to obtain the lower inequality for $S$ and for the reverse inequality one can use $\tau(x) \leq 2^{-1}(\tau(x-y) + \tau(x+y))$, which is also valid\footnote{For even $m$ the symmetry of $\eqref{eqqiktaubm}$ implies its validity if either $x \geq X$ or $x+my \geq X$. In case $m = 2$, this implies that $\tau(x-y) - 2\tau(x) + \tau(x+y) \geq 0$ is also valid for $x \geq X$ and all $y \in \R$.} for $x \geq X$ and $y \in \R$. The only caveat is that one now needs to apply \eqref{eqqikexpmadj} instead of \eqref{eqqikexpm}. The treatment for general even $m$ is similar.

In view of an eventual one-sided Tauberian theorem in the style of Theorem \ref{thrikimain}, it is technically more accurate to delay the application of the triangular inequality resulting in the factor $2^m$ in \eqref{eqqikreshigh} until after one translated the integrals in terms of the Laplace transform with the argument from section \ref{sqikexact}. Namely, the application of the monotone convergence theorem below may become problematic for the integral $\int^{\infty}_{-\infty} S(x+jy) \phi(\lambda y) \dif y$, but remains valid if $S(x+jy)$ is replaced with $\Delta^{m}_{y}(S; x) +S(x)$, $\Delta^{m}_{-y}(S; x) +S(x)$, $\Delta^{m}_{y}(S; x) -S(x)$ or $ m^{-1}\Delta^{m}_{y}(S;x-y) + S(x)$, the different expressions appearing in the proof of Lemma \ref{lemtaubargm} prior to the application of the triangular inequality.

If $\phi$ satisfies the fourth and fifth property of Lemma \ref{lemanalyticcuts}, then the integrals $\int^{\infty}_{-\infty} S(x+ jy) \phi(\lambda y)\mathrm{d}y$ can be estimated if $\hat{S} \in \mathcal{T}_\ast$ with the argument of the next Section. We also remark that for even $m$, Lemma \ref{lemtaubargm} requires that the third property of Lemma \ref{lemanalyticcuts} be replaced with $\phi(y) \geq 0$ for all real $y$. This adaptation of Lemma \ref{lemanalyticcuts} can be shown via an analogous, but simpler argument than the one presented in section \ref{qiksecextest} and we leave it to the reader.

\begin{remark}
Lemma \ref{lemtaubargm} works well when $m$ is a fixed number. However, in some applications a priori estimates on infinitely many derivatives of $S$ may be available and then it may be advantageous to additionally optimize the parameter $m$. In this regard we mention that the factor $2^m$ in \eqref{eqqikreshigh} can sometimes be significantly improved. Rather than exploiting \eqref{eqqiktaubm} directly to get an inequality for $\tau(x)$ in terms of its translations, one can apply it to $\Delta^{m}_{y}(\tau; x- \ell y)$ (or $\Delta^{m}_{-y}(\tau; x+ \ell y)$), where $\ell$ is a natural number, although one might need to slightly alter the definition of $\mathfrak{T}_m$ to certify one can apply the inequality \eqref{eqqiktaubm} and \eqref{eqqikexpm} to this expression for all $x \geq X$ and $y \in \R$. The advantage is that the coefficient of the term $\tau(x)$ is larger than $1$ which leads to an improvement. Taking $\ell$ in the vicinity of $m/2$, the factor $2^m$ in \eqref{eqqikreshigh} becomes the total number of terms in the sum of \eqref{eqqiktaubm} divided by the coefficient of the term $\tau(x)$ and which is about $2^m/{m \choose m/2}\ll \sqrt{m}$ in view of Stirling's formula.

Yet, more significant is the quantity $C_m$ which will generally grow much faster. If one requires that $\hat{\phi}$ is compactly supported, then Lemma \ref{lemanalyticcuts} allows one to pick $\phi$ such that $\phi(y) \ll_{\beta} \exp(-|y|^\beta)$ for some\footnote{In principle the Denjoy-Carleman theorem offers a slightly better rate.} $\alpha < \beta < 1$ such that $C_m \ll_{\alpha,\beta} A_{\alpha,\beta}^m (m!)^{1/\beta}$ for some $A_{\alpha,\beta} > 0$.

In some circumstances the quality of the estimate for $C_m$ can be improved if one relaxes the assumption of the support of $\hat{\phi}$. Namely, for certain hypotheses imposed on the Laplace transforms of $S$, comparable or slightly worse estimates to the ones achieved in Section \ref{secest} can be obtained for $\int^{\infty}_{-\infty} S(x+ y) \phi(\lambda y)\mathrm{d}y$ if $\hat{\phi}(y)$ is only expected to obey a certain decay rate instead of being compactly supported. Then, depending on the decay of $\hat{\phi}$, the admissible decay for $\phi$ may be strengthened, perhaps even to $\phi$ being compactly supported, and this leads to better estimates for $C_m$. In this context, it is worthwhile to raise the following question which is analogous to Lemma \ref{lemanalyticcuts}.

\begin{openproblem} Given $\xi: \R_+ \rightarrow \R$ a non-decreasing function. What is the best possible sequence $M_n$ such that there exists a sequence of real-valued functions $\psi_n \in \mathcal{S}(\R)$ for which $\psi_n(0) = 1$, $|\psi_n(t)| \leq \exp(-\xi(|t|))$ and
$$ |\psi_n^{(n)}(t)| \leq M_n, \ \ \ \text{for all }n \in \N \text{ and all } t \in \R? 
$$ 
\end{openproblem}
This question is related to the non-triviality of spaces of type $\mathcal{S}$, sometimes also referred to as \emph{Gelfand-Shilov} spaces.

\end{remark}

\section{The estimation of the low frequencies: Completion of the proof of Theorem \ref{thrikimain}} \label{secest}
                     
\subsection{Translation of the convolution: the exact behavior}  \label{sqikexact}                   
                                                                                                                                      
As follows from Lemma \ref{lemtaubarg2}, our remaining task is to estimate $\int^{\infty}_{-\infty} S(x+y) \phi(\lambda y) \mathrm{d}y$ by exploiting the information of the class $\mathcal{T}_\ast$ which the boundary distribution of the Laplace transform of $S$ belongs to. We first translate this integral expression in terms of the Fourier transform of $S$. In the argument below we assume without loss of generality that $S$ is real-valued.

Let $\phi \in \mathcal{S}(\R)$ satisfy the properties from Lemma \ref{lemtaubarg2} and have a Fourier transform supported on $[-1,1]$. We observe that \eqref{eqcructaubarg} implies that $F(x) \ll_{F,X,f} (1+|x|)\exp(|x|^{\alpha})$. Hence, from the convergence of $\mathcal{L}\{S;s\}$ on $\Re s > 0$ and the Tauberian condition that $S(x) + F(x)$ is non-decreasing (from $X$ onwards), one may verify that $S(x) \ll_{S,\sigma} e^{\sigma x}$ for $\sigma > 0$ and $x \geq X$ and thus that $e^{-\sigma x}S(x) \in \mathcal{S}'(\R)$ for each $\sigma > 0$ , see e.g. the argument at the beginning of the proof of Theorem 2.1 in \cite{d-simpletauberian}. We obtain for $\lambda \geq 1$ and $x \geq X$ that
 \begin{align*}
 \lambda \int^{\infty}_{-\infty} S(x+y) \phi(\lambda y) \mathrm{d}y & = \lim_{\sigma \rightarrow 0^{+}} \lambda \left(\int^{\infty}_{0} + \int^{0}_{-x}\right) e^{-\sigma (x+y)}\left(S(x+y) + F(x+y)\right) \phi(\lambda y) \mathrm{d}y \\
& \quad - \lim_{\sigma \rightarrow 0^{+}} \lambda \int^{\infty}_{-\infty} e^{-\sigma (x+y)} F(x+y) \phi(\lambda y) \mathrm{d}y\\
& = \lim_{\sigma \rightarrow 0^{+}} \lambda \int^{\infty}_{-\infty} e^{-\sigma (x+y)}S(x+y) \phi(\lambda y) \mathrm{d}y \\
& = \lim_{\sigma \rightarrow 0^{+}}\frac{1}{2\pi}\int^{\infty}_{-\infty} \mathcal{L}\{S;\sigma + it\} e^{ixt} \hat{\phi}(-t/\lambda)\mathrm{d}t\\
& = \frac{1}{2\pi}\left\langle g(t), e^{ixt} \hat{\phi}(-t/\lambda)\right\rangle,
\end{align*}
where the switch of limit and integral in the first line is justified by dominated convergence on the interval $(-x,0)$ and by monotone convergence on $(0,\infty)$; notice that the fourth property of $\mathfrak{T}$ implies that $S(x) + F(x) \geq 0$ for $x \geq X$. The switch of limit and integral for the integral involving only $F$ is justified by dominated convergence and the second property of Lemma \ref{lemtaubarg2}.
\par
In the remainder of this Section, we derive estimates for $\left\langle g(t), e^{ixt} \hat{\phi}(-t/\lambda)\right\rangle$ separately for each class $\mathcal{T}_\ast$ concluding the proof of Theorem \ref{thrikimain}. The integral $\lambda \int^{\infty}_{-\infty} S(x+y) \phi(-\lambda y) \mathrm{d}y = (2\pi)^{-1}\left\langle g(t), e^{ixt} \hat{\phi}(t/\lambda)\right\rangle$ can be estimated analogously and its analysis shall be omitted.

We shall not explicitly treat $\mathcal{T}_{\mathrm{DifI}}$, $\mathcal{T}_{\mathrm{HCI}}$ and $\mathcal{T}_{\mathrm{SAI}}$. Their analysis is entirely analogous to $\mathcal{T}_{\mathrm{Dif}}$, $\mathcal{T}_{\mathrm{HC}}$ and $\mathcal{T}_{\mathrm{SA}}$ respectively, with the only exception that one avoids using H\"older's inequality.

Before we complete the proof of the main Theorem \ref{thrikimain}, we take a slight detour to briefly discuss the \emph{exact behavior} of the quantified Ingham-Karamata theorem under flexible Tauberian conditions.

\begin{remark} The above argument and Lemma \ref{lemtaubarg2} establish the exact behavior for the Laplace transform of the quantified Ingham-Karamata theorem under the flexible one-sided Tauberian condition $\mathfrak{T}$.

Indeed, suppose that $E(x): \R \rightarrow \R_+$ is an admissible estimate for $S$, that is $S(x) \ll E(x)$, which satisfies the regularity assumption $E(x+ y)/E(x) \ll 1 + |y|^{\gamma}$ for some $\gamma > 0$; in particular $E$ is polynomially bounded and does not decay faster than the inverse of all polynomials. Then, for $\lambda \geq 1$, these assumptions yield
$$ \lambda \int^{\infty}_{-\infty} S(x + y) \phi(\lambda y) \dif y \ll_\phi  E(x), \ \ \ \text{for all } \phi \in \mathcal{S}(\R),
$$ 
as elements from $\mathcal{S}(\R)$ decay faster than any polynomial. Applying the Fourier transform, we obtain in particular,
\begin{equation} \label{eqqikexact} \left\langle g(t), e^{ixt} \varphi(t/\lambda)\right\rangle \ll_\varphi E(x), \ \ \ \text{for all } \varphi \in \mathcal{D}(\R),
\end{equation} 
where $g$ is the boundary distribution of the Laplace transform of $S$ on the line $i\R$. We emphasize that the implicit constant in \eqref{eqqikexact} is independent of $\lambda$.

Conversely, if $\eqref{eqqikexact}$ is fulfilled for some function $\lambda = \lambda(x) \geq f(x)/E(x)$, then Lemma \ref{lemtaubarg2} and the discussion at the beginning of this section imply that $S(x) \ll_{S} E(x)$. In other words, if $S$ satisfies the Tauberian condition $\mathfrak{T}$ and $E$ obeys $E(x+ y)/E(x) \ll 1 + |y|^{\gamma}$, then $S(x) \ll_{S} E(x)$ is valid if and only if \eqref{eqqikexact} holds for all functions $\lambda(x) \geq f(x)/E(x)$ (where the implicit constant in \eqref{eqqikexact} depends additionally on $S$). Therefore, \eqref{eqqikexact} represents the exact behavior of Theorem \ref{thrikimain}, provided the regularity assumption on $E$.

\medskip

Naturally, one may replace in the exact behavior the validity of \eqref{eqqikexact} for all functions $\lambda \geq f(x)/E(x)$ with the validity for some function $\lambda \geq f(x)/E(x)$. One may also replace in \eqref{eqqikexact} the class $\mathcal{D}(\R)$ with $\mathcal{S}(\R)$ or even with a set of two functions $\{\varphi_{1},\varphi_{2}\}$ for which the Fourier transforms $\hat{\varphi}_1(x)$ and $\hat{\varphi}_2(-x)$ satisfy the hypotheses of Lemma \ref{lemtaubarg2} for $\phi$. We decided to state \eqref{eqqikexact} with the class $\mathcal{D}(\R)$ as this closely resembles the \emph{local pseudofunction} exact behavior of the unquantified Ingham-Karamata theorem, where $E(x)$ in \eqref{eqqikexact} is replaced with a $o(1)$-function. Furthermore, requiring that $\varphi$ is compactly supported is often advantageous in applications, as demonstrated by our present study, but also by the works \cite{Debruyne-VindasPNTEquivalences, d-v-l1} in the unquantified setting.

\medskip

If $E(x)$ decays faster than any polynomial, then $S(x) \ll E(x)$ no longer guarantees \eqref{eqqikexact} for all $\varphi \in \mathcal{D}(\R)$, since the decay of $\phi$ may be too slow to witness $\lambda \int^{\infty}_{-\infty} S(x + y) \phi(\lambda y) \dif y \ll_\phi  E(x)$. A solution is to lay further restrictions on the space $\mathcal{D}(\R)$. If the regularity assumption for $E$ is replaced with $E(x+ y)/E(x) \ll 1 + \exp(|y|^{\alpha})$, for some $\alpha < 1$, then, mutatis mutandis, one can derive \eqref{eqqikexact} for all $\varphi \in \mathcal{D}_{x^\alpha}(\R)$ from $S(x) \ll E(x)$. Here $\mathcal{D}_{x^\alpha}(\R)$ consists of all smooth compactly supported functions $\varphi$ whose Fourier transform obey $\hat{\varphi}(x) \ll \exp(-c |x|^{\alpha})$ for all $c > 0$. As follows by the Denjoy-Carleman theorem this space is non-trivial if $\alpha < 1$.

If $E(x)$ decays even faster, such as $E(x) = \exp(-x)$, or any other (quasi-analytic) rate $E(x)$ such that $\hat{\varphi}(x) \ll E(x)$ inhibits the existence of a non-trivial compactly supported $\varphi(t)$, then it is no longer possible to achieve the exact behavior \eqref{eqqikexact} with respect to a class of compactly supported functions $\varphi$. The exact behavior may still be reformulated in another way. For example, if the regularity condition on $E$ is relaxed to $E(x+ y)/ E(x) \ll_C 1$ for all $|y| \leq C$, then $S(x) \ll E(x)$ implies \eqref{eqqikexact} for all $\varphi \in \mathcal{F}(\mathcal{D}(\R))$, yet the lack of compactly supported functions in this space can be a significant challenge in applications. 

\medskip

Finally, we comment that the function $\varphi$ in \eqref{eqqikexact} may be allowed to depend on $x$. Namely, if for each (sufficiently large) $x$, there exist $\varphi_{1,x}, \varphi_{2,x} \in \mathcal{D}(\R)$ whose Fourier transforms satisfy the hypotheses of Lemma \ref{lemtaubarg2} for $\phi$ and, for some $\lambda = \lambda(x) \geq f(x)/E(x)$ there holds
\begin{equation} \label{eqqikexact2} \left\langle g(t), e^{ixt} \varphi_{1,x}(t/\lambda)\right\rangle \ll E(x), \ \ \  \left\langle g(t), e^{ixt} \varphi_{2,x}(-t/\lambda)\right\rangle \ll E(x),
\end{equation}
then the Tauberian condition $\mathfrak{T}$, Lemma \ref{lemtaubarg2} and the discussion at the beginning of this section imply that $S(x) \ll E(x)$. Of course it is important here that the implicit constants in \eqref{eqqikexact2} are independent of $x$. The observation that we may allow the test function $\varphi$ to depend on $x$ plays a crucial role in our treatment of the class $\mathcal{T}_{\mathrm{An}}$ below.
 
\end{remark}

\subsection{Differentiable functions: the class $\mathcal{T}_{\mathrm{Dif}}$}

In this section we assume that $g \in \mathcal{T}_{\mathrm{Dif}}(N,G,p,D)$, $\lambda \geq 1$ and $\hat{\phi} \in \mathcal{D}(-1,1)$. 

We first estimate $g^{(j)}$ in terms of $g^{(N)}$ where $j \leq N$. By successively applying the formula $g^{(n)}(t) = g^{(n)}(0) + \int^{t}_{0} g^{(n+1)}(u) \dif u$ repeatedly, we obtain for $0 \leq j \leq N$, 
\begin{equation*}
\int^{\lambda}_{-\lambda} \left|g^{(j)}(t)\right| \mathrm{d}t \ll \sum_{\ell = 0}^{N-j-1}\lambda^{\ell+1} |g^{(j+\ell)}(0)| + \lambda^{N-j} \int^{\lambda}_{-\lambda} \left|g^{(N)}(t)\right| \mathrm{d}t.
\end{equation*}
Therefore, integrating by parts $N$ times gives
\begin{align*}
 \left|\left\langle g(t), e^{ixt} \hat{\phi}(-t/\lambda)\right\rangle\right| & = \left|\frac{1}{x^{N}} \int^{\lambda}_{-\lambda} e^{ixt}  \left(g(t) \hat{\phi}(-t/\lambda)\right)^{(N)} \mathrm{d}t \right|\\
& \ll_{N,\phi} \frac{1}{x^{N}} \max_{0 \leq j \leq N} \lambda^{j-N} \int^{\lambda}_{-\lambda} \left|g^{(j)}(t)\right| \mathrm{d}t\\
&\ll_{N,D} \frac{1}{x^{N}}\left( 1 +  \int^{\lambda}_{-\lambda} |g^{(N)}(t)| \mathrm{d}t\right).
\end{align*}
An application of H\"older's inequality now yields
\begin{equation} \label{eqrfrdif}
 \left|\left\langle g(t), e^{ixt} \hat{\phi}(-t/\lambda)\right\rangle\right| \ll_{N,\phi,D} \frac{1}{x^{N}}+ \frac{\left\| g^{(N)}/G\right\|_{L^{p}}}{x^{N}} \left\| G \chi_{(-\lambda,\lambda)}\right\|_{L^{q}},
 \end{equation}
where $q$ is the conjugate index of $p$. 

Combining \eqref{eqrfrdif} with Lemma \ref{lemtaubarg2} concludes the proof of Theorem \ref{thrikimain} for the class $\mathcal{T}_{\mathrm{Dif}}$.
	
\subsection{H\"older continuous functions: the class $\mathcal{T}_{\mathrm{HC}}$}
	
In this section we assume $g \in \mathcal{T}_{\mathrm{HC}}(N,G,p,\omega,D,\delta_0)$, $\lambda \geq 1$ and $\hat{\phi} \in \mathcal{D}(-1/2, 1/2)$. The support of $\hat{\phi}$ is restricted further mainly for technical convenience. Note that this has as consequence that the admissible value for $A$ in Lemma \ref{lemanalyticcuts} is doubled.

We begin by discussing $\mathcal{T}_{\mathrm{HC}}$. Recall that a function $\mathfrak{g}(t): \R \rightarrow \C$ is called uniformly H\"older continuous with respect to $L^{p}$, $G: \R \rightarrow (0,\infty)$ and non-decreasing $\omega: (0,\delta_0] \rightarrow \R^{+}$ if 
\begin{equation} \label{eqrfruhc}
 \left\|\frac{\sup_{\left|u\right| \leq \delta}\left|\mathfrak{g}(t+u) - \mathfrak{g}(t)\right|}{G(t)}\right\|_{L^{p}(\R)} \ll \omega(\delta), \quad 0 < \delta \leq\delta_{0},
\end{equation}
where the $L^{p}$-norm is taken with respect to the variable $t$. In particular, if $\omega(\delta) = o(1)$ as $\delta \rightarrow 0+$, then $k \in \mathcal{T}_{\mathrm{HC}}(N,G,p,\omega,D,\delta_0)$ implies that $\mathfrak{g}$ is continuous.

Now, if $\omega(y_k) = o(y_k)$ for a sequence $y_k \rightarrow 0^{+}$ and if for each compact set $K \subset \R$, $G$ satisfies the mild restriction 
\begin{equation} \label{eqqikauxcondhcong} \sup_{t \in K} \sup_{|u| \leq \delta_1}  \left|\frac{G(t+ u)}{G(t)} \right| < \infty,
\end{equation}
for some $\delta_1 > 0$, then one can show that $\mathfrak{g}$ must be constant\footnote{The same comment applies to $\mathcal{T}_{\mathrm{HCI}}$. Then one does not even need to impose \eqref{eqqikauxcondhcong} on $G$.}. (One shows that the norm in the left-hand side of \eqref{eqrfruhc} with $j \delta_2 + \delta_3$ replacing $\delta$ and $L^{p}(K)$ replacing $L^{p}(\R)$ is $\ll_{K,G} j\omega(\delta_2) + \omega(\delta_3)$ for all $0 < \delta_{2}, \delta_3 \leq \delta_0,$ and all natural $j \ll \delta_2^{-1}$. Together with the non-decrease of $\omega$ and $\omega(y_k) = o(y_k)$ this implies that an admissible $\omega$ in \eqref{eqrfruhc} is $\omega(\delta) \equiv 0$ from which it follows that $\mathfrak{g}$ is constant.) Therefore, the assumption $\omega(y) \gg_\omega y$ as $y \rightarrow 0^{+}$ that was imposed in the definition of $\mathcal{T}_{\mathrm{HC}}$ is not too significant. 

We shall need an estimate for $\int^{\lambda}_{-\lambda} |\mathfrak{g}(t)| \mathrm{d}t$ if $\mathfrak{g}$ satisfies \eqref{eqrfruhc}. Let $0 < \delta \leq \delta_0$. We have
\begin{align*}
 \int^{\lambda}_{0} |\mathfrak{g}(t)| \mathrm{d}t & \leq |\mathfrak{g}(0)| \lambda  + \int^{\lambda}_{0} |\mathfrak{g}(t) - \mathfrak{g}(0)| \mathrm{d}t \\
 &  \leq |\mathfrak{g}(0)| \lambda + \sum_{j = 0}^{\lfloor \lambda/\delta\rfloor-1} \int^{(j+1) \delta}_{j\delta} |\mathfrak{g}(t) - \mathfrak{g}(0)| \mathrm{d}t + \int^{\lambda}_{\lfloor \lambda/\delta\rfloor \delta}  |\mathfrak{g}(t) - \mathfrak{g}(0)|\mathrm{d}t.
\end{align*}
The integral from $j\delta$ to $(j+1)\delta$ is estimated as
\begin{align*} \int^{(j+1) \delta}_{j\delta} |\mathfrak{g}(t) - \mathfrak{g}(0)| \mathrm{d}t & \leq \int^{(j+1)\delta}_{j\delta} \sum_{\ell = 0}^{j-1} |\mathfrak{g}(t - \ell\delta) - \mathfrak{g}(t-(\ell+1)\delta)| + |\mathfrak{g}(t-j\delta) - \mathfrak{g}(0)| \mathrm{d}t. \displaybreak[1] \\
& \leq \sum_{\ell = 0}^{j-1} \int^{(j+1-\ell)\delta}_{(j-\ell)\delta} |\mathfrak{g}(t) - \mathfrak{g}(t-\delta)|\mathrm{d}t + \int_{0}^{\delta} |\mathfrak{g}(t) - \mathfrak{g}(0)| \mathrm{d}t \displaybreak[1] \\ 
& \leq \int^{(j+1)\delta}_{0}\sup_{|u| \leq \delta}|\mathfrak{g}(t+u) - \mathfrak{g}(t)|\mathrm{d}t \ll \omega(\delta) \left(\int^{(j+1)\delta}_{0} G(t)^{q} \mathrm{d}t\right)^{1/q},
\end{align*}
via H\"older's inequality and \eqref{eqrfruhc}. The integral from $\lfloor \lambda/\delta\rfloor \delta$ to $\lambda$ is estimated analogously and we obtain for $\lambda > \delta$ that
\begin{equation} \label{eqrfruhcaux}
 \int^{\lambda}_{0} |\mathfrak{g}(t)| \mathrm{d}t \ll_{\omega,\delta} |\mathfrak{g}(0)| \lambda + \lambda \left\| G \chi_{(-\lambda,\lambda)}\right\|_{L^{q}}. 
\end{equation}

\medskip

Now we come to the estimation of $\langle g(t), e^{ixt} \hat{\phi}(-t/\lambda)\rangle$ if $g \in \mathcal{T}_{\mathrm{HC}}$ and $\hat{\phi} \in \mathcal{D}(-1/2,1/2)$. Integrating by parts $N$ times gives
\begin{equation*}
\left|\left\langle g(t), e^{ixt} \hat{\phi}(-t/\lambda)\right\rangle\right| \ll_{N} \frac{1}{x^{N}} \max_{0 \leq j \leq N} \left|\int^{\lambda/2}_{-\lambda/2} g^{(j)}(t) e^{ixt} \hat{\phi}^{(N-j)}(-t/\lambda) \lambda^{-N+j} \mathrm{d}t\right|.
\end{equation*}
We focus first on $j = N$. Due to the symmetry of $e^{ixt}$, we may write
\begin{align*}
\int^{\lambda/2}_{-\lambda/2} g^{(N)}(t) e^{ixt} \hat{\phi}(-t/\lambda) \mathrm{d}t 
& = -\int^{\lambda/2-\pi/x}_{-\lambda/2-\pi/x} g^{(N)}(t+\pi/x) e^{ixt} \hat{\phi}(-(t+\pi/x)\lambda^{-1})\mathrm{d}t. 
\end{align*}
Hence, if $x \geq \pi/\delta_{0}$ and $\lambda \geq 2\pi/x$ (which is implied by $\lambda \geq 1$ if $x \geq 2\pi$),
\begin{align*}
\left|\int^{\lambda/2}_{-\lambda/2} g^{(N)}(t) e^{ixt} \hat{\phi}(-t/\lambda) \mathrm{d}t\right| & \leq \frac{1}{2} \int^{\infty}_{-\infty} \left| g^{(N)}(t) \hat{\phi}(-t/\lambda) - g^{(N)}(t+\pi/x) \hat{\phi}(-(t+\pi/x)\lambda^{-1})\right|\mathrm{d}t\\
& \leq \frac{1}{2} \int^{\infty}_{-\infty} \left| g^{(N)}(t) - g^{(N)}(t+\pi/x)\right| \left|\hat{\phi}(-t/\lambda)\right| \mathrm{d}t \\
& \ \ + \frac{1}{2} \int^{\infty}_{-\infty} \left|g^{(N)}(t+\pi/x)\right| \left|\hat{\phi}(-t/\lambda)- \hat{\phi}(-(t+\pi/x)\lambda^{-1})\right| \mathrm{d}t \\
& \ll_{\phi}  \int^{\lambda/2}_{-\lambda/2}  \left| g^{(N)}(t) - g^{(N)}(t+\pi/x)\right| \mathrm{d}t + \frac{1}{x\lambda} \int^{\lambda}_{-\lambda} \left|g^{(N)}(t)\right| \mathrm{d}t\\
& \ll_{\omega,D,\delta_{0}} \left( \omega(\pi/x) + \frac{1}{x}\right)  \left\| G \chi_{(-\lambda,\lambda)}\right\|_{L^{q}} + \frac{1}{x},
\end{align*}
where in the last transition we used \eqref{eqrfruhc}, \eqref{eqrfruhcaux} and H\"older's inequality. 

For the other $j$, we imitate the above argument with $\hat{\phi}^{(N-j)}$ instead of $\hat{\phi}$ and $g^{(j)}$ instead of $g^{(N)}$ and use at the end the recursion relations 
  \begin{align*}                  
		\int^{\lambda}_{-\lambda} \left|g^{(j)}(t)\right|	\mathrm{d}t & \ll_{D} \lambda\left(1 + \int^{\lambda}_{-\lambda} \left|g^{(j+1)}(t)\right|	\mathrm{d}t\right),	\\
	\int^{\lambda/2}_{-\lambda/2} \left|g^{(j)}(t+\pi/x) - g^{(j)}(t)\right|	\mathrm{d}t	& \ll \frac{1}{x} \int^{\lambda}_{-\lambda} \left|g^{(j+1)}(t)\right|	\mathrm{d}t,		
	\end{align*}
	valid for $0 \leq j < N$ and all $\lambda \geq 2\pi/x$.

Collecting these estimates and taking into account that $\omega(\pi/x) \gg_\omega 1/x$, we obtain
\begin{equation} \label{eqrikhcres}
 \left\langle g(t), e^{ixt} \hat{\phi}(-t/\lambda)\right\rangle \ll_{N,\phi,\omega,D,\delta_{0}} \frac{\omega(\pi/x)}{x^{N}} \left( 1 +  \left\| G \chi_{(-\lambda,\lambda)}\right\|_{L^{q}} \right). 
\end{equation}
	
\subsection{Ultradifferentiable functions: the class $\mathcal{T}_{\mathrm{SA}}$} \label{secriknqu}

In this section we no longer restrict ourselves to a finite number of derivatives of the Laplace transform, but we consider all derivatives at once. We let $M_n$ be a logarithmically convex subanalytic sequence, that is, $M_n$ satisfies (M.1) and there exist $C', L > 0$ such that $M_{n} \geq C' L^{n} n^{n}$ for all natural $n$. We note that the constants $C$ and $A$ refer in this section to \eqref{eqanalyticcuts}. We suppose that $g \in \mathcal{T}_{\mathrm{SA}}(M_n,B,G,H,p)$, namely that $g$ is smooth and satisfies \eqref{equltradif} and \eqref{eqaux} for some locally integrable functions $G_n$. We let $\phi_n$ be a test function satisfying the properties of Lemma \ref{lemanalyticcuts}, for an $n$ that will be judiciously chosen (depending on $x$ and $\lambda$). 

We may suppose without loss of generality that $M_0 = 1$. Then the logarithmic convexity of $M_n$ implies that the sequence $M_j^{1/j}$ is non-decreasing. Integrating by parts $n$ times, applying H\"older's inequality, inserting the bounds \eqref{equltradif}, \eqref{eqaux} and \eqref{eqanalyticcuts}, and exploiting the subanalyticity of $M_n$, we find
\begin{align*}
 \left|\left\langle g(t), e^{ixt} \hat{\phi}_{n}(-t/\lambda)\right\rangle\right| & \ll_{C} {x^{-n}} \sum_{j = 0}^{n} {n \choose j} B^{j}M_{j}\frac{A^{n-j}n^{n-j}}{\lambda^{n-j}}  \left\| G_j \chi_{(-\lambda,\lambda)}\right\|_{L^{q}} \\
 & \ll_{C'} x^{-n} \sum_{j = 0}^{n} {n \choose j} B^{j}M_{n}^{j/n}(A/L\lambda)^{n-j}M_{n}^{(n-j)/n} G(\lambda)^{j}H(\lambda) \\
& \ll x^{-n} \left(B G(\lambda)+ \frac{A}{L\lambda}\right)^{n} M_{n} H(\lambda).
\end{align*}

Now we pick $n$, depending on $x$ and $\lambda$ (and $B$, the sequence $M_n$, $A, L$ and $G$) such that the final quantity is minimized. In conclusion, we obtain that for each $x \geq X$ and $\lambda \geq 1$, there exists $n$ such that
\begin{equation} \label{eqressuban}
 \left|\left\langle g(t), e^{ixt} \hat{\phi}_{n}(-t/\lambda)\right\rangle\right| \ll_{C,C'} \exp\left(-M\left(\frac{x}{BG(\lambda)+AL^{-1}\lambda^{-1}}\right)\right)H(\lambda),
\end{equation}
where $M$ is the associated function of the sequence $M_n$. We emphasize that the implicit constant in \eqref{eqressuban} is independent of $x$ and $\lambda$. Because of the properties of the sequence $\phi_n$, Lemma \ref{lemtaubarg2} may be applied. Observe crucially that all the implicit constants in the error terms remain independent of $x$ and $\lambda$, hence achieving \eqref{eqrikires} for $g \in \mathcal{T}_{\mathrm{SA}}$.
\begin{remark} \label{qikremnotherform} \ 
\begin{enumerate}
\item The assumption \eqref{eqaux} may at first appear a bit arbitrary, but these exponential bounds do often appear in practice, see e.g. the discussion on analytic functions in the next section. The main advantage of this expression is that it allows us in \eqref{eqressuban} to compute the infimum over $n$ in terms of the associated function $M$.

On the other hand, the same argument as above can be performed if in \eqref{eqaux} we replace $G(\lambda)^{n}H(\lambda)$ with the more general $H_n(\lambda)$, for a sequence of non-decreasing functions $H_n(\lambda)$, say this generalization corresponds to a class $\mathcal{T}_{\mathrm{SAA}}$ of assumptions for the Laplace transform. Then, our adjusted error term $E_{\mathrm{SAA}}$ in Theorem \ref{thrikimain} becomes
\begin{equation} \label{eqqiksaaerror}
E_{\mathrm{SAA}}(M_n,B, H_n,p,x,\lambda) \ll_{M_n} \inf_{n \in \N} x^{-n} M_n \sum_{j = 0}^{n} {n \choose j} \frac{B^{j}A^{n-j}}{L^{n-j}\lambda^{n-j}}  H_j(\lambda).
\end{equation}
In some instances $\mathcal{T}_{\mathrm{SAA}}$ indeed leads to a better error term in Theorem \ref{thrikimain}, but usually only when the error term in \eqref{eqrikires} decays rather fast. For example, in section \ref{qiksecanal}, the assertion \eqref{eqmkest} under the eventual non-decrease of $M(t)/\log^{\beta}t$ in Theorem \ref{coran} shall be established with this more general formulation.
\item If the sequence $M_n$ is non-quasianalytic, one does not require the subtle argument involving a sequence $\phi_n$ of test functions, but one can pick a single test function $\phi$ since the estimate \eqref{eqderivfour} can be used instead of \eqref{eqanalyticcuts}. 

In addition, if, for a non-quasianalytic $M_n$, the sequence $Q_n = M_n/n!$ is logarithmically convex and $G$ is bounded from below by a positive real number, say $\mathfrak{W}$, then one may replace $BG(\lambda)+A/L\lambda$ in \eqref{eqressuban} with\footnote{The implicit constant in \eqref{eqressuban} then depends also on $B$, $\mathfrak{W}$ and the sequence $M_n$.} $BG(\lambda)$. Namely, replacing the bounds $|\hat{\phi}_n^{(j)}(t)| \ll A^j n^j$ with $|\hat{\phi}^{(j)}(t)| \ll_{B, \mathfrak{W}, M_n} B^j \mathfrak{W}^{j} 2^{-j}M_j$ yields
\begin{align*}
 \left|\left\langle g(t), e^{ixt} \hat{\phi}(-t/\lambda)\right\rangle\right| & \ll_{B,\mathfrak{W}, M_n} {x^{-n}} \sum_{j = 0}^{n} {n \choose j} B^{j}M_{j}\frac{B^{n-j}M_{n-j}}{2^{n-j}\lambda^{n-j}} G(\lambda)^{j} \mathfrak{W}^{n-j} H(\lambda) \\
& \ll x^{-n} B^n M_n G(\lambda)^{n}H(\lambda) \sum_{j = 0}^{n} \frac{Q_j Q_{n-j}}{Q_n} 2^{j-n} \\ 
& \ll_{Q_{0}} x^{-n} B^n M_n G(\lambda)^{n}H(\lambda).
\end{align*}

This estimation breaks down for general subanalytic sequences $M_n$ as the sum $\sum_{j= 0}^n {n \choose j} M_j n^{n-j}/M_n$ is not uniformly bounded in $n$. The underlying reason for its failure is that $|\hat{\phi}_n^{(j)}(t)| \ll A^j j!$ does not hold for \emph{all} $j \leq n$, and must be false for some derivatives. Therefore, unless one wishes to invoke a growth requirement on $\lambda$ to control the derivatives of $\phi_n(t/\lambda)$, typically leading to hypotheses such as \eqref{eqqikcondstahn}, one must accept a small extra contribution.

In any case, as we will see in section \ref{qiksecanal}, not quite achieving $BG(\lambda)$ in \eqref{eqressuban} turns out to be irrelevant under very mild constraints on the functions $G$ and $H$.
\end{enumerate}
\end{remark}

\subsection{Analytic functions: the class $\mathcal{T}_{\mathrm{An}}$} \label{qiksecanal}

We now turn to the important case when the Laplace transform admits an analytic extension. We reduce its analysis to that of $\mathcal{T}_{\mathrm{SA}}$. We recall that $G : \R_+ \rightarrow \R_+$ is continuous and non-decreasing with $G(0) \neq 0$ and we set $\tilde{H}(t) := \sup_{|t-u| \leq 2 G(0)\inv} H(u)$.

Let $g \in \mathcal{T}_{\mathrm{An}}(G,H)$. It suffices to estimate the derivatives $g^{(n)}(t)$ for $t > 0$, the estimates for negative $t$ being analogous.  
Via Cauchy's theorem
\begin{equation} \label{eqcauchy}
\left|g^{(n)}(t)\right| = \frac{n!}{2\pi}  \left|\int_{\Gamma} \frac{g(z)}{(z-t)^{n+1}} \mathrm{d}z \right|,
\end{equation}
where $\Gamma$ is the square whose center $t$ lies at distance $1/G(t+G(0)\inv)$ from the sides and in which $g$ is analytic---the reason for taking a square instead of a circle is that we shave off an extra factor $\sqrt{n}$ below. As $|g(z)|$ is bounded by $\tilde{H}(t)$ on this square, we find
\begin{align*}
 \frac{n!}{2\pi}  \left|\int_{\Gamma} \frac{g(z)}{(z-t)^{n+1}}\mathrm{d}z\right| & \leq \frac{2}{\pi} n! \tilde{H}(t) \int^{\infty}_{-\infty} \frac{1}{((1/G(t+G(0)\inv))^{2} + w^{2})^{(n+1)/2}} \mathrm{d}w \\
& = \frac{2}{\pi}n! \tilde{H}(t) G(t+G(0)\inv)^{n} \int^{\infty}_{-\infty} \frac{\mathrm{d}w}{(1+w^{2})^{(n+1)/2}} \\
& \ll \frac{n!}{\sqrt{n}} \tilde{H}(t) G(t+G(0)\inv)^{n}, 
\end{align*}
since for natural $\ell$ (and hence the $O$-estimate also holds for real $\ell \geq 1$, say),
\begin{equation*}
 \int^{\infty}_{-\infty} \frac{\mathrm{d}w}{(1+w^{2})^{\ell}} = 2\pi i\operatorname*{res}_{w = i} (1+w^{2})^{-\ell} = \frac{\pi (2\ell-2)!}{(\ell-1)!^{2}2^{2\ell-2}} = O\left(\frac{1}{\sqrt{\ell}}\right),
\end{equation*}
via Stirling's formula. Therefore, by another application of Stirling's formula,
\begin{equation} \label{eqrealestimate}
 \left|g^{(n)}(t)\right| \ll n^{n}e^{-n}G(t+G(0)\inv)^{n}\tilde{H}(t),\ \ \ \text{for all } n \in \N.
\end{equation}

We now apply Theorem \ref{thrikimain} to $\mathcal{T}_{\mathrm{SA}}(n^n,e^{-1}, G(. + G(0)\inv), \int^{.}_{0}\tilde{H}, \infty)$ with $G_n(t) = G(t+G(0)\inv)^{n} \tilde{H}(t)$. We conclude that for each $x \geq X$ and $\lambda \geq 1$, there exists $\phi_{n} \in \mathcal{F}(\mathcal{D}(-1,1))$ (which may depend on $x$ and $\lambda$) satisfying the properties of Lemma \ref{lemtaubarg2} such that
\begin{equation} \label{eqestimateanalytic}
 \left|\left\langle g(t), e^{ixt} \hat{\phi}_{n}(-t/\lambda)\right\rangle\right|  \ll_{C}  \exp\left(-\frac{x}{G(\lambda+G(0)\inv) + Ae/\lambda }\right) \int^{\lambda}_{0} \tilde{H}(t) \mathrm{d}t.
\end{equation}
Here $C$ and $A$ are those occurring in Lemma \ref{lemanalyticcuts}.

We now simplify \eqref{eqestimateanalytic} by exploiting that $tH(t)$ is non-decreasing and by additionally requiring that $\hat{\phi}_n$ is supported in $(-1/2,1/2)$; this comes at the cost of doubling the value of $A$, see Lemma \ref{lemanalyticcuts}. Repeating the analysis of $\mathcal{T}_{\mathrm{SA}}$, one realizes that restricting the support of $\phi$ has the advantage one only needs to bound $\| G \chi_{(-\lambda/2,\lambda/2)}\|_{L^{1}}$ and this results that one may replace $G(\lambda + G(0)\inv)$ with $G(\lambda/2 + G(0)\inv)$ and $\int^{\lambda}_{0} \tilde{H}$ with $\int^{\lambda/2}_{0} \tilde{H}$ in \eqref{eqestimateanalytic}.

If $\lambda \geq 2G(0)\inv$, one has $G(\lambda/2 + G(0)\inv) \leq G(\lambda)$ since $G$ is non-decreasing. Furthermore, since $tH(t)$ is non-decreasing, one finds that $\tilde{H}(t) \ll H(2t)$ for $t \geq 3G(0)\inv$, say. This leads to $\int^{\lambda/2}_{0} \tilde{H}(t) \mathrm{d}t \ll \int^{\lambda}_{0} H(t)\mathrm{d}t + G(0)\inv\max_{t \leq 5G(0)\inv} H(t) \ll_{G(0),H} \int^{\lambda}_{0} H(t)\mathrm{d}t$.

Summarizing, for each $x \geq X$ and $\lambda \geq \max\{1,2G(0)\inv\}$, there exists $\phi_n$ (depending on $x$ and $\lambda$) satisfying the properties of Lemma \ref{lemtaubarg2} such that
\begin{equation} \label{eqestimateanalytic2}
 \left|\left\langle g(t), e^{ixt} \hat{\phi}_{n}(-t/\lambda)\right\rangle\right|  \ll_{C,G(0),H}  \exp\left(-\frac{x}{G(\lambda)(1+ 2AeG(\lambda)\inv\lambda^{-1})}\right) \int^{\lambda}_{0} H(t) \mathrm{d}t,
\end{equation}
with $C$ and $A$ as in Lemma \ref{lemanalyticcuts}. This concludes the proof of Theorem \ref{thrikimain} for all $\mathcal{T}_\ast$.
  
\pagebreak[2]
\begin{remark} \ 
\begin{enumerate} 
\item  In practice the restriction $\lambda \geq 2G(0)\inv$ for $E_{\mathrm{An}}(x,\lambda)$ is insignificant as the locations of $\lambda$ where the infimum in \eqref{eqrikires} is close to be attained, will usually tend to $\infty$ as $x \rightarrow \infty$.
\item A few technical optimizations are possible when either $G$ or $H$ do not satisfy the non-decreasing assumptions imposed in the definition of $\mathcal{T}_{\mathrm{An}}$. In case $tH(t)$ is not non-decreasing, one does still obtain \eqref{eqestimateanalytic}, and it may be advantageous to define $\tilde{H}$ more accurately as $\tilde{H}(t) = \sup_{|t - u| \leq \sqrt{2} d} H(u)$, where $d$ is the distance of $t$ to the sides of the square $\Gamma$. Another possible optimization is to select $\Gamma$ as the largest square inside the region of analytic continuation of $g$ having center $t$; this does not necessarily require that $G$ is non-decreasing or continuous. These mostly technical improvements usually only affect the implicit constant in \eqref{eqrikires} and we decided not to incorporate them in our main Theorem \ref{thrikimain}.
\item In many situations the factor $(1 + 2eAG(\lambda)\inv\lambda\inv)$ in the exponential of $E_{\mathrm{An}}$ is irrelevant, see e.g. the derivation of Theorem \ref{coran} later.
\item Sometimes one has also non-trivial information on some derivative of the analytic continuation of the Laplace transform or one can only establish estimates for some primitive of the Laplace transform. If $g^{(n_0)} \in \mathcal{T}_{\mathrm{An}}$ for some $n_0 \in \mathbb{Z}$, one achieves \eqref{eqrealestimate} where every $n$ on the right-hand-side is replaced with $n - n_0$. This means the optimal choice for $n$ is shifted by $n_0$. Ultimately one obtains again \eqref{eqestimateanalytic}, but now the right-hand side of these expressions has an extra factor $x^{-n_0}$. 
\end{enumerate}
\end{remark}

We now show Theorem \ref{coran}. It rephrases Theorem \ref{thrikimain} for $\mathcal{T}_{\mathrm{An}}$ with the classical Tauberian condition of $S(x) + Cx$ being non-decreasing. This allows us to compare it with \cite[Th. 2.1]{stahn} more easily. We recall the definitions of positive increase and $\eta$-regular growth.\par

A function $K: \R_+ \rightarrow \R_+$ is said to be of \emph{positive increase} if there are constants $a > 0$, $t_0$ such that for all $R \geq t_0$,
\begin{equation} \label{eqqikpi}
 t^{-a}K(t) \ll R^{-a}K(R), \ \ \ t_0 \leq t \leq R.
\end{equation}
A function of positive increase must therefore grow at least faster than $t^a$ for some $a > 0$ as $t \rightarrow \infty$.

Let $\eta: \R_{+} \rightarrow (0,\infty)$ be a positive non-decreasing function. A positive function $V: \R_+ \rightarrow \R_+$ is of \emph{$\eta$-regular growth} if there exist $C', t_{0} > 0$ such that
\begin{equation} \label{eqregmk}
 \frac{V(C't)}{V(t)} \geq 1 + \frac{1}{\eta(t)}, \ \ \ t \geq t_{0}.
\end{equation}

\begin{proof}[Proof of Theorem \ref{coran}] Let $t_S$ be a number such that $|\mathcal{L}\{S;s\}| \leq K(|s|)/s$ as soon as $|s| \geq t_S$ and $s\in \Omega_M$. The function $S$ satisfies the Tauberian condition $\mathfrak{T}(0,Cx,C,0)$ and the boundary distribution of the Laplace transform belongs to $\mathcal{T}_{\mathrm{An}}(M,H)$, where 
$ H(t) = K(t)/t$ when $t\geq t_S$ and $H(t) = \max_{|s| \leq t_S, s \in \Omega_M} |\mathcal{L}\{S;s\}|$ otherwise. Theorem \ref{thrikimain} thus yields that
\begin{equation} \label{qikeqresultsimple}
 |S(x)| \ll_{S,M,K} \inf_{\lambda \geq 1}\left\{\frac{1}{\lambda} + \exp\left(-\frac{x}{M(\lambda)(1+\eta(\lambda)^{-1})}\right)\int^{\lambda}_1 \frac{K(t)}{t} \dif t\right\}, 
\end{equation}
where $\eta(\lambda) = \lambda M(\lambda)/675$.

The first assertion \eqref{equsefulresult} now follows by selecting $\lambda = M_{K,\log}\inv(cx)$ in \eqref{eqrikires}. We estimate $\int^{\lambda}_{1} K(t)t^{-1}\mathrm{d}t \ll_{K} K(\lambda) \log \lambda $  and $1+\eta(\lambda)\inv \leq c^{-1}$ which is allowed since $\lambda \rightarrow \infty$ as $x \rightarrow \infty$.\par

If $M_{K,\log}$ has $\eta$-regular growth, we pick $\lambda$ to be a solution of the equation $x = (1+\eta(\lambda)\inv )M_{K,\log}(\lambda)$ in \eqref{qikeqresultsimple}. Then $S(x) \ll_{S,M,K} \lambda\inv \ll_{M,K} M_{K,\log}\inv(x)\inv$. The latter estimate is, due to the increasing nature of $M_{K,\log}$, equivalent to the existence of $C'$, such that $x \leq M_{K,\log}(C'\lambda)$ for sufficiently large $x$ and this follows from \eqref{eqregmk}.\par
If furthermore $K$ is of positive increase, we have $\int^{\lambda}_{1} K(t)t^{-1}\mathrm{d}t \ll_{K} K(\lambda)$ and \eqref{eqmkest} follows by selecting $\lambda$ to be a solution of the equation $x = (1+\eta/\lambda)M_{K}(\lambda)$ in \eqref{qikeqresultsimple} upon noticing that the $\lambda M(\lambda)/675$-regular growth of $M_{K,\log}$ is equivalent to that of $M_{K}$. 
\par
Alternatively, if $M(t)/\log^{\beta} t$ is non-decreasing for $t \geq t_{1} \geq 1$, say, we do not immediately apply Theorem \ref{thrikimain} with $g \in \mathcal{T}_{\mathrm{An}}$, but we slightly improve the available bounds for the derivatives of the Fourier transform and then appeal to $\mathcal{T}_{\mathrm{SAA}}$, see Remark \ref{qikremnotherform}.

We start from \eqref{eqrealestimate}, that is $|g^{(n)}(t)| \ll_{S} n^n e^{-n} M(t + M(0)\inv)^n  K(t+2M(0)\inv)/(t-2M(0)\inv) =: n^n e^{-n}G_n(t)$ for all natural $n$ and sufficiently large $t$, say $t \geq t_S + 3M(0)\inv$. If $t \leq t_S + 3M(0)\inv$ we have $|g^{(n)}(t)| \ll_{S,M,K} n^n e^{-n} M(t + M(0)\inv)^n =: n^n e^{-n} G_n(t)$. We can now improve \eqref{eqaux}; we only integrate on $(-\lambda/2,\lambda/2)$ because we chose the support of the sequence $\phi_n$ to belong to $[-1/2,1/2]$. We find for $j \leq n$ and sufficiently large $\lambda$ that
\begin{align*}
 \int_{-\lambda/2}^{\lambda/2}G_j(t) \mathrm{d}t & \ll_{M,K,t_1}  M(\lambda)^{j}K(\lambda) + \int^{\lambda}_{t_{1}} \frac{M(t)^{j}K(t)}{t}\mathrm{d}t \\
 & \ll M(\lambda)^{j}K(\lambda) + \frac{M(\lambda)^{j}K(\lambda)}{\log^{\beta j} \lambda} \int^{\lambda}_{1}  \frac{\log^{\beta j} t }{t} \mathrm{d}t\\
& = M(\lambda)^{j}K(\lambda) \left(1 + \frac{\log \lambda}{\beta j+1}\right) \\
& \ll M(\lambda)^{j}2^{n-j}K(\lambda) \left(1 + \frac{\log \lambda}{\beta n+1}\right).
\end{align*}
In the last transition we used that $2^{n-j} \geq (n+1)/(j+1) \geq (\beta n + 1)/(\beta j +1)$ as we may assume without loss of generality that $0 < \beta \leq 1$. 

In conclusion, via \eqref{eqqiksaaerror} we obtain for sufficiently large $x$ that
$$ S(x) \ll_{S,M,K} \inf_{\lambda \geq 1}\left\{\frac{1}{\lambda} + \inf_{n \in \N} \frac {n^n M(\lambda)^{n} \left(1 + 2\eta(\lambda)^{\inv}\right)^n }{e^n x^n }\left(1 + \frac{\log \lambda}{\beta n+1}\right) K(\lambda)\right\}.
$$
We thus have to certify that the term $\log \lambda/(\beta n + 1)$ does not drastically alter the infimum after optimizing $n$, what amounts to $\log \lambda \ll_{M,K} n$. As we select $n = x/(M(\lambda)(1+2\eta(\lambda)\inv)) + O(1)$, it is clear that $\log \lambda \ll_{M,K} n$ since $\lambda$ is chosen to be a solution of  $x = (1+2\eta(\lambda)\inv) M_{K}(\lambda)$. Since $\lambda \rightarrow \infty$ as $x \rightarrow \infty$, this ultimately leads to \eqref{eqmkest} after taking into account the $\eta/2$-regular growth of $M_K$ (which is equivalent to that of $M_{K,\log}$).

The impossibility of improving \eqref{equsefulresult} beyond the $M_K$-estimate will be shown in Section \ref{qiksecopti}.
\end{proof}

We end the discussion of $\mathcal{T}_{\mathrm{An}}$ with a few remarks.

\begin{remark} \label{remrfran} \:
\begin{enumerate}
\item We recall our convention in this work that the analytic continuation of $\mathcal{L}\{S;s\}$ to $\Omega_{M}$ indicates that it is analytic on the interior of $\Omega_M$ with a continuous extension to the boundary of $\Omega_{M}$. If $M$ is strictly increasing, it actually suffices for Theorems \ref{coran} and \ref{thrikimain} to have the analytic continuation (with according bounds) of $\mathcal{L}\{S;s\}$ only in the interior of $\Omega_M$.

\item In comparison with Stahn's Theorem \ref{thqikstahn}, we require the extra hypothesis that $M_{K,\log}$ is of regular growth to achieve the $M_{K,\log}\inv(x)$ and $M_{K}\inv(x)$-decay rates. However, this additional restriction is rather weak and is in practice always fulfilled. In fact, if $K(t) \ll \exp(CtM(t))$ for a certain $C > 0$, then $M_{K,\log}$ is automatically of $C'tM(t)$-regular growth for any $C' > 0$. Furthermore, even if $K$ does not admit this exponential bound, the function $K$ should grow very irregularly, by having sudden growth spurts followed by long periods of staying quasi-constant (where $M$ would have to stay quasi-constant as well), to force $M_{K,\log}$ not to be of $tM(t)$-regular growth. 
\item On the other hand Stahn required a two-sided Tauberian condition and the additional hypothesis \eqref{eqqikcondstahn}. In his paper \cite{stahn} he raised the question whether \eqref{eqqikcondstahn} was optimal in a certain sense. Theorem \ref{coran} settles this question; we do not need to impose any growth constraints between $M$ and $K$.
\item \label{qikremmkopt} Apart from omitting \eqref{eqqikcondstahn} in Theorem \ref{coran}, we also establish the error term $M_{K}^{-1}(x)^{-1}$ for a bigger class of functions $M$ and $K$ than Stahn; we allow $M(t)$ to have logarithmic growth. Even in the important case when $M = K$, this generates sometimes improved rates, such as for $M(t) = K(t) = \log(t+e)$. 

In general, whether the $M_{K,\log}$-estimate in Theorem \ref{coran} can be improved towards the $M_{K}$-estimate largely hinges on the estimation of
\begin{equation} \label{eqqikcrucint}
 \int^{\lambda}_{1} \frac{M(t)^{n} K(t)}{t} \mathrm{d}t. 
\end{equation}
The $M_{K}$-error term requires the estimate $M(\lambda)^{n}K(\lambda)$ for this integral, which is unclear if both $M$ and $K$ grow relatively slowly. For those $M$ and $K$ where the estimate $M(\lambda)^{n}K(\lambda)$ cannot be reached, one may still be able to get an improvement over the $M_{K,\log}$-error term (and thus Theorem \ref{coran}) if the trivial estimation $M(\lambda)^{n}K(\lambda) \log \lambda$ for the integral can be refined. Since the estimate $M(\lambda)^{n}K(\lambda)$ clearly cannot be reached for \eqref{eqqikcrucint} for all non-decreasing functions $M$ and $K$, it appears unlikely to us that the $M_{K}$-estimate in Theorem \ref{coran} should hold without imposing any additional conditions on $M$ and/or $K$. 
\item In his paper \cite{stahn} Stahn also considers the Tauberian condition when the $m$'th derivative of $S$ exists and is bounded. We can obtain a result similar to Theorem \ref{coran} for this Tauberian condition by combining the analysis of this section with Lemma \ref{lemtaubargm} applied to $F_m(x) = Cx^m/m!$; note that one can adapt the argument in section \ref{sqikexact} to ensure that $\int^{\infty}_{-\infty} S(x+y) \lambda \phi(\lambda y) \dif y = (2\pi)\inv \left\langle g(t), e^{ixt} \hat{\phi}(-t/\lambda)\right\rangle$ holds under this Tauberian condition.

The rate in \eqref{equsefulresult} then becomes $1/(M_{K,m,\log}^{-1}(cx))^m$ for all $c < 1$ where $M_{K,m,\log}\inv$ is the inverse function of $M_{K,m,\log}(t) = M(t) (\log K(t) + m \log t + \log \log t)$. Again $c$ can be picked equal to $1$ if $M_{K,m,\log}$ is of $tM(t)/(675 m)$-regular growth. The rate in \eqref{eqmkest} becomes $1/(M_{K,m}^{-1}(x))^m$ where $M_{K,m}(t) = M(t)(\log K(t) + m \log t)$ (where the assumption on the regular growth is now on $M_{K,m}$ and $675$ in the regular growth is replaced with $1350$ in case $M(t)/\log^{\beta}t$ is eventually non-decreasing). The implicit constants for all these estimates obviously now also depend on $m$. Naturally these conclusions also hold under the one-sided Tauberian condition of $S^{(m)}$ being bounded from below.

\item There do exist functions $S$ with bounded derivative whose Laplace transform admits an analytic continuation beyond $\Re s = 0$ that does not satisfy Stahn's restriction \eqref{eqqikcondstahn}. Concretely, in \cite{b-d-v-abswi}, see also \cite{d-v-abswi}, one constructs for an arbitrarily slowly decaying function $\varepsilon(x) \rightarrow 0$, a function $S$ with bounded derivative whose Laplace transform admits an entire analytic extension that satisfies the oscillation estimate $S(x) = \Omega(\varepsilon(x))$. Setting $M(t) \equiv 1$ and $\varepsilon(x)$ a function that is decaying slower to $0$ than $1/\sqrt{\log x}$ say, the function constructed in \cite{b-d-v-abswi} is clearly analytic in $\Omega_M$, but any acceptable growth bound $K$ for this function cannot satisfy \eqref{eqqikcondstahn} as Stahn's theorem would otherwise yield that $S(x) \ll_{S} 1/\log x$, a contradiction. 
\end{enumerate}
\end{remark}

\section{Optimality of Theorem \ref{thrikimain}} \label{qiksecopti}

The optimality of quantified versions of the Ingham-Karamata theorem under analytic hypotheses for the Laplace transform has already been intensively investigated in the literature. There are, to the knowledge of the author, two main approaches known to obtain the optimality for theorems related to Theorem \ref{coran}. The first one originated in \cite{b-t} and relied on a very delicate construction of a certain complex measure to show that the $M_{K}$-estimate is best possible when $M$ an $K$ are polynomial and given some restrictions on their degrees. This approach has been generalized in \cite{b-b-t} and then again significantly by Stahn \cite{stahn} who showed that the $M_{K}$-estimate is optimal for a reasonably large class of functions $M$ and $K$.\par
The second approach employs an attractive functional analysis technique based on the open mapping theorem. In the context of Theorem \ref{coran}, this technique was first exploited in \cite{DebruyneSeifert1} and was later improved in \cite{DebruyneSeifert2}. Actually, the latter paper contains, with respect to Theorem \ref{coran}, the most general optimality results available in the literature so far \cite[Th. 2.3]{DebruyneSeifert2}. It says that, if $M_{K}$ has at most exponential growth, then one cannot improve Theorem \ref{coran} beyond the $M_{K}$-estimate. 

Although the first constructive approach does not provide the optimality result under the most general constraints for $M$ and $K$, it has the advantage that it delivers an explicit example exhibiting the oscillation estimate $\Omega(1/M_{K}^{-1}(x))$ while satisfying the hypotheses of the Tauberian theorem. This feature is lacking in the functional analysis framework; there one shows there cannot exist a positive real-valued function $\varepsilon(x) \rightarrow 0$ such that the error $\varepsilon(x)/M_{K}^{-1}(x)$ would be acceptable for Theorem \ref{coran}. \par
We also mention that, in some particular cases for $M$ and $K$, it is possible to construct simple extremal oscillatory examples determining the optimality of Theorem \ref{coran}, see \cite{b-d-v-osc} when $M$ and $K$ are certain polynomials. Compared to the other two approaches, this method has the advantage that the examples are easy to describe and even witness the $\Omega(1/M_{K}^{-1}(x))$-estimate on the whole positive half-axis $(0,\infty)$ rather than only on a subsequence $x_k \rightarrow \infty$, but it has the downside of being rigid and it appears difficult to generalize.

In this Section we investigate whether the error terms of Theorem \ref{thrikimain} are \emph{best possible}. We shall use the functional analysis framework which appears to be the most flexible and delivers the most general results. Our main goal here is to develop the technology to incorporate a more general Tauberian condition and a wider variety of hypotheses on the Laplace transform. Yet, we also optimize the existing technique of \cite{DebruyneSeifert2}, see Corollary \ref{coroptian} below; this drastically weakens the constraints on $M$ and $K$ for when one is able to show the optimality of the $M_K$-estimate in Theorem \ref{coran}.

As we mentioned in the Introduction, the open mapping theorem dualizes the optimality question and involves the selection of test functions. Similar to the proof of the Tauberian Theorem \ref{thrikimain}, it is technically less complicated to construct test functions for boundary assumptions of the Laplace transforms of non-quasianalytic nature than for those of quasianalytic nature. In the non-quasianalytic setting, we may appeal to the Denjoy-Carleman theorem, while for hypotheses of more general subanalytic nature, we shall use the sequence of functions that was constructed in \cite[Lemma 2.1]{DebruyneSeifert2}.

We show our optimality result under the two-sided Tauberian condition $|S'(x)| \leq f(x)$, where $f$ satisfies certain mild regularity assumptions, roughly amounting to $f(x) \ll \exp(x^{\alpha})$ for some $\alpha < 1$. As a consequence, if $F$ possesses the relevant regularity, we actually show the optimality of a more restrictive Tauberian theorem than Theorem \ref{thrikimain}; the one-sided Tauberian condition of $S(x) + F(x)$ being non-decreasing from Theorem \ref{thrikimain} is weaker than $|S'(x)| \leq F'(x) =: f(x)$. Thus, we obtain a stronger optimality result. We also stress that we work with Laplace transforms and thus that the functions should have support on $[0,\infty)$; this necessitates a much more delicate argument, see section \ref{qiksecnegaxis}, than if no restrictions were put on the support.

\subsection{Statement of the optimality theorem} \label{secqikstate}

Before we state our optimality result, we first restrict the classes $\mathcal{T}_{\ast}$ slightly and we denote these new classes as $\widetilde{\mathcal{T}}_{\ast}$. Apart from $\widetilde{\mathcal{T}}_{\mathrm{SA}}$ and $\widetilde{\mathcal{T}}_{\mathrm{SAI}}$ which require a somewhat different, but still highly related, definition, the restrictions of the new classes are mostly for technical convenience only. In section \ref{qikremopti} we discuss in detail how the optimality Theorem \ref{thrfrsop} with $\widetilde{\mathcal{T}}_{\ast}$ relates to the Tauberian Theorem \ref{thrikimain} with $\mathcal{T}_{\ast}$. The classes $\widetilde{\mathcal{T}}_{\ast}$ consist of the same functions as $\mathcal{T}_{\ast}$, but with the following extra requirements.

\begin{itemize}
\item The sequence $M_n/n!$ is logarithmically convex.
\item For $\widetilde{\mathcal{T}}_{\mathrm{HC}}$ and $\widetilde{\mathcal{T}}_{\mathrm{HCI}}$, one has $\delta_0 < 1/3$ and the function $\omega$ is \emph{subadditive}, that is, $\omega(x +y) \leq \omega(x) + \omega(y)$ for all $x,y > 0$.
\item For the classes $\widetilde{\mathcal{T}}_{\mathrm{Dif}}$, $\widetilde{\mathcal{T}}_{\mathrm{DifI}}$, $\widetilde{\mathcal{T}}_{\mathrm{HC}}$ and $\widetilde{\mathcal{T}}_{\mathrm{HCI}}$, we no longer suppose that $|g^{(j)}(0)| \leq D$, for a given $D$, but only that some $D$ exists.
\item The new classes $\widetilde{\mathcal{T}}_{\mathrm{SANQ}}$ and $\widetilde{\mathcal{T}}_{\mathrm{SAINQ}}$ consist, except for the differences mentioned below, of the same elements as $\widetilde{\mathcal{T}}_{\mathrm{SA}}$ and $\widetilde{\mathcal{T}}_{\mathrm{SAI}}$ respectively, but now the sequence $M_n$ is non-quasianalytic.
\item For $\widetilde{\mathcal{T}}_{\mathrm{Dif}}$ and $\widetilde{\mathcal{T}}_{\mathrm{HC}}$, the function $G$ is even\footnote{This is to ensure that the definition of $\widetilde{\mathcal{T}}_{\ast}$ is symmetric with respect to the origin. Note that this is no restriction for our optimality theorem since for any real-valued function, the Laplace transform is always symmetric with respect to the real axis.}. The functions $G_n$ are even for $\widetilde{\mathcal{T}}_{\mathrm{SA}}$ and $\widetilde{\mathcal{T}}_{\mathrm{SANQ}}$.
\item Instead of \eqref{eqaux}, the functions $G_n$ in the definition of the class $\widetilde{\mathcal{T}}_{\mathrm{SANQ}}$ satisfy
\begin{equation} \label{eqrfrscogn}
  \left\|\frac{\chi_{(\lambda,2\lambda)}(t)}{G_{n}(t)} \right\|_{L^p} \ll \frac{\lambda}{G(\lambda)^{n}H(\lambda)}, \quad \lambda \rightarrow \infty.
\end{equation}
Recall that $\chi_{E}$ denotes the indicator function of a set $E$. Here $G \geq 1$. 
The sequence $G_n$ is incorporated in the class $\widetilde{\mathcal{T}}_{\mathrm{SANQ}} = \widetilde{\mathcal{T}}_{\mathrm{SANQ}}(M_n,B,G_n,G,H,p)$.
\item Instead of \eqref{eqaux}, the functions $G_n$ in the definition of the class $\widetilde{\mathcal{T}}_{\mathrm{SA}}$ satisfy, as $\lambda \rightarrow \infty$,
\begin{equation} \label{eqqiknewcondsa}
 \left\|\frac{\chi_{\{u:|u| \geq \lambda\}}(t)}{ t^2 G_{n}(t)} \right\|_{L^p} \ll \frac{1}{ G(\lambda)^{n}H(\lambda)}, \ \ \ \left\|\frac{\chi_{(-\lambda,\lambda)}(t) }{(1+ |t|) G_{n}(t)} \right\|_{L^{p}} \ll \tilde{B}^{n} \lambda \frac{M_n}{n!},
 \end{equation}
 \begin{equation*}
 \left\|\frac{\chi_{\{u:|u| \geq \lambda\}}(t)}{(1+|t|/\lambda)^4 G_{n}(t)} \right\|_{L^p} \ll \frac{\lambda}{G(\lambda)^n H(\lambda)},
\end{equation*}
for some $\tilde{B} > 0$. We take $\tilde{B}$ and the sequence $G_n$ into account as parameters for the class $\widetilde{\mathcal{T}}_{\mathrm{SA}} = \widetilde{\mathcal{T}}_{\mathrm{SA}}(M_n, B,\tilde{B}, G_n, G, H, p)$.
\item The order of the derivative $N$ in $\widetilde{\mathcal{T}}_{\mathrm{Dif}}$ and $\widetilde{\mathcal{T}}_{\mathrm{DifI}}$ should be different from $0$. Furthermore the estimates \eqref{equltradif} and \eqref{eqqiksai} on the derivatives in the classes $\widetilde{\mathcal{T}}_{\mathrm{SA}}$, $\widetilde{\mathcal{T}}_{\mathrm{SANQ}}$, $\widetilde{\mathcal{T}}_{\mathrm{SAI}}$ and $\widetilde{\mathcal{T}}_{\mathrm{SAINQ}}$ need only to hold from $n \geq 1$ onward instead of for $n \geq 0$.
\item For the classes $\widetilde{\mathcal{T}}_{\mathrm{Dif}}$ and $\widetilde{\mathcal{T}}_{\mathrm{HC}}$ one has $\left\| (1+|t|)^{-2}/G(t)\right\|_{L^{p}} \ll 1$, while for the classes $\widetilde{\mathcal{T}}_{\mathrm{SA}}$ and $\widetilde{\mathcal{T}}_{\mathrm{SANQ}}$ there holds $\left\| (1+|t|)^{-2}/G_n(t)\right\|_{L^{p}} \ll 1$.
\item  For $\widetilde{\mathcal{T}}_{\mathrm{SA}}$ and $\widetilde{\mathcal{T}}_{\mathrm{SANQ}}$ we assume now also that $G$ is non-decreasing and that $1/H$ is locally bounded. For $\widetilde{\mathcal{T}}_{\mathrm{SANQ}}$ we suppose additionally that $H(t)$ is bounded from below near $t= 0$. For $\widetilde{\mathcal{T}}_{\mathrm{SAINQ}}$ and $\widetilde{\mathcal{T}}_{\mathrm{SAI}}$ we suppose that $G$ is bounded from below, say $G \geq 1$.
\end{itemize}

After these introductory comments we come to the main theorem of this Section.

\begin{theorem} \label{thrfrsop} Let $f: \mathbb{R}_{+} \rightarrow \mathbb{R}_{+}$ be a differentiable non-decreasing function such that $|f'(x)/f(x)| \ll x^{-\varepsilon}$ for some $\varepsilon > 0$, $\log f(x) \ll x^{1-\varepsilon}$, $f \gg 1$ as $x \rightarrow \infty$. 
If $W: \mathbb{R}_{+} \rightarrow \mathbb{R}_{+}$ is such that each real-valued continuously differentiable $S: (0,\infty) \rightarrow \mathbb{R}$ with $|S'(x)| \leq f(x)$ and infinitely differentiable Fourier transform $\hat{S} \in \widetilde{\mathcal{T}}_{\ast}$, obeys the bound $|S(x)| \ll_{S,W} 1/W(x)$, then
\begin{equation} \label{eqresultopti}
 W(x) \ll_{f,\widetilde{\mathcal{T}}_{\ast},W} \inf_{\lambda \geq 1} \left(\frac{\lambda}{f(x)} + \widetilde{E}_{\ast}(x,\lambda)\right), \ \ \  x \rightarrow \infty,
\end{equation}
where $\widetilde{E}_{\ast}(x,\lambda)$ is given below.
\end{theorem}

The Fourier transform is interpreted here initially as the boundary (ultra)distribution of the Laplace transform of $S$, that is $\left\langle\mathcal{L}\{S;\sigma + it\}, \phi(t)\right\rangle \rightarrow \left\langle \hat{S}(t), \phi(t)\right\rangle$ as $\sigma \rightarrow 0^+$ for all compactly supported functions $\phi$ whose Fourier transform $\hat{\phi}(x)$ obeys $\hat{\phi}(x) \ll_{\phi,\varepsilon} \exp(-|x|^{1-\varepsilon/2})$, say, with $\varepsilon$ given by Theorem \ref{thrfrsop}. Yet, as the boundary (ultra)distribution corresponds to a continuous function, a standard application (of an ultradistributional version) of the edge-of-the-wedge theorem involving the Poisson transform, delivers that the Fourier transform is in fact a continuous extension of the Laplace transform to the imaginary axis.

Evidently, the above theorem implies that for each positive real-valued function $\varepsilon: \R_+ \rightarrow (0,\infty)$ tending to $0$ as $x \rightarrow \infty$, there exists a function $S = S(\varepsilon, f, \widetilde{\mathcal{T}}_{\ast}) : (0,\infty) \rightarrow \mathbb{R}$ with derivative $|S'(x)| \leq f(x)$ and whose Fourier transform, interpreted appropriately, belongs to $\widetilde{\mathcal{T}}_{\ast}$, but for which 
$$ S(x) \gg_{\varepsilon,f,\widetilde{\mathcal{T}}_{\ast}} \varepsilon(x) \inf_{\lambda \geq 1} \left(\frac{\lambda}{f(x)} + \widetilde{E}_{\ast}(x,\lambda)\right)^{-1}, \ \ \ \text{infinitely often as } x \rightarrow \infty.
$$
In section \ref{qikremopti}, we shall sometimes refer to such functions $S$ as \emph{counterexamples}.

The functions $\widetilde{E}_{\ast}(x,\lambda)$ are given by
\begin{itemize}
\item $\widetilde{E}_{\mathrm{Dif}}(N,G, p,x,\lambda) = x^{N}\lambda^{-1} \left\| \chi_{(\lambda, 2\lambda)}/G \right\|_{L^{p}}$,
\item $\widetilde{E}_{\mathrm{DifI}}(N, G, x,\lambda) = x^{N}/G(\lambda)$,
\item $\widetilde{E}_{\mathrm{HC}}(N, G, p, \omega, \delta_0, x,\lambda) = x^{N}\lambda^{-1} \left\| \chi_{(\lambda, 2\lambda)}/G \right\|_{L^{p}}/ \omega(1/x)$,
\item $\widetilde{E}_{\mathrm{HCI}}(N, G, \omega, \delta_0, x,\lambda) = x^{N}/(\omega(1/x)G(\lambda))$,
\item $\widetilde{E}_{\mathrm{SANQ}}(M_n, B, G_n, G, H, p, x,\lambda) = \exp\left(M\left(\frac{x}{BG(\lambda)}\right)\right)/H(\lambda)$, 
\item $\widetilde{E}_{\mathrm{SA}}(M_n, B,\tilde{B}, G_n, G, H, p, c, x,\lambda) = \exp\left(M\left(\frac{x}{B G(\lambda)}\right) + Q\left(\frac{c \log x}{B \lambda  G(\lambda)}\right) \right)/H(\lambda)$, \linebreak[4] where $Q$ is the associated function of the sequence $M_n/n!$. The estimate \eqref{eqresultopti} holds for each $c > 0$, but naturally the implicit constant depends on $c$.
\item $\widetilde{E}_{\mathrm{SAINQ}}(M_n,B, G, H, x,\lambda) = \exp\left(M\left(\frac{x}{B G(\lambda)}\right)\right)/H(\lambda)$, 
\item $\widetilde{E}_{\mathrm{SAI}}(M_n,B, G, H, c, x,\lambda) = \exp\left(M\left(\frac{x}{B G(\lambda)}\right) + Q\left(\frac{c \log x}{B \lambda G(\lambda)}\right) \right)/H(\lambda)$, where $Q$ is the associated function of the sequence $M_n/n!$. Again the estimate \eqref{eqresultopti} holds for each $c > 0$, but the implicit constant depends on $c$.
\item $\widetilde{E}_{\mathrm{An}}(G, H,c, x,\lambda) = \exp\left(\frac{x}{G(\lambda)}\right)/(\lambda H(\lambda))$ if $c \log x \leq G(\lambda) \lambda$ and $\infty$ otherwise. Again any $c> 0$ in \eqref{eqresultopti} is allowed but the implicit constant depends on $c$.
\end{itemize}

The rest of this Section is dedicated to the proof of Theorem \ref{thrfrsop}.

\subsection{Proof of Theorem \ref{thrfrsop}: The duality argument}

First we argue that we may actually take any \emph{complex-valued} $S$ in the hypothesis of the statement of Theorem \ref{thrfrsop}. Indeed, if a complex-valued $S$ has a continuous derivative $|S'(x)| \leq f(x)$ and its Fourier transform is infinitely differentiable and belongs to $\widetilde{\mathcal{T}}_{\ast}$, then these properties also hold for $\Re S$ and  $\Im S$ since $\widehat{\Re S}(t) = (\hat{S}(t) + \overline{\hat{S}(-t)})/2$ and $\widehat{\Im S}(t) = (\hat{S}(t) - \overline{\hat{S}(-t)})/2i$ and the definition of $\widetilde{\mathcal{T}}_\ast$ is \emph{symmetric} with respect to the origin. Thus, applying the ``real-valued" hypothesis of Theorem \ref{thrfrsop}, one has $|S(x)| \leq |\Re S(x)| + |\Im S(x)| \ll_{S,W} 1/W(x)$ and the ``complex-valued" hypothesis therefore follows from the ``real-valued" one. 

We start by collecting all the functions that satisfy the hypotheses of the Tauberian theorem in a Fr\'echet space $V_{1}$, that is, all complex-valued continuously differentiable functions $S: (0,\infty) \rightarrow \C$ such that $|S'(x)| \ll_{S} f(x)$ and whose Fourier transform, initially interpreted as ultradistributional boundary values of the Laplace transform, is an infinitely differentiable function that belongs to the class $\widetilde{\mathcal{T}}_{\ast}$. The space $V_{1}$ is topologized via the countable family of seminorms
\begin{equation} \label{eqrfrsnorm}
 \left\|S\right\|_{1,\nu,m} =  \sup_{y > 0} \frac{|S'(y)|}{f(y)} +  \left\| \hat{S}\right\|_{\widetilde{\mathcal{T}}_{\ast}} + \sup_{\substack{ t \in [-\nu,\nu]\\ j \leq m}} |\hat{S}^{(j)}(t)|, \ \ \ \nu,m \in \mathbb{N},
\end{equation}
where $ \left\| g\right\|_{\widetilde{\mathcal{T}}_{\ast}}$ denotes the norm associated with $\widetilde{\mathcal{T}}_{\ast}$, that is,
\begin{itemize}
\item $ \left\| g\right\|_{\widetilde{\mathcal{T}}_{\mathrm{Dif}}} =  \left\|g^{(N)}/G\right\|_{L^{p}}$,
\item $ \left\| g\right\|_{\widetilde{\mathcal{T}}_{\mathrm{DifI}}} = \displaystyle{\sup_{R > 0} \frac{\int^{R}_{-R} |g^{(N)}(t)| \mathrm{d}t}{G(R)}}$,
\item  $ \left\| g\right\|_{\widetilde{\mathcal{T}}_{\mathrm{HC}}} = \displaystyle{\sup_{0 < \delta \leq \delta_0} \frac{1}{\omega(\delta)}\left\|\frac{\sup_{|u| \leq \delta}|g^{(N)}(t+u) - g^{(N)}(t)|}{G(t)}\right\|_{L^{p}} } $,
\item $ \left\| g\right\|_{\widetilde{\mathcal{T}}_{\mathrm{HCI}}} = \displaystyle{\sup_{R > 0} \sup_{0 < \delta \leq\delta_0} \frac{\int^{R}_{-R} \sup_{|u| \leq \delta}|g^{(N)}(t+u) - g^{(N)}(t) | \mathrm{d}t}{G(R) \omega(\delta)}}$,
\item $ \left\| g\right\|_{\widetilde{\mathcal{T}}_{\mathrm{SANQ}}} =  \left\| g\right\|_{\widetilde{\mathcal{T}}_{\mathrm{SA}}} = \displaystyle{\sup_{\substack{n \in \mathbb{N}\\ n \geq 1}} \left\|\frac{g^{(n)}(t)}{B^n M_n G_{n}(t)}\right\|_{L^{p}}}$,
\item $ \left\| g\right\|_{\widetilde{\mathcal{T}}_{\mathrm{SAINQ}}} =  \left\| g\right\|_{\widetilde{\mathcal{T}}_{\mathrm{SAI}}} = \displaystyle{\sup_{R > 0} \sup_{\substack{n \in \mathbb{N}\\ n \geq 1}} \frac{\int^{R}_{-R} |g^{(n)}(t)| \mathrm{d}t }{B^n M_{n} G(R)^{n}H(R)}}$,
\item $ \left\| g\right\|_{\widetilde{\mathcal{T}}_{\mathrm{An}}} = \displaystyle{ \sup_{|\Im z| \leq  1/G(|\Re z|)} \frac{|g(z)|}{H(|z|)}}$.
\end{itemize}
It is a standard verification that $V_1$ is complete for each $\widetilde{\mathcal{T}}_{\ast}$ and thus $V_1$ is in fact Fr\'echet. (Recall that our convention on analyticity on closed regions in this work asserts that a function $g \in \widetilde{\mathcal{T}}_{\mathrm{An}}$ only needs to be analytic on the interior of $\{z: |\Im z| \leq  1/G(|\Re z|)\}$ with a continuous extension to the boundary of this region.)

We let $V_{2}$ be the Fr\'echet space that consists of the functions $S \in V_1$ additionally satisfying the estimate $S(x) \ll_S 1/W(x)$; it is topologized via the countable family of seminorms
$$ \left\|S\right\|_{2,\nu,m} = \left\|S\right\|_{1,\nu,m} + \sup_{y > 0} |S(y)|W(y).
$$

Now, by the hypotheses of the Theorem \ref{thrfrsop}---$1/W$ is an admissible rate for the Tauberian theorem---both Fr\'echet spaces $V_{1}$ and $V_{2}$ consist of the same elements. Therefore, the open mapping theorem implies that the canonical injection $\iota : V_2 \rightarrow V_1$ is an isomorphism. Hence, there exists $\nu_0,m_0 \in \mathbb{N}$ such that
\begin{equation} \label{eqqikopenmap} \sup_{y > 0} |S(y)|W(y)  \ll_{f,\widetilde{\mathcal{T}}_{\ast},W} \left\|S\right\|_{1,\nu_0,m_0} \ \ \ \text{for all }S \in V_{1}.
\end{equation}

The next step is to select a suitable $S$.

\subsection{Proof of Theorem \ref{thrfrsop}: The test functions} \label{qiksectestfun}

We shall consider two distinct families of functions for $S$. One is designed to handle the boundary hypotheses $\widetilde{\mathcal{T}}_{\ast}$ of non-quasianalytic nature for the Laplace transform, that is $\widetilde{\mathcal{T}}_{\mathrm{Dif}}$, $\widetilde{\mathcal{T}}_{\mathrm{DifI}}$, $\widetilde{\mathcal{T}}_{\mathrm{HC}}$, $\widetilde{\mathcal{T}}_{\mathrm{HCI}}$, $\widetilde{\mathcal{T}}_{\mathrm{SANQ}}$ and $\widetilde{\mathcal{T}}_{\mathrm{SANQI}}$, and the other more complicated family handles the boundary hypotheses of more general subanalytic nature, that is $\widetilde{\mathcal{T}}_{\mathrm{SA}}$, $\widetilde{\mathcal{T}}_{\mathrm{SAI}}$ and $\widetilde{\mathcal{T}}_{\mathrm{An}}$. In the sequel we shall distinguish between these two cases.

\textbf{Non-quasianalytic $\widetilde{\mathcal{T}}_{\ast}$:} We select $S$ as 
$$S_{x,\lambda}(y) = h(\lambda(x-y)) \chi_{(0,\infty)}(y),$$
 where $h \in \mathcal{F}(\mathcal{D}(1,2))$ is a function such that $h(0) = 1$, $\max_{y \in \mathbb{R}} |h(y)| = 1$, $|h^{(j)}(y)| \ll_{\varepsilon,B,M_n} \exp(-\overline{M}(|y|))$ for $j = 0,1,2$ and $|\hat{h}^{(n)}(t)| \ll_{\varepsilon, B, M_n} \tilde{M}_n$. Here $\tilde{M}_n$ and $\overline{M}_{n}$ are two logarithmically convex, non-quasianalytic sequences and are constructed in such a fashion that they obey the properties described below. The function $\overline{M}(y)$ is the associated function of the sequence $\overline{M}_n$ and the existence of $h$ follows from the Denjoy-Carleman theorem.
 
First, the sequence $\tilde{M}_n$ is selected to be a logarithmically convex, non-quasianalytic minorant of the sequences\footnote{For $\widetilde{\mathcal{T}}_{\mathrm{Dif}}$, $\widetilde{\mathcal{T}}_{\mathrm{DifI}}$, $\widetilde{\mathcal{T}}_{\mathrm{HC}}$ or $\widetilde{\mathcal{T}}_{\mathrm{HCI}}$, one may simply pretend that $B =1$ and $M_{n} = n^{n/2}$, say.}  $2^{-n}B^{n}M_{n-2}$ (for $n \geq 2$), $B^n M_n$ and $(ne\inv\gamma\inv)^{n/\gamma}$ where $1/2 \leq \gamma < 1$ is such that $1-\gamma < \varepsilon$, say $\gamma = 1-\varepsilon/2$. Its associated function satisfies $\tilde{M}(y) \geq y^{\gamma} + O_{\gamma}(1)$, but also $\int^{\infty}_1 \tilde{M}(y)y^{-2} \dif y < \infty$, since $\tilde{M}_n$ is non-quasianalytic. 

The second sequence $\overline{M}_n$ is defined as $\overline{M}_n = \min_{0\leq j \leq n}\{\tilde{M}_j \tilde{M}_{n-j}\}$. One may verify that that its associated function is $\overline{M}(y) = 2\tilde{M}(y)$ and that the sequence is logarithmically convex, see e.g. \cite[Lemma 3.5]{komatsu}. As $\int^{\infty}_1 \overline{M}(y)y^{-2} \dif y = 2 \int^{\infty}_1 \tilde{M}(y)y^{-2} \dif y < \infty$, the sequence $\overline{M}_n$ is non-quasianalytic.
 
We will show later that $S_{x,\lambda} \in V_1$. From the above properties one can then deduce \eqref{eqrfrsm} below where $h_k$ and $S_{k,x,\lambda}$ are replaced with $h$ and $S_{x,\lambda}$ respectively. 

\textbf{Subanalytic $\widetilde{\mathcal{T}}_{\ast}$:} Obviously $\hat{S}_{x,\lambda}$ cannot belong to $\widetilde{\mathcal{T}}_{\mathrm{SA}}$, $\widetilde{\mathcal{T}}_{\mathrm{SAI}}$ or $\widetilde{\mathcal{T}}_{\mathrm{An}}$ if $M_n$ is quasianalytic. We therefore have to relinquish the compact support of $\hat{h}$ and replace it with sufficiently fast decay at $0$. This involves the introduction of an extra parameter $k$. We appeal to the ingenious construction from \cite[Lemma 2.1]{DebruyneSeifert2}; there one constructs a sequence of complex-valued differentiable functions $h_{k} \in L^{1}(\mathbb{R})$ (denoted $f_k$ in \cite{DebruyneSeifert2} and depending in principle also on the positive parameters $c$ and $\rho$) such that
\begin{enumerate}
 \item \label{qikpropseqlow} $\inf_{k\in \mathbb{N}} h_{k}(0) \geq  1$,
 \item \label{qikpropseql1} the functions $h_{k}$ and $h_{k}'$ are continuous and uniformly bounded in $L^1$, that is, $\| h_{k}\|_{L^{1}(\mathbb{R})} \ll 1$ and $\| h_{k}'\|_{L^{1}(\mathbb{R})} \ll 1$,
 \item \label{qikpropseqlinf} the sequence $h_{k}$ and its derivatives are uniformly bounded in $L^{\infty}$, that is, $\| h_{k}\|_{L^{\infty}(\mathbb{R})} \ll 1$ and $\| h'_{k}\|_{L^{\infty}(\mathbb{R})} \ll_{c,\rho} 1$,
 \item \label{qikpropseqdeczero}the Fourier transforms $\hat{h}_{k}(z)$ possess an analytic continuation to the ball $\{z : |z| \leq \rho\}$ where they satisfy $\hat{h}_{k}(z) \ll_{c,\rho} |z|^k (2\rho)^{-k}$,
 \item \label{qikpropseqana} the Fourier transforms $\hat{h}_{k}(z)$ can be analytically continued to the strip $\Lambda_{k,c} := \{z : |\Im z| \leq c^{-1}/\log (k+1)\}$ and are uniformly bounded there; more specifically 
 \begin{equation} \label{eqqikbostrip}
  \sup_{k \in \mathbb{N}} \sup_{z \in \Lambda_{k,c}} (1 + |z|^4) |\hat{h}_{k}(z)| \ll_{c} 1.
\end{equation}
\end{enumerate}

The weight $1/(1+ |z|^4)$ for $\hat{h}_{k}$ in \eqref{eqqikbostrip} is not explicitly mentioned in the statement of \cite[Lemma 2.1]{DebruyneSeifert2}, but is evident from the construction. The family of functions $S$ to be considered in \eqref{eqqikopenmap} is
$$S_{k,x,\lambda}(y) = h_{k}(\lambda(x-y)) \chi_{(0,\infty)}(y).$$

The functions $S_{x,\lambda}$ and $S_{k,x,\lambda}$ indeed belong to $V_1$; in the next sections we verify this by estimating the relevant norms by sometimes imposing requirements on the involved parameters; it is also clear that  $S_{x,\lambda}'$ and $S_{k,x,\lambda}'$ are continuous on $(0,\infty)$ and from the subsequent sections we will also derive that $\hat{S}_{x,\lambda}$ and $\hat{S}_{k,x,\lambda}$ are infinitely differentiable. Therefore, by applying \eqref{eqqikopenmap} and exploiting property (\ref{qikpropseqlow}), one obtains, for $x >0$,
\begin{equation} \label{eqrfrsm} 
 W(x) \leq \sup_{y > 0} |h_{k}(\lambda(x-y))| W(y) \ll_{f,\widetilde{\mathcal{T}}_{\ast},W}  \sup_{y > 0} \frac{|S_{k,  x,\lambda}'(y)|}{f(y)} +  \left\| \hat{S}_{k,  x,\lambda}\right\|_{\widetilde{\mathcal{T}}_{\ast}} + \sup_{\substack{|t| \leq \nu_0 \\ n \leq m_0}} \left|\hat{S}_{k,x,\lambda}^{(n)}(t)\right|.
\end{equation}
The rest of the proof is devoted to estimating these norms. 

\subsection{Proof of Theorem \ref{thrfrsop}: Analysis of the norm of the Tauberian condition}

We shall show that the relevant norms are $\ll \lambda/f(x) + \tilde{E}(x,\lambda)$. During the proof we shall sometimes assume that $\lambda \geq C$, for some constant $C$ depending only on $f, W$ and $\widetilde{\mathcal{T}}_{\ast}$, but this is allowed since the hypotheses imposed on these classes imply that $\inf_{1\leq \lambda \leq C} \widetilde{E}_{\ast}(x,\lambda) \gg_{C,\widetilde{\mathcal{T}}_{\ast}} \widetilde{E}_{\ast}(x,C)$. 

We start with by estimating the norm that encodes the Tauberian condition. This shall be $\ll \lambda/f(x)$. We begin with the family $S_{x,\lambda}$ for the non-quasianalytic classes $\widetilde{\mathcal{T}}_{\ast}$.

\textbf{Non-quasianalytic $\widetilde{\mathcal{T}}_{\ast}$:}
We now analyze the weighted $L^{\infty}$-norm in \eqref{eqrfrsm} of the family $S_{x,\lambda}$. If the supremum is restricted to $y \geq x$, the desired estimate $O(\lambda/f(x))$ clearly follows as $f$ is non-decreasing and $h' \in L^{\infty}(\mathbb{R})$. 

Let $C$ be such that $|f'(x)/f(x)| \leq Cx^{-\varepsilon}$. In the range $C'x := C^{1/\varepsilon}x/(1+C^{1/\varepsilon}) \leq y \leq x$, we find, since $\lambda > 1$,  
\begin{align*}
 \frac{|h'(\lambda(x-y))| }{f(y)} & \ll_{\varepsilon} \exp\left(-|x-y|^{\gamma} - \log f(x) + |x-y| \sup_{C'x \leq u \leq x} \left|\frac{f'(u)}{f(u)}\right| \right)\\
 & \ll f(x)^{-1} \exp(-|x-y|^{\gamma} + (1+C^{1/\varepsilon})^{\varepsilon}|x-y|x^{-\varepsilon})\\
& \ll f(x)^{-1} \exp(-|x-y|^{\gamma} + |x-y|^{1-\varepsilon}) \ll f(x)^{-1},
\end{align*}     
by the mean-value theorem applied to $\log f$. Finally, as $f \gg 1$ we find for $0 < y <  C'x$ that
\begin{align*}
 \frac{|h'(\lambda(x-y))| }{f(y)}  & \ll_{\varepsilon} f(x)^{-1}\exp\left(-(1-C')^{\gamma}x^{\gamma} + \log f(x)\right) \\
& \ll f(x)^{-1}\exp\left(-(1-C')^{\gamma}x^{\gamma} + C'' x^{1-\varepsilon} \right) \ll_{C',C'',\varepsilon} f(x)^{-1}, 
\end{align*}
after choosing $C''$ such that $\log f(x) \leq C'' x^{1-\varepsilon}$.

\textbf{Subanalytic $\widetilde{\mathcal{T}}_{\ast}$:} The argument for $S_{k,x\lambda}$ is similar to the one for the family $S_{x,\lambda}$ and we only briefly discuss the changes. The main difference is that we now use the estimate $|h'_{k}(\lambda(x-y))| \ll_{c,\rho} \exp(-c^{-1} |x-y|/\log (k+1))$ which follows from (\ref{qikpropseqana}) after moving the contour in the formula $h_{k}'(y) = (2\pi)^{-1} \int^{\infty}_{-\infty} it e^{iyt} \hat{h}_{k}(t) \dif t$. Working with this alternative bound generates an additional restriction. In the range $C' x \leq y \leq x$, the desired estimate now follows from $\exp(-c^{-1}|x-y|/\log(k+1) + (1 + C^{1/\varepsilon})^{\varepsilon} |x-y| x^{-\varepsilon}) \ll 1 $, which is only valid if one accepts 
\begin{equation} \label{eqqikextracondk} x \geq C_{c,\varepsilon} (\log (k+1))^{1/\varepsilon},
\end{equation}
for a sufficiently large constant $C_{c,\varepsilon}$, depending only on $c$ and $\varepsilon$ and $C$. The same restriction \eqref{eqqikextracondk} (with $C_{c,\varepsilon}$ additionally depending on $C''$) is also necessary in the estimation of the range $0 < y < C'x$. We shall later verify that our eventual choice for $k$ indeed fulfills \eqref{eqqikextracondk}. 

\subsection{Proof of Theorem  \ref{thrfrsop}: The contribution of the negative half-axis} \label{qiksecnegaxis}

For the estimation of the $ \| \cdot \|_{\widetilde{\mathcal{T}}_{\ast}}$-norm, we write $S_{k,x,\lambda}(y) = h_{k,x,\lambda}(y) - h_{k,x,\lambda}(y)\chi_{(-\infty,0]}$ where $h_{k,x,\lambda}(y) = h_{k}(\lambda(x-y))$ and estimate the norm of both terms separately. $S_{x,\lambda}$ is decomposed analogously. In this section we analyze the norm of the function that is supported on the negative axis. 

The estimation of $|\hat{S}_{k,x,\lambda}^{(n)}(t)|$ and $|\hat{S}_{x,\lambda}^{(n)}(t)|$ for $t \in [-\nu_0,\nu_0]$ and $n \leq m_0$ happens analogously by decomposing $S$ into the parts supported on the positive and negative half-axes, but shall be technically simpler than for the $ \| \cdot \|_{\widetilde{\mathcal{T}}_{\ast}}$-norm.

\textbf{Non-quasianalytic $\widetilde{\mathcal{T}}_{\ast}$:} Let $h_{x,\lambda,-}(y) = h(\lambda(x-y))  \chi_{(-\infty,0]}(y)$; we estimate $\hat{h}^{(n)}_{x,\lambda,-}$. We obtain for $n \geq 2$, all real $x, \lambda, |t| \geq 1$ that 
\begin{align*} \hat{h}^{(n)}_{x,\lambda,-}(t) & = \int^{0}_{-\infty} e^{-iyt} (-iy)^n h(\lambda(x-y)) \dif y = -\frac{1}{t^2} \int^{0}_{-\infty} e^{-iyt} \left((-iy)^n h(\lambda(x-y))\right)'' \dif y \displaybreak[1] \\ 
& \ll  |t|^{-2} \int^{-1}_{-\infty} n^2 |y|^n \max_{j \in \{0,1,2\}} \lambda^{j} |h^{(j)}(\lambda(x-y))| \dif y  \\
& \ \ \ \ \ + |t|^{-2}\int^{0}_{-1} n^2  \max_{j \in \{0,1,2\}} \lambda^{j} |h^{(j)}(\lambda(x-y))| \dif y =: I_1 + I_2. 
\end{align*}
Because of the bound $|h^{(j)}(y)| \ll_{\varepsilon,B,M_n} \exp(-\overline{M}(|y|)) \ll_{\varepsilon} \exp(-|y|^{\gamma})$, $j= 0,1,2$, the second integral becomes 
$$I_2 \ll_{\varepsilon,B,M_n} |t|^{-2}n^2 \lambda^2 \exp(-\lambda^{\gamma} x^{\gamma}) \leq |t|^{-2} n^2\lambda^2 \exp(-\lambda^{\gamma} - x^{\gamma}) \ll_\varepsilon \frac{|t|^{-2}n^2}{f(x)} \ll_{B,M_n} \frac{B^n M_n}{|t|^{2}f(x)},$$
for $\lambda, x \geq 4$, say, since $\log f(x) \ll x^{1-\varepsilon}$. The first integral is estimated as 
\begin{align*}
I_1 & \ll_{\varepsilon, B, M_n} |t|^{-2} \int^{-1}_{-\infty} n^2 |y|^n \lambda^2 \exp\left(- \overline{M}(\lambda(x + |y|))\right) \dif y\\
& \ll |t|^{-2} \int^{-1}_{-\infty} n^2 |y|^n \lambda^2 \exp\left(- \tilde{M}(\lambda x) - \tilde{M}(\lambda |y|)\right) \dif y\\
& \ll_{\varepsilon} |t|^{-2} \exp(-x^{\gamma}) \int^{-1}_{-\infty} \frac{n^2|y|^n\tilde{M}_{n+2}}{\lambda^{n} |y|^{n+2}} \dif y \ll_{\varepsilon} \frac{n^2 \tilde{M}_{n+2} }{|t|^{2} f(x)} \ll \frac{ B^n M_n}{|t|^{2} f(x)},
\end{align*}   
where we used $\lambda \geq 1$, $\overline{M}(\lambda(x + |y|) = 2\tilde{M}(\lambda(x+|y|)) \geq \tilde{M}(\lambda x) + \tilde{M}(\lambda |y|)$, $\tilde{M}(x) \geq x^{\gamma} + O_{\gamma}(1) $ and $\log f(x) \ll x^{1-\varepsilon}$ with $1-\varepsilon < \gamma$. In conclusion, we obtain $|\hat{h}^{(n)}_{x,\lambda,-}(t) |\ll_{\varepsilon,B,M_n} B^n M_n |t|^{-2}/f(x)$ for $n \geq 2$, real $|t| \geq 1$, $x \geq 4$ and $\lambda \geq 4$. 

With the same calculation, one also achieves $|\hat{h}'_{x,\lambda,-}(t)| \ll_{\varepsilon, B,M_n} |t|^{-2}/f(x)$ under the same assumptions for $x, \lambda$ and $|t|$; the only difference is that in the second integration by parts, a contribution from the boundary at $0$ appears, but it is also $\ll_{\varepsilon,B,M_n} |t|^{-2}/f(x)$. 
Similarly, by refraining to integrate by parts, one deduces $\hat{h}^{(n)}_{x,\lambda,-}(t) \ll_{\varepsilon,B,M_n} B^n M_n/f(x)$ for $|t| \leq 1$ for all $n \in \mathbb{N}$, including $n =0$ (under $x,\lambda \geq 4$). With these estimates it is then easy to see that $\| \hat{h}_{x,\lambda,-} \|_{\widetilde{\mathcal{T}}_{\ast}} \ll_{\varepsilon, B, M_n,\widetilde{\mathcal{T}}_{\ast}} 1/f(x) \leq \lambda/f(x)$ for all $\widetilde{\mathcal{T}}_{\ast}$ of non-quasianalytic nature; note that we here use the assumptions $\| (1+|t|)^{-2}/G(t) \|_{L^{p}} \ll 1$ and $\| (1+|t|)^{-2}/G_n(t) \|_{L^{p}} \ll 1$ for $\widetilde{\mathcal{T}}_{\mathrm{Dif}}$, $\widetilde{\mathcal{T}}_{\mathrm{HC}}$ and $\widetilde{\mathcal{T}}_{\mathrm{SANQ}}$ respectively.

From the above considerations it follows immediately that $|\hat{h}^{(n)}_{x,\lambda,-}(t)| \ll_{\varepsilon,B,M_n,m_0} 1/f(x)$ for all $t \in [-\nu_0,\nu_0]$ and $n \leq m_0$.

\textbf{Subanalytic $\widetilde{\mathcal{T}}_{\ast}$:} For the sequence $h_{k}$ we cannot quite perform the same argument as above as we require a larger region of analyticity for $\hat{h}_{k,x,\lambda,-}(t)$ near $t = 0$ than near $t = \lambda$, say; the reasoning above together with $h_{k}(y) \ll_{c,\rho} \exp(-c^{-1}|y|/\log(k+1))$ only delivers bounds for $\hat{h}_{k}(z)$ in horizontal strips. So, in order to efficiently capture the analyticity of $\hat{h}_{k}$ in hourglass-shaped regions, an additional more subtle argument is needed.

Let again $h_{k,x,\lambda,-}(y) = h_{k}(\lambda(x-y))  \chi_{(-\infty,0]}(y)$ and $h_{k,x,\lambda,+}(y) =  h_{k}(\lambda(x-y)) \chi_{(0,\infty)}(y)$. We do begin with an estimate in the half-plane $\Im z \geq -\lambda(2c\log(k+1))\inv$. We note first that the property \eqref{eqqikbostrip} implies, after switching the contour in $h_{k}^{(j)}(y) = (2\pi)^{-1} \int^{\infty}_{-\infty} (it)^j e^{iyt} \hat{h}_{k}(t) \dif t$, that $h_{k}^{(j)}(y) \ll_{c,\rho} \exp(-c^{-1}|y|/\log(k+1))$ for $j = 0,1,2$. 

So, in the half-plane $\Im z \geq -\lambda(2c\log(k+1))\inv$ we obtain
\begin{align*}
&\hat{h}_{k,x,\lambda,-}(z)  = \int^{0}_{-\infty} e^{-izy} h_{k}(\lambda(x-y)) \dif y = - \frac{h_{k}(\lambda x)}{iz} - \frac{\lambda}{iz} \int^{0}_{-\infty} e^{-izy} h'_{k}(\lambda(x-y)) \dif y \\
& \ \ \ \ll_{c,\rho} \frac{\exp(-c^{-1}x/\log(k+1))}{|z|} + \frac{\lambda}{|z|} \int^0_{-\infty} \exp\left(-\frac{\lambda x}{c \log(k+1)} - \frac{\lambda|y|}{2c \log(k+1)}\right) \dif y \\
& \ \ \ \ll_c \frac{ \log (k+1)}{|z|} \exp\left(-\frac {x}{c \log(k+1)}\right) \ll_{c,\varepsilon} \frac{1}{|z| f(x) },
\end{align*}
since $\log f(x) \ll_{\varepsilon} x^{1-\varepsilon}$ if, say, $k \ll  x$ which our later choice of $k$ shall verify (although the implicit constant will depend on some admissible parameters). Note that we crucially use here that $y \leq 0$  (and the lower bound for $\Im z$) to bound $|e^{-izy}|$. Naturally, we also achieve $\hat{h}_{k,x,\lambda,-}(z) \ll_{c,\rho,\varepsilon} 1/f(x)$ if one does not integrate by parts.

If $\Im z \leq -\lambda(2c\log(k+1))\inv$, we write $ \hat{h}_{k,x,\lambda,-}(z) = \hat{h}_{k,x,\lambda}(z) - \hat{h}_{k,x,\lambda,+}(z)$ and estimate $\hat{h}_{k,x,\lambda,+}$. We find
\begin{align*}
 \hat{h}_{k,x,\lambda,+}(z) & = \int^{\infty}_0 e^{-izy} h_{k}(\lambda(x-y)) \dif y = \frac{h_k(\lambda x)}{iz} - \frac{\lambda}{iz}\int^\infty_0 e^{-izy} h_{k}'(\lambda(x-y)) \dif y \\
 & \ll_{c,\rho} \frac{\exp(-c^{-1}x/\log(k+1))}{|z|} + \frac{\lambda}{|z|}\int^{\infty}_{0} \exp\left(- \frac{\lambda y + \lambda |x-y|}{2c\log(k+1)}\right) \dif y \\
 & \ll_c \frac{\lambda}{|z|}  \max_{y \geq 0} \exp\left(- \frac{\lambda y + \lambda |x-y|)}{4c\log(k+1)}\right) \int^{\infty}_{0} \exp\left(-\frac{\lambda y }{4c\log(k+1)}\right) \dif y  \\
 & \ll_c \frac{\log (k+1)}{|z|} \exp\left(-\frac{x}{4c\log(k+1)}\right) \ll_{c,\varepsilon} \frac{1}{|z| f(x)},
\end{align*}
if $k \ll x$, say. Analogously, if one does not integrate by parts, one obtains $|\hat{h}_{k,x,\lambda,+}(z)| \ll 1/f(x)$.

Summarizing, after inserting $\hat{h}_{k,x,\lambda}(z) = e^{-ixz}\lambda^{-1} \hat{h}_{k}(-z/\lambda)$ and exploiting the bound from property (\ref{qikpropseqdeczero}), we obtain, say, if $\rho \geq 2$, that
\begin{equation} \label{eqqiknegaest}
\hat{h}_{k,x,\lambda,-}(z) \ll_{c,\rho,\varepsilon} \begin{cases}
 \frac{\min\{1, |z|^{-1}\}}{f(x)} & \text{ if }  |\Im z| \leq \frac{\lambda }{2c \log(k + 1)}, \\
  \frac{\exp(x/G(0))}{\lambda 2^k} + \frac{\min\{1, |z|^{-1}\}}{f(x)} & \text{ if } |z| \leq 2\lambda \text{ and }  |\Im z| \leq 1/G(0),
\end{cases}
\end{equation}
provided $k \ll x$. Here $G(0)$ corresponds to the function $G$ in the definition of $\widetilde{\mathcal{T}}_{\mathrm{An}}$. For $\widetilde{\mathcal{T}}_{\mathrm{SAI}}$ and $\widetilde{\mathcal{T}}_{\mathrm{SA}}$, we shall modify the strip $|\Im z| \leq 1/G(0)$, see below. 

For $\widetilde{\mathcal{T}}_{\mathrm{An}}$, the estimate \eqref{eqqiknegaest} then leads to $\| \hat{h}_{k,x,\lambda,-}\|_{\widetilde{\mathcal{T}}_{\mathrm{An}}} \ll_{c,\rho, \varepsilon, G, H} 1/f(x)$ after selecting $k = \lfloor 2x/(G(0) \log 2)\rfloor + 1$ and provided that 
\begin{equation} \label{eqqikcondnegpart}
 \lambda \geq \frac{1}{G(0)}, \ \ \ \text{and} \ \ \ \frac{\lambda}{2c\log(k+1)} \geq \frac{1}{G(\lambda)},
\end{equation}
since $tH(t)$ is non-decreasing and $H(t)$ is bounded from below near $t = 0$. The restrictions \eqref{eqqikcondnegpart} serve to guarantee that if $|z| \geq 2\lambda$ with $|\Im z| \leq 1/G(|\Re z|)$, then $|\Im z| \leq \lambda(2c\log(k+1))\inv$ must hold and thus the first bound of \eqref{eqqiknegaest} is applicable. Note that with our choice for $k$, the second restriction of \eqref{eqqikcondnegpart} is fulfilled for sufficiently large $x$ due to $c\log x \leq \lambda G(\lambda)$, the assumption in $\widetilde{E}_{\mathrm{An}}$, after replacing $c$ in \eqref{eqqikcondnegpart} (and throughout the proof) with $c/3$, say. 

For $\widetilde{\mathcal{T}}_{\mathrm{SA}}$ and $\widetilde{\mathcal{T}}_{\mathrm{SAI}}$ we adjust the second estimate of \eqref{eqqiknegaest}. For $\widetilde{\mathcal{T}}_{\mathrm{SAI}}$ we use instead
\begin{equation} \label{eqqiknegaesadj}
\hat{h}_{k,x,\lambda,-}(z) \ll_{c,\rho,\varepsilon}  \frac{\exp(x/BL)}{\lambda 2^k} + \frac{\min\{1, |z|^{-1}\}}{f(x)} \ \ \text{ if } |z| \leq 2\lambda \text{ and }  |\Im z| \leq 1/BL,
\end{equation}
where $L$ refers to (SA), see section \ref{qiksecprelim}. For $\widetilde{\mathcal{T}}_{\mathrm{SA}}$ one has to replace $L$ with $\tilde{B}\inv$ everywhere in \eqref{eqqiknegaesadj}. Then, if $\lambda \geq 1/(BL)$, Cauchy's inequalities yield for real $t$ that
\begin{equation} \label{eqqiknegaesderiv}\hat{h}_{k,x,\lambda,-}^{(n)}(t) \ll_{c,\rho, \varepsilon,B, M_n}B^n L^n n!\left(\frac{\exp(x/BL)}{\lambda 2^k} + \frac{\min\{1, |t|^{-1}\}}{f(x)}\right), \ \ \ |t| \leq \lambda, \ \ \ n \geq 1.
\end{equation}
Therefore, when estimating the norm $\|\hat{h}_{k,x,\lambda,-}\|_{\widetilde{\mathcal{T}}_{\mathrm{SAI}}}$, the contribution of the part $|t| < \lambda$ is $\ll_{c,\rho,\varepsilon,\widetilde{\mathcal{T}}_{\mathrm{SAI}}} \log(1+\lambda)/f(x)$ after selecting $k = \lfloor 2x/(BL\log2)\rfloor +1$. Similarly, the contribution of the part $t < \lambda$ to the norm $\|\hat{h}_{k,x,\lambda,-}\|_{\widetilde{\mathcal{T}}_{\mathrm{SA}}}$ is $\ll_{c,\rho,\varepsilon,\widetilde{\mathcal{T}}_{\mathrm{SA}}} \lambda/f(x)$ provided that $\lambda \geq \tilde{B}/B$ after selecting $k = \lfloor2\tilde{B}/(B\log 2)\rfloor + 1$ and appealing to \eqref{eqqiknewcondsa}. Both of these estimates are acceptable. Note also that all of our choices for $k$ verify \eqref{eqqikextracondk} for sufficiently large $x$ (depending only on the sequence $M_n$, $\tilde{B}$, $B$, $c$, $\varepsilon$, $C'$ and $C''$).

For the range $|t| \geq \lambda$, we proceed in the same way as for the family $S_{x,\lambda}$. However, instead of $h^{(j)}(y) \ll_{M_n} \exp(-\overline{M}(y))$, we use $h_{k}^{(j)}(y) \ll_{c,\rho} \exp(-c^{-1}|y|/\log(k+1))$ for $j = 0,1,2$. Integrating by parts twice the formula $\hat{h}_{k, x, \lambda, -}^{(n)}(t) = \int^{0}_{-\infty} (-iy)^n e^{-ity} h_{k}(\lambda(x-y))\dif y$ and inserting the adjusted exponential bound now delivers 
$$ \hat{h}_{k,x, \lambda, -}^{(n)}(t)\ll_{c,\rho} \frac{n! c^n \log^{n+1} (k+1)}{ \lambda^{n-1}|t|^2}, \ \ \ n \geq 1,
$$
after some computations. Thus, the contribution of the part $|t| \geq \lambda$ to $\|\hat{h}_{k,x,\lambda,-}\|_{\widetilde{\mathcal{T}}_{\mathrm{SAI}}}$ is, in view of the non-decrease of $G$ and $H$, 
\begin{equation} \label{eqqiknegpartsuban} \ll_{c,\rho} \sup_{n \geq 1 }\frac{n! c^n \log^{n+1}(k+1)}{ \lambda^n B^n M_n G(\lambda)^n H(\lambda)}.
\end{equation}
We delay the estimate of this quantity to section \ref{qiksecopticon} as we will encounter similar, but more complicated expressions there. In view of \eqref{eqqiknewcondsa}, the contribution of $|t| \geq \lambda$ for the norm $\|\hat{h}_{k,x,\lambda,-}\|_{\widetilde{\mathcal{T}}_{\mathrm{SA}}}$ is also of the form \eqref{eqqiknegpartsuban}, but with an extra factor $\lambda$.

Observe finally that \eqref{eqqiknegaesadj} and \eqref{eqqiknegaesderiv} imply for $n \leq m_0$ and $|t| \leq \nu_0$ that $\hat{h}^{(n)}(t) \ll_{c,\rho,\varepsilon, \widetilde{\mathcal{T}}_{\mathrm{SAI}}} 1/f(x)$ as soon as $\lambda \geq \nu_0$ with the given choice of $k$ in $\widetilde{\mathcal{T}}_{\mathrm{SAI}}$. The modifications for $\widetilde{\mathcal{T}}_{\mathrm{An}}$ and $\widetilde{\mathcal{T}}_{\mathrm{SA}}$ are obvious.

\subsection{Conclusion of the proof of Theorem \ref{thrfrsop}: Analysis of the norm for $\widetilde{\mathcal{T}}_\ast$ } \label{qiksecopticon}
Now we conclude the proof of Theorem  \ref{thrfrsop} by estimating $\|\hat{h}_{k,x,\lambda}\|_{\widetilde{\mathcal{T}}_\ast}$ and $\|\hat{h}_{x,\lambda}\|_{\widetilde{\mathcal{T}}_\ast}$, where $h_{k,x,\lambda}(y)$ and $h_{x,\lambda}(y)$ are respectively $h_{k}(\lambda(x-y))$ and $h(\lambda(x-y))$. We recall their Fourier transforms are $\hat{h}_{k,x,\lambda}(t) =  \lambda^{-1} e^{-ixt} \hat{h}_{k}(-t/\lambda)$ and $\hat{h}_{x,\lambda}(t) =  \lambda^{-1} e^{-ixt} \hat{h}(-t/\lambda)$.

Note also that $\sup_{|t|\leq \nu_0, n \leq m_0} |\hat{h}^{(n)}_{x,\lambda}(t)| = 0$ as soon as $\lambda \geq \nu_0$, while property (\ref{qikpropseqdeczero}), see also \eqref{eqqikderdeczero} below, implies that $\sup_{|t|\leq \nu_0, n \leq m_0} |\hat{h}^{(n)}_{k,x,\lambda}(t)| \ll_{m_0,c,\rho} (1+x^{m_0}) 2^{-k} \ll_{m_0,\widetilde{\mathcal{T}}_{\ast},\varepsilon} 1/f(x)$ as soon as $\lambda \geq \nu_0$ with the choices for $k$ that were made in section \ref{qiksecnegaxis}.

\textbf{Non-quasianalytic $\widetilde{\mathcal{T}}_{\ast}$:} We only estimate in detail the norm $\|\hat{h}_{x,\lambda}\|_{\widetilde{\mathcal{T}}_{\mathrm{SAINQ}}}$. The only modifications for other hypotheses $\widetilde{\mathcal{T}}_{\ast}$ of non-quasianalytic nature, are either the insertion of an $L^p$-estimate, such as \eqref{eqrfrscogn}, or it involves the simpler case of only finitely many derivatives. At the end, we briefly explain how one handles the H\"older continuity of $\widetilde{\mathcal{T}}_{\mathrm{HC}}$ and $\widetilde{\mathcal{T}}_{\mathrm{HCI}}$.

Since $G \geq 1$, the functions $G,H$ are both non-decreasing, the sequence $M_n/n!$ is logarithmically convex and the support of $\hat{h}$ belongs to the interval $(1,2)$ with $\hat{h}^{(n)} \ll_{\varepsilon,B,M_n} B^n M_n$, we obtain, with a slight abuse of the notation of the derivative,
\begin{align*}
& \sup_{n \in \mathbb{N}}\sup_{R > 0} \int^{R}_{-R}\frac{|\hat{h}^{(n)}_{x,\lambda}(t)| \mathrm{d}t}{B^n M_n G(R)^{n} H(R)} = \lambda^{\inv} \sup_{n \in \mathbb{N}}\sup_{R > 0} \int^{R}_{-R}\frac{\abs[2]{\left(e^{-ixt}\hat{h}(-\lambda^{\inv}t)\right)^{(n)}} \mathrm{d}t}{B^n M_n G(R)^{n} H(R)} \\
 & \quad \quad \leq \sup_{n \in \mathbb{N}} \sup_{R > 0} \sum_{j = 0}^{n} {n \choose j} \frac{x^{j}\lambda^{-n+j}}{B^{n} M_{n}G(R)^{n}H(R)}  \int^{\lambda^{-1}R}_{-\lambda^{-1}R} |\hat{h}^{(n-j)}(t)| \mathrm{d}t \\
& \quad \quad \ll_{\varepsilon,B,M_n} \frac{1}{H(\lambda)} \sup_{n \in \mathbb{N}} \sum_{j = 0}^{n} {n \choose j} \frac{x^{j} M_{n-j} \lambda^{-n+j}}{B^{j} M_{n}G(\lambda)^{n}}
 \leq \frac{1}{H(\lambda)} \sup_{n \in \mathbb{N}}  \sum_{j = 0}^{n} \frac{x^j}{B^j M_j G(\lambda)^j}  \frac{n! M_j M_{n-j}}{ M_n j! (n-j)! } \lambda^{-n+j} \\
&\quad \quad \leq \frac{1}{H(\lambda)} \exp\left(M\left(\frac{x}{B G(\lambda)}\right)\right) \sum_{j = 0}^n \frac{n! M_j M_{n-j}}{M_n j! (n-j)! } \lambda^{-n+j} 
 \ll_{M_0}  \frac{1}{H(\lambda)} \exp\left(M\left(\frac{x}{B G(\lambda)}\right)\right),
\end{align*}
if $\lambda \geq 2$, say. This concludes the proof of Theorem \ref{thrfrsop} for $\widetilde{\mathcal{T}}_{\mathrm{SAINQ}}$.

We briefly mention the changes for $\widetilde{\mathcal{T}}_{\mathrm{HC}}$ and $\widetilde{\mathcal{T}}_{\mathrm{HCI}}$. One has to adjust the support of $\hat{h}$ in section \ref{qiksectestfun} to $(4/3,5/3)$ to ensure that $\hat{h}(\lambda^{-1}(t+u)) = 0$ for $t \notin [\lambda, 2\lambda]$ as $\lambda \geq 1$ and $|u| \leq \delta < \delta_0 \leq 1/3$, and one requires the additional estimate 
$$ \sup_{0 < \delta \leq \delta_0} \sup_{|u| \leq \delta} \frac{|e^{-ix(t+u)} - e^{-ixt}|}{\omega(\delta)} \ll \frac{1}{\omega(1/x)}.
$$
If $\delta \geq 1/x$, this follows by bounding the complex exponentials by $1$ and the non-decrease of $\omega$. If $\delta \leq 1/x$, one can estimate the left-hand side by $\ll x\delta/\omega(\delta) \ll 1/\omega(1/x)$ where the last step is justified as $\omega$ is non-decreasing and subadditive.

\textbf{Subanalytic $\widetilde{\mathcal{T}}_{\ast}$:} We consider $\widetilde{\mathcal{T}}_{\mathrm{SAI}}$ in detail; the adjustments for $\widetilde{\mathcal{T}}_{\mathrm{SA}}$ should pose no additional difficulties. 

In contrast to the family $h_{x,\lambda}$, we can no longer exploit the support of $h_{k,x,\lambda}$ to argue that the contribution of the part $R \leq \lambda$ to the norm $\| \hat{h}_{k,  x,\lambda}\|_{\widetilde{\mathcal{T}}_{\mathrm{SAI}}}$ is negligible. We use instead the decay of $\hat{h}_{k}$ near the origin. Via Cauchy's inequalities, property (\ref{qikpropseqdeczero}) implies that
\begin{equation} \label{eqqikderdeczero}\hat{h}_{k}^{(n)}(t) \ll_{c,\rho}  \frac{2^n n! }{2^k \rho^n}, \ \ \ |t| \leq 1, \ \ \ n \in \mathbb{N},
\end{equation}
if $\rho \geq 2$, say. Recall that $L$ is a fixed number such that $n! \ll_{M_n} L^{-n} M_n$ as the sequence $M_n$ is subanalytic. From \eqref{eqqikderdeczero} we obtain for $R \leq \lambda$ that
\begin{align*} 
\frac{1}{\lambda B^n M_n} \int^{R}_{-R} \abs[2]{\left(e^{-ixt}\hat{h}_{k}(-\lambda^{\inv} t)\right)^{(n)}} \mathrm{d}t
& \ll_{c,\rho} \sum_{j = 0}^n {n \choose j} \frac{x^j \lambda^{-n+j} 2^{n-j} (n-j)! }{B^n M_n 2^k \rho^{n-j}} \\
& \ll_{M_n} \sum_{j = 0}^n \frac{x^j  2^{n-j}}{j! 2^k \rho^{n-j} L^n B^n} \ll 2^{-k} \sum_{j = 0}^{n}  \frac{x^j }{L^j B^j j!} \\
& \ll \exp\left(\frac{x}{LB} - k \log 2\right) \ll_{\varepsilon, L,B} 1/f(x),
\end{align*}
if $\rho \geq 2/(LB)$ and $k \geq 2x/(LB\log 2)$, say, which our choice for $k$ in section \ref{qiksecnegaxis} for $\widetilde{\mathcal{T}}_{\mathrm{SAI}}$ verifies. 

It only remains to analyze the norm $\| \hat{h}_{k,  x,\lambda}\|_{\widetilde{\mathcal{T}}_{\mathrm{SAI}}}$ if the quantity $R$ in the definition of the norm is larger than $\lambda$. We now exploit (\ref{qikpropseqana}), the bound \eqref{eqqikbostrip} implies for real $t$ that
$$ \hat{h}_{k}^{(n)}(t) \ll_{c} \frac{c^n \log^n(k+1) n!}{1 + |t|^4}.
$$
Inserting this bound into the expression for the norm $\| \hat{h}_{k, x,\lambda}\|_{\widetilde{\mathcal{T}}_{\mathrm{SAI}}}$, a similar calculation as carried out above for $h_{x,\lambda}$ then yields, since $G, H$ are non-decreasing and $R \geq \lambda$,
\begin{align*}
 & \lambda^{\inv} \sup_{\substack{n \in \mathbb{N} \\ R \geq \lambda}} \int^{R}_{-R}\frac{\abs[2]{\left(e^{-ixt}\hat{h}_k(-\lambda^{\inv} t)\right)^{(n)}} \mathrm{d}t}{B^n M_n G(R)^{n} H(R)} \\
 &  \ \ \  \qquad \ll_{c} \sup_{\substack{n \in \mathbb{N} \\ R \geq \lambda}} \sum_{j = 0}^n {n \choose j} \frac{x^j c^{n-j} \log^{n-j}(k+1) (n-j)!}{B^n M_n G(\lambda)^n H(\lambda) \lambda^{n-j}} \int^{R}_{-R} \left(1 + \frac{|t|}{\lambda}\right)^{-4} \lambda^{-1} \dif t \\
 & \ \ \  \qquad \ll \frac{1}{H(\lambda)} \sup_{n \in \mathbb{N}} \sum_{j = 0}^n \frac{x^j}{B^j M_j G(\lambda)^j} \frac{2^{n-j} c^{n-j} \log^{n-j}(k+1) (n-j)!}{B^{n-j}M_{n-j} \lambda^{n-j} G(\lambda)^{n-j}} \frac{n! M_j M_{n-j}}{M_n j! (n-j)!2^{n-j}} \\
 & \ \ \  \qquad \ll_{M_0} \frac{1}{H(\lambda)} \exp\left(M\left(\frac{x}{BG(\lambda)}\right) + Q\left(\frac{2c\log(k+1)}{B\lambda  G(\lambda)}\right)\right), 
\end{align*}
with $M$ and $Q$ the associated functions of the sequences $M_n$ and $Q_n$ respectively. Here we exploited that $Q_n = M_n/n!$ is logarithmically convex and thus that $Q_{n} \geq Q_{j}Q_{n-j}/Q_{0}$. We note that \eqref{eqqiknegpartsuban} can be bounded by the term corresponding to $j = 1$ in the second line of the above estimation as $c \log^2(k+1)/(x\lambda) \ll_c \log^2(k+1)/x \ll_{B,L} 1$, and similarly if \eqref{eqqiknegpartsuban} were to be multiplied with $\lambda$. Inserting the choice for $k$ and replacing throughout the proof $c$ with $c/3$, say, then delivers the desired estimate given the increasing nature of $Q$. This concludes the proof of Theorem \ref{thrfrsop} for ${\widetilde{\mathcal{T}}_{\mathrm{SAI}}}$. 

The analysis for ${\widetilde{\mathcal{T}}_{\mathrm{SA}}}$ is similar, except that one now has to use the bounds \eqref{eqqiknewcondsa}. There arises an extra factor $\lambda$ in the estimation of the part $|t| \leq \lambda$, but this is acceptable.  

The treatment of ${\widetilde{\mathcal{T}}_{\mathrm{An}}}$ is also fairly similar. For $|z| \leq \lambda + G(0)\inv$, say, with $|\Im z| \leq 1/G(0)$, one uses (\ref{qikpropseqdeczero}), say, assuming $\rho \geq 1 + G(0)\inv$, to find $\lambda^{-1} |e^{-ixz}\hat{h}_{k}(-z/\lambda)| \ll_{c,\rho} \exp(x/G(0) - k\log 2) \ll_{G,\varepsilon} 1/f(x)$ with the appropriate choice for $k = \lfloor 2x/(G(0)\log2)\rfloor + 1$ made in section \ref{qiksecnegaxis} for the class ${\widetilde{\mathcal{T}}_{\mathrm{An}}}$. 
For $|z| \geq \lambda + G(0)\inv$ but still satisfying $|\Im z| \leq 1/G(|\Re z|)$, we have $|\Re z| \geq \lambda$ and we can use (\ref{qikpropseqana}) and the non-decrease of $G(t)$ and $tH(t)$ to find
$$\frac{e^{-ixz}\hat{h}_{k}(-z/\lambda)}{\lambda H(|z|)} \ll_{c,\rho} \frac{\exp\left(x/G(\lambda)\right)}{\lambda H(\lambda)}.
$$
The only requirement is that $-z/\lambda$ must belong to $\Lambda_{k,c}$, but this is implied by $c \log(k+ 1) \leq \lambda G(\lambda)$, see also \eqref{eqqikcondnegpart}. This concludes the proof of Theorem \ref{thrfrsop} for all ${\widetilde{\mathcal{T}}_{\ast}}$.

\subsection{Concluding remarks} \label{qikremopti}
We conclude this Section by providing the corresponding optimality statement in Theorem \ref{coran} and end with a few comments how the rates in Theorem \ref{thrfrsop} compare with those of Theorem \ref{thrikimain} and a discussion of the main innovations in the proof of Theorem \ref{thrfrsop}.

\begin{corollary} \label{coroptian} Let $M,K, W \colon\R_+\to(0,\infty)$ be non-decreasing functions and assume that $M$ and $K$ are continuous. Suppose further that there exists a constant $C$ such that 
\begin{equation} \label{eqopticondan} K(t) \leq \exp\left(\exp(CtM(t))\right), \quad t \rightarrow \infty.
\end{equation}
\label{eq:result}If $|S(x)| \ll_{S,W} W(x)\inv$, $x\to\infty$, for every Lipschitz continuous function $S\colon\R_+\to\R$ such that the Laplace transform of $S$ extends analytically to the region $\Omega_M$, defined in \eqref{eq:defomega}, and is bounded by $K(|z|)/(1+|z|)$ there, then
 \begin{equation}
 W(x)\ll_{W,M,K} M_K\inv(x),\quad x\to\infty,
 \end{equation}
 where $M_K\inv$ is the inverse function of $M_K(t) = M(t) \left(\log t + \log K(t)\right)$.
\end{corollary}
The corollary follows from Theorem \ref{thrfrsop} by taking $f(x) = 1$ and ${\widetilde{\mathcal{T}}_{\ast}} ={\widetilde{\mathcal{T}}_{\mathrm{An}}}$ with $G(t) = M(t)$ and $H(t) = K(t)/(1+t)$. One chooses $\lambda = M_{K}\inv(x)$ in \eqref{eqresultopti} and the hypothesis \eqref{eqopticondan} guarantees that one may pick $c > 0$ (depending on $M$ and $K$, but independent of $x$ and $\lambda$) such that the restriction $c\log x \leq \lambda G(\lambda)$ in the definition of $\widetilde{E}_{\mathrm{An}}$ is fulfilled with this choice for $\lambda$ as $x$ tends to $\infty$. The assumption \eqref{eqopticondan} is much more general than the corresponding hypothesis $M_{K}(x) \ll \exp(Cx)$ for some $C$ from \cite[Th. 2.3]{DebruyneSeifert2}, which was the previous best optimality result. In particular, this shows that the $M_K^{-1}$-estimate in Theorem \ref{coran} cannot be improved in case $M = K$ for all non-decreasing functions $M$.

\begin{enumerate}
\item 
The underlying reason we were able to relax the restriction $M_{K}(x) \ll \exp(Cx)$ from \cite[Th. 2.3]{DebruyneSeifert2} lies in the exploited region of analyticity of the functions $\hat{h}_{k}$. In \cite{DebruyneSeifert2} one essentially only used that $\hat{h}_{k}$ is analytic in the strip $\{z : |\Im z| \leq c^{-1}/\log(k+1)\}$, see \cite[Eq. 2.12]{DebruyneSeifert2}, while here we truly take advantage of the analytic continuation of $\hat{h}_{k}(z)$ to $\{z: |z| \leq 1\}$, say. This enabled us to weaken the restriction for $\lambda$ that ensured that $S_{k,x,\lambda}$ indeed belonged to $V_1$.  

\item For $\widetilde{\mathcal{T}}_{\mathrm{SA}}$, $\widetilde{\mathcal{T}}_{\mathrm{SAI}}$ and $\widetilde{\mathcal{T}}_{\mathrm{An}}$, the extra factor $\exp(Q(c \log x/ \lambda G(\lambda)))$ in $\widetilde{E}_{\mathrm{SAI}}$ may inhibit the establishment of the optimality of Theorem \ref{thrikimain} for $\mathcal{T}_{\mathrm{SA}}$, $\mathcal{T}_{\mathrm{SAI}}$ and $\mathcal{T}_{\mathrm{An}}$ if the optimal choices for $\lambda$ in \eqref{eqrikires} are growing too slowly to $\infty$ (as $x \rightarrow \infty$), cf. \eqref{eqopticondan}. Ideally one selects $\lambda$ in \eqref{eqresultopti} as if $\exp(Q(c \log x/ \lambda G(\lambda)))$ is not present in $\tilde{E}_{\ast}$ or that $\lambda$ is so large that the size of $Q(c \log x/ \lambda G(\lambda))$ pales in comparison with that of $M(x/G(\lambda))$.

Typically, the relevant growth requirement on $\lambda$ weakens as the sequence $M_n$ strays further from describing analyticity, and dissipates entirely if $M_n$ is in the vicinity of the realm of non-quasianalytic sequences. For example, the sequence $M_n = (n\log (n+1))^n$ is still quasianalytic, but only barely. Its associated function is $M(x) \asymp x/\log x$ while the associated function of $Q_n = M_n/n!$ is $Q(x) \ll \exp(Cx)$ for some $C$. Therefore, for a sufficiently small $c > 0$, one always has $Q(c\log x/ \lambda G(\lambda)) \leq M(x/G(\lambda)) + O_{M_n}(1)$ if $G$ is bounded from below and the dominant term inside the exponential in $\widetilde{E}_{\mathrm{SAI}}$ is always $M(x/G(\lambda))$. Usually this is sufficient to achieve the optimality of Theorem \ref{thrikimain} for $\mathcal{T}_{\mathrm{SAI}}$. 

Namely, if $\lambda G(\lambda) \gg \log x$ then $Q(c\log x/\lambda G(\lambda)) \ll_c 1$ which immediately yields the optimality as now $Q(c\log x/\lambda G(\lambda))$ indeed does not play any role in \eqref{eqresultopti}. On the other hand, if $\lambda G(\lambda) \ll \log x$, the optimal choices for $\lambda$ in \eqref{eqrikires} are then typically tending so slowly to $\infty$ that one has the same rate in \eqref{eqrikires} (with a worse implicit constant) if one replaces the factor $\exp(M(x/G(\lambda)))$ in $E_{\mathrm{SAI}}$ with $\exp(2M(x/G(\lambda)))$.

\item Apart from the restriction in case the boundary assumptions $\mathcal{T}_\ast$ are of quasianalytic nature discussed above and the presence of the factor $(1+\kappa B\inv G(\lambda)\inv \lambda^{-1})\inv$ in the argument of  associated function $M$ in $E_{\mathrm{SAI}}$ which, as we discussed in Section \ref{secest}, is usually insignificant, the two terms in the infimum in \eqref{eqresultopti} are the reciprocals of the ones in \eqref{eqrikires} for the hypotheses $\mathcal{T}_{\mathrm{DifI}}$, $\mathcal{T}_{\mathrm{HCI}}$ and $\mathcal{T}_{\mathrm{SAI}}$. Then, if, say, $E(x,\lambda)$ is a continuous non-decreasing function in $\lambda$ (and therefore its reciprocal $\tilde{E}(x,\lambda)$ is continuous and non-increasing), the infima in both \eqref{eqrikires} and \eqref{eqresultopti} are essentially reached when the two terms in the infima are equal and this happens for the same $\lambda$ in both  \eqref{eqrikires} and \eqref{eqresultopti}. Therefore, under these circumstances Theorem \ref{thrfrsop} truly yields the optimality of the rate in Theorem \ref{thrikimain} for the classes $\mathcal{T}_{\mathrm{DifI}}$, $\mathcal{T}_{\mathrm{HCI}}$ and $\mathcal{T}_{\mathrm{SAI}}$ (if the function $F$ in the statement of Theorem \ref{thrikimain} is differentiable with derivative $f$ satisfying the hypotheses of the Theorem \ref{thrfrsop}).

For the other $\mathcal{T}_{\ast}$ another discrepancy might occur between \eqref{eqrikires} and \eqref{eqresultopti} if the estimate in the $L^{p}$-norm in $\tilde{E}_{\ast}$ does not match the integral in $E_{\ast}$. This generally happens when the bounds for the derivatives of the Laplace transform are close to being integrable. For example, Theorem \ref{thrfrsop} gives that the rate $E_{\mathrm{Dif}}$ is the optimal one in Theorem \ref{thrikimain} for the class $\mathcal{T}_{\mathrm{Dif}}$ if $G$ is of the form $G(t) = \min\{1,t^{-1/q}\}H(t)$ for a non-decreasing function $H$ of positive increase, but not if $G(t) = \min\{1,t^{-1/q}\}$ (with $q$ the conjugate index of $p$). In the latter case Theorem \ref{thrikimain} delivers (with $F(x) = x$) the rate $S(x) \ll_S x^{-N} \log^{1/q} x$ as $x \rightarrow \infty$, while Theorem \ref{thrfrsop} only yields the barrier $x^{-N}$. 

For $\mathcal{T}_{\mathrm{An}}$, a similar comment applies. When the bound $H$ in $\mathcal{T}_{\mathrm{An}}$ does not satisfy $\int^{\lambda}_{-\lambda} H(t) \dif t \asymp \lambda H(\lambda)$, there is a possibility that the estimate from Theorem \ref{thrikimain} is slightly worse than the barrier given in Theorem \ref{thrfrsop}. On the other hand, if the function $G$ governing the shape of the region of analytic continuation, is growing sufficiently fast, one might still be able to reach the barrier from Theorem \ref{thrfrsop}, even if the above asymptotic relation for $H$ fails. For example, in Theorem \ref{coran}, we also reached the $M_K$-estimate \eqref{eqmkest} if only $M(t)/\log^{\beta} t$ was eventually non-decreasing, see also the fourth item of Remark \ref{remrfran}.

 Similar comments apply to $\mathcal{T}_{\mathrm{SA}}$. We just mention that $\widetilde{\mathcal{T}}_{\mathrm{SA}}$ is reasonably compatible with $\mathcal{T}_{\mathrm{SA}}$. As an illustration, let $G_{n}(t) \asymp G(t)^{n}H(t)\min\{1,t^{-1/q}\}$, for non-decreasing $G,H$ and where $H$ is of positive increase. Then all of \eqref{eqaux}, \eqref{eqrfrscogn} and \eqref{eqqiknewcondsa} are fulfilled implying \eqref{eqrikires} generally matches \eqref{eqresultopti}. As before, a slight discrepancy between the two rates might occur if $G_n$ is close to $t^{-1/q}$.

\item We comment now why we only allowed derivatives from $n \geq 1$ onwards for certain $\widetilde{\mathcal{T}}_{\ast}$. This is mainly due to $h_{x,\lambda,-}$, which was defined as the difference between $S_{x,\lambda}$ and $h_{x,\lambda}$ and for which we only have $\hat{h}_{x,\lambda,-}(t) \ll t^{-1}$. Thus, if one were to include $n = 0$ in the definition of $\widetilde{\mathcal{T}}_{\ast}$, one would need to introduce some undesirable extra hypotheses to achieve \eqref{eqresultopti}; for $\widetilde{\mathcal{T}}_{\mathrm{DifI}}$ for example, one requires $G(R) \gg \log R$ just to ensure that the norm $\| \hat{h}_{x,\lambda,-} \|_{\widetilde{\mathcal{T}}_{\mathrm{DifI}}}$ is even finite. 

On the other hand, we can only reach $\hat{h}_{x,\lambda,-}(t) \ll t^{-1}$ primarily because of the discontinuity of $S_{x,\lambda}$ at $0$. It was introduced as we simply defined $S_{x,\lambda}$ to be the multiplication of $h_{x,\lambda}$ with the indicator function of the positive half-axis. If $S_{x,\lambda}$ were to be defined to transition smoother from the $0$-function on the negative half-axis to $h_{x,\lambda}$, one ought to be able to obtain better estimates for $\hat{h}_{x,\lambda,-}(t)$; this way, one might treat the optimality for $\widetilde{\mathcal{T}}_{\ast}$ including $n = 0$, without the need to invoke extra hypotheses. With a smoother transition, one might also weaken the requirements $\| (1+|t|)^{-2}/G(t) \|_{L^{p}} \ll 1$ and $\| (1+|t|)^{-2}/G_n(t) \|_{L^{p}} \ll 1$ for $\widetilde{\mathcal{T}}_{\mathrm{Dif}}$, $\widetilde{\mathcal{T}}_{\mathrm{HC}}$ and $\widetilde{\mathcal{T}}_{\mathrm{SA}}$ respectively. 

Yet, in practice the Ingham-Karamata theorem is not applied if one exploits no regularity (differentiability, analyticity, etc.) of the Fourier transform, nor it is used when the derivatives of the Fourier transform are integrable; we therefore decided not to pursue Theorem \ref{thrfrsop} under this mild generalization for $\widetilde{\mathcal{T}}_{\ast}$. 

\item Compared to the work in \cite{DebruyneSeifert2} which is also based on the open mapping theorem, the proof of our optimality theorem features a number of technical improvements, such as the acquisition of \emph{real-valued} counterexamples, a treatment of \emph{flexible Tauberian conditions}, an exposition for more general hypotheses $\widetilde{\mathcal{T}}_{\ast}$ for the Laplace transform, the insertion of a different family of test functions for assumptions $\widetilde{\mathcal{T}}_{\ast}$ of non-quasianalytic nature and the already mentioned improvement for $\widetilde{\mathcal{T}}_{\mathrm{An}}$ in Corollary \ref{coroptian}.

We also emphasize that the analysis to obtain counterexamples $S$ having support on the positive half-axis is rather different than in \cite{DebruyneSeifert2}. There one essentially first established the existence of a counterexample $S$ with no restrictions on the support and then argued via an edge-of-the-wedge argument that the function $S_{+} = S \chi_{[0,\infty)}$ inherits the desired properties from $S$. However, it was crucial there that the hypotheses for the Laplace transform were of \emph{analytic} nature. For the weaker hypotheses $\mathcal{T}_{\ast}$ treated in this work, it is not necessarily true that for \emph{arbitrary} $\hat{S} \in \widetilde{\mathcal{T}}_{\ast}$, then also $\widehat{S \chi_{[0,\infty)}} \in \widetilde{\mathcal{T}}_{\ast}$ unless $\widetilde{\mathcal{T}}_{\ast} = \widetilde{\mathcal{T}}_{\mathrm{An}}$. Here we must carry out in section \ref{qiksecnegaxis} a delicate estimation of the norm $\|S\|_{\widetilde{\mathcal{T}}_{\ast}}$ for $S$ belonging to a specific family of test functions.
  
\item Theorem \ref{thrfrsop} provides examples $S$ that, in addition to the assumptions of the Tauberian Theorem \ref{thrikimain}, witness some other properties as well, such as the two-sided Tauberian condition, the fact that its Fourier transform is infinitely differentiable or that the derivative $S'$ is continuous. The primary reason we added these extra restrictions were to turn $V_1$ into a Fr\'echet space. However, it should go without say that other properties may be imposed on the counterexample $S$, as long the spaces $V_1$ and $V_2$ continue to be Fr\'echet (or another type of space for which the open mapping theorem holds) after incorporating the extra (semi)norms that these properties entail, the functions $S_{x,\lambda}$ and $S_{k,x,\lambda}$ continue to belong to $V_1$ (or to the completion of $V_1$ with respect to the corresponding modification of the seminorm $\|\cdot\|_{2,\nu_0,\mu_0}$ in \eqref{eqqikopenmap}) and that one can adequately estimate the extra seminorms.   

For example, if one wishes $S$ to be $L^1$, as in \cite{DebruyneSeifert2}, then one has to add the term $\| S\|_{L^1}$ to the definition of the norms \eqref{eqrfrsnorm} for the space $V_1$ (and require that all elements of $V_1$ are $L^1)$. These modified spaces $V_1$ and $V_2$ are Fr\'echet and $S_{k,x,\lambda}$ and $S_{k,x}$ belong to them, but $\| S_{x,\lambda}\|_{L^1}$ and $\|S_{k,x,\lambda}\|_{L^1}$ are only $\ll \lambda^{-1}$---$h$ and $h_k$ are uniformly bounded in $L^1$---and this bound may be larger than the right-hand side of \eqref{eqresultopti} if $f(x) \rightarrow \infty$ and if that right-hand side is larger than $\sqrt{f(x)}$ after optimizing $\lambda$. This is actually an intrinsic barrier as it is easy to see that, for a non-decreasing function $f(x)$ tending to $\infty$ as $x \rightarrow \infty$, all differentiable functions $S \in L^1(0,\infty)$ with $|S'(x)| \leq f(x)$ must satisfy $S(x) = o(\sqrt{f(x)})$ as $x \rightarrow \infty$.

Another example is the application for $C_0$-semigroups considered in \cite{DebruyneSeifert1,DebruyneSeifert2}.  There $f(x) = 1$ and one requires that the counterexamples $S$ are bounded and have a derivative $S'$ that is uniformly continuous. These modifications can easily be inserted in the space $V_1$ by adding the $L^{\infty}$-norm to the seminorms \eqref{eqrfrsnorm} topologizing the space $V_1$. The estimation of this extra $L^\infty$-norm does not produce troublesome terms as $\| S_{x,\lambda}\|_{L^\infty} \ll 1$ and $\|S_{k,x,\lambda}\|_{L^\infty} \ll 1$. 

\item Naturally the open mapping argument may be adapted to treat Tauberian conditions $\mathfrak{T}_{m}$ involving \emph{higher-order} derivatives. In case the Tauberian condition only considers finitely many derivatives, it should be possible to alter the construction of \cite[Lemma 2.1]{DebruyneSeifert2} to guarantee that $h_{k}^{(j)} \ll_{c,\rho,m} 1$ for $j \leq m$ and to modify the proof of Theorem \ref{thrfrsop} to this more general setting. In case the Tauberian condition involves infinitely many derivatives, the construction of an adequate sequence $h_{k}$ may be more delicate.

\end{enumerate}

\section{Vector-valued functions} \label{secvec}
 
Thus far, we have only considered the case when $S$ was scalar-valued. The main reason for this restriction is paradoxically to gain generality as we were able to treat one-sided Tauberian conditions. It is not immediately obvious what the corresponding generalization for vector-valued functions ought to be, although some appropriate notions are available, see \cite[Ch. 4]{a-b-h-n}. Furthermore, the proofs of most Tauberian theorems concerning functions having values in a Banach space $Y$ can actually be reduced to scalar-valued functions. In this Section we shall briefly sketch a reduction procedure that is not so commonly used in this context. We also mention that an alternative approach for generating the vector-valued Tauberian theorem is to generalize the proof of the scalar-valued one, as has been done often in the literature, see e.g. \cite{b-t, seifert}. \par
Consider the vector-valued function $\textbf{S} : \mathbb{R}_{+} \rightarrow Y$, where $Y$ denotes a Banach space. We shall analyze the scalar-valued function $S_{\textbf{e}}(x) :=\left\langle \textbf{S}(x),\textbf{e}\right\rangle$ for an arbitrary unit vector $\textbf{e} \in Y'$, the dual space of $Y$. If $\textbf{S}$ satisfies an acceptable Tauberian condition and its Laplace transform admits certain boundary behavior, one may apply the scalar-valued theorem to $S_{\textbf{e}}$ to obtain 
\begin{equation}
 \left\langle  \textbf{S}(x)  , \textbf{e}\right\rangle \leq C E(x), \ \ \ x \geq X,
\end{equation}
for a certain constant $C$ and a certain error function $E(x)$. We remark here that the constant $C$ and function $E(x)$ are in fact \emph{independent} of $\textbf{e}$ as one would estimate the quantities 
\begin{equation}
 \left|\left\langle \hat{\textbf{S}}^{(N)}(t),\textbf{e}\right\rangle\right| \text{ or } \left|\left\langle \hat{\textbf{S}}^{(N)}(t+u) -\hat{\textbf{S}}^{(N)}(t),\textbf{e}\right\rangle\right|
\end{equation}
via
\begin{equation}
 \left\| \hat{\textbf{S}}^{(N)}(t)\right\| \text{ or } \left\| \hat{\textbf{S}}^{(N)}(t+u) -\hat{\textbf{S}}^{(N)}(t)\right\|,
\end{equation}
since $\textbf{e}$ was a unit vector. 

Then, the Hahn-Banach theorem gives that 
\begin{equation}
 \left\|  \textbf{S}(x)\right\| \leq C E(x), \ \ \ x \geq X,
\end{equation}
delivering the vector-valued Tauberian theorem. The generalizations of the concepts ultradifferentiability and H\"older continuity to Banach-space-valued functions should be obvious. 

\appendix

\section{Table of decay rates} \label{rfrappendix}

For convenience of the reader, we perform the optimization in \eqref{eqrikires} for some useful, frequently arising boundary assumptions for $\mathcal{L}\{S;s\}$ under the classical Tauberian condition that $S'(x)$ is bounded (above or below). This corresponds to $f(x) = C$ in Theorem \ref{thrikimain}. We denote $g(t) = \mathcal{L}\{S;it\}$ as the (continuous) extension of $\mathcal{L}\{S;s\}$ to the imaginary axis.  In the table below, the $O$-notations in the left column are to be interpreted as $|t| \rightarrow \infty$ and the ones on the right as $x \rightarrow \infty$, but they are, contrary to the rest of the paper, allowed to depend on all other introduced parameters. 

For the hypotheses within the class $\mathcal{T}_{\mathrm{HC}}$, we always consider $p = \infty$ and the function $\omega(\delta) = \delta^{\beta'}\log^{-\gamma'}(1/\delta)$ with sufficiently large $\delta_0$, where $0 \leq \gamma' \leq 1$ and with the additional restriction that $\gamma' > 0$ if $\beta' = 0$ and $\gamma' < 0$ if $\beta' = 1$. For the class $\mathcal{T}_{\mathrm{SA}}$, we always consider $M_n = n^{\alpha'n}$ with $\alpha' \geq 1$ (and $M_0 = 1$) and $G \equiv 1$. $N$ represents a natural number but all other parameters are real numbers. If no (further) restrictions on a parameter are mentioned, it concerns an arbitrary real number. 

\begin{center}
\begin{longtable}{|p{8.6cm}|p{6.3 cm}|}
\hline
 \textbf{Boundary behavior of $\mathcal{L}\{S;s\}$} & \textbf{Error term}\\
\hline \hline
 $g^{(N)}(t) = O(\left|t\right|^{M} \log^{\gamma} \left|t\right|)$. ($N \geq 1, M > -1)$ & $O(x^{-N/(M+2)} \log^{\gamma/(M+2)} x)$\\
\hline
$g^{(N)}(t) = O(\left|t\right|^{-1} \log^{\gamma} \left|t\right|)$. ($N \geq 1,  \gamma > -1$) & $O(x^{-N}\log^{\gamma +1} x)$\\
\hline

$g^{(N)}(t) = O(e^{C\left|t\right|^{\gamma}})$. ($N \geq 1$, $ \gamma > 0$, $C > 0$) & $O\left(\log^{-\frac{1}{\gamma}} x\right)$
 \\
\hline

$g\in \mathcal{T}_{\mathrm{HC}}$:  $G(t) = \left|t+1\right|^{M}\log^{\gamma}\left|t + 2\right|$. ($M > -1$, $(N,\beta') \neq (0,0)$) & $O(x^{-(N+\beta')/(M+2)}\log^{(\gamma -\gamma')/(M+2)} x)$\\
\hline

$g\in \mathcal{T}_{\mathrm{HC}}$:  $G(t) = \left|t+1\right|^{-1}\log^{\gamma}\left|t + 2\right|$. ($\gamma > -1$, $(N,\beta') \neq (0,0)$) & $O(x^{-N-\beta'}\log^{\gamma -\gamma' +1} x)$\\
\hline

$g\in \mathcal{T}_{\mathrm{SA}}$:  $H(t) = \left|t+1\right|^{M}\log^{\gamma}\left|t + 2\right|$. ($M > -1$, 
$B > 0$) & $O\Big({\exp\Big(-\frac{\alpha'x^{1/\alpha'}}{eB^{1/\alpha'}(M+2)}\Big)} x^{\gamma/\alpha'(M+2)}\Big)$\\
\hline

$g\in \mathcal{T}_{\mathrm{SA}}$:  $H(t) = \left|t+1\right|^{-1}\log^{\gamma}\left|t + 2\right|$. ($\gamma > -1$, 
$B > 0$) & $O\Big({\exp\Big(-\frac{\alpha'x^{1/\alpha'}}{eB^{1/\alpha'}}\Big)} x^{(\gamma+1)/\alpha'}\Big)$\\
\hline
$g\in \mathcal{T}_{\mathrm{SA}}$:  $H(t) = e^{C\left|t\right|^{\gamma}}$. ($B,C,\gamma > 0$) & $O\left(x^{-1/\alpha'\gamma}\right)$\\
\hline

$g \in \mathcal{T}_{\mathrm{An}}$: $G(t) = c^{-1}$, $ H(t) = t^{M} \log^{\gamma}(t + 2)$. ($M > -1$, $c > 0$)  & $O(e^{-cx/(M+2)} x^{\gamma/(M+2)})$\\
\hline
$g \in \mathcal{T}_{\mathrm{An}}$: $G(t) = c^{-1}$, $H(t) = t^{-1} \log^{\gamma}(t + 2)$.
($\gamma > -1, c > 0$). & $O(e^{-cx}x^{\gamma + 1} )$\\
\hline

 \mbox{$g \in \mathcal{T}_{\mathrm{An}} $: $G(t) = c^{-1} \max\{t_0, \log^{\alpha} t\} $,} \mbox{$H(t) = t^{M} \log^{\gamma} (t+2) $}. ($\alpha, c >0$, $M > -1$) & $O\Big({\exp\Big(-\left(\frac{cx}{(M+2)}\right)^{1/(\alpha + 1)}\Big)} \allowbreak x^{\gamma/(M+2)(\alpha + 1)^{2}}\Big)$\\
\hline
\mbox{$g \in \mathcal{T}_{\mathrm{An}}$: $G(t) = c^{-1} \max\{t_0, \log^{\alpha} t\}$,} \mbox{$H(t) = t^{-1} \log^{\gamma} (t+2) $}. ($\alpha, c >0$, $\gamma > -1$) & $O\Big(\exp\left(-(cx)^{1/(\alpha + 1)}\right) x^{(\gamma+1)/(\alpha + 1)^{2}}\Big)$\\
\hline 
$g \in \mathcal{T}_{\mathrm{An}}$: $G(t) = c^{-1} \max\{t_0, \log^{\alpha}(t) \log^{\beta}(\log t)\}$, $H(t) = t^M \log^{\gamma}(t+2)$. ($\alpha > 0, M >  -1$ or $\alpha > 0$, $M = -1$, $\gamma > -1$) (*) & $ { \exp\Big(-\Big(\frac{(\alpha + 1)^{\beta}cx}{(M+2)  \log^{\beta} x} \Big)^{1/(\alpha + 1)}  }\allowbreak {(1 + O(\log \log x/\log x))}
   \Big)  $\\
\hline
\mbox{$g \in \mathcal{T}_{\mathrm{An}}$: $G(t) = c^{-1} \max\{t_0, t^\alpha \log^\beta t\}$,} \mbox{$H(t) = t^M \log^{\gamma} (t+2)$}. ($\alpha > 0$) & $O\left(x^{-1/\alpha} \log^{(\beta +1)/\alpha} x\right)$ \\
\hline
\mbox{$g \in \mathcal{T}_{\mathrm{An}}$: $G(t) = c^{-1} \max\{t_0, t^\alpha\}$,} \mbox{$H(t) = \exp(Ct^{\gamma})$}. ($\alpha \geq 0$, $\gamma > 0$) & $O\left(x^{-1/(\alpha+ \gamma)}\right)$ \\
\hline
\mbox{$g \in \mathcal{T}_{\mathrm{An}}$: $G(t) = c^{-1} \max\{t_0, \exp(Ct^\alpha)\}$,} \mbox{$H(t) = \exp(Ct^{\gamma})$}. ($\alpha > 0$, $\gamma > 0$) & $O\left(\log^{-1/\alpha} x\right)$ \\
\hline
\end{longtable}
\end{center}
(*) The Tauberian theorem gives in principle a better decay rate, but the inverse function that arises after the optimization cannot be approximated as well as in the other examples and we just calculated the main term inside the exponential.

\end{document}